\documentclass{amsart}
\usepackage{easyReview}
\usepackage{amssymb}
\usepackage{amsmath}
\usepackage{amsfonts}
\usepackage{diagbox}
\usepackage{mathrsfs}
\usepackage{booktabs}
\usepackage{multirow}
\usepackage{color}
\usepackage{url}
\usepackage{ulem}
\usepackage{algorithm}
\usepackage{algorithmic}
\usepackage{caption}
\usepackage{subcaption}
\usepackage{graphicx}

\newcommand{\N}{\mathbb{N}}

\def\bt{\beta}

\def\rank{\operatorname{rank}}

\newcommand{\reff}[1]{(\ref{#1})}

\newcommand{\bnum}{\begin{enumerate}}
\newcommand{\enum}{\end{enumerate}}

\newcommand{\bit}{\begin{itemize}}
\newcommand{\eit}{\end{itemize}}

\newcommand{\be}{\begin{equation}}
\newcommand{\ee}{\end{equation}}

\newcommand{\baray}{\begin{array}}
\newcommand{\earay}{\end{array}}

\newcommand{\bca}{\begin{cases}}
\newcommand{\eca}{\end{cases}}

\newcommand{\bcen}{\begin{center}}
\newcommand{\ecen}{\end{center}}

\newcommand{\bbm}{\begin{bmatrix}}
\newcommand{\ebm}{\end{bmatrix}}

\newcommand{\bpm}{\begin{pmatrix}}
\newcommand{\epm}{\end{pmatrix}}

\newcommand{\btab}{\begin{tabular}}
\newcommand{\etab}{\end{tabular}}

\newtheorem{theorem}{Theorem}[section]

\newtheorem{lemma}[theorem]{Lemma}

\newtheorem{defi}[theorem]{Definition}

\newtheorem{remark}[theorem]{Remark}

\newcommand{\bm}[1]{{\mbox{\boldmath $#1$}}}
\newcommand*{\R}{\mathbb{R}}

\def\beq#1{\begin{equation}\label{#1}}
\def\eeq{\end{equation}}
\def\bep{\begin{proof}}

\def\ep{\end{proof}}
\def\bt{\begin{theorem}}
\def\et{\end{theorem}}
\def\bl{\begin{lemma}}
\def\el{\end{lemma}}
\def\reff#1{(\ref{#1})}

\def\ignore#1{}

\def\bea#1{\begin{array}{#1}}
\def\ea{\end{array}}

\input

\makeatletter
\@namedef{subjclassname@2020}{%
	\textup{2020} Mathematics Subject Classification}
\makeatother
\usepackage{geometry}
\begin{document}
\setcounter{page}{1}

\title{Efficient Low Rank Matrix Recovery With Flexible Group Sparse Regularization\footnote{*Corresponding author: Xinzhen Zhang}}
%Inexact Restarted Augmented Lagrangian Algorithm For Fast Low Rank Matrix Recovery With Flexible Group Sparse Regularization

\author{Quan Yu, Minru Bai, Xinzhen Zhang*}
\address{School of Mathematics, Hunan University, Changsha 410082, China.}
\email{quanyu@hnu.edu.cn, minru-bai@hnu.edu.cn}

\address{School of Mathematics, Tianjin University, Tianjin 300350, China.
} \email{xzzhang@tju.edu.cn}

\begin{abstract}
	 In this paper, we present a novel approach to the low rank matrix recovery (LRMR) problem by casting it as a group sparsity problem. Specifically, we propose a flexible group sparse regularizer (FLGSR) that can group any number of matrix columns as a unit, whereas existing methods group each column as a unit. We prove the equivalence between the matrix rank and the FLGSR under some mild conditions, and show that the LRMR problem with either of them has the same global minimizers. We also establish the equivalence between the relaxed and the penalty formulations of the LRMR problem with FLGSR. We then propose an inexact restarted augmented Lagrangian method, which solves each subproblem by an extrapolated linearized alternating minimization method. We analyze the convergence of our method. Remarkably, our method linearizes each group of the variable separately and uses the information of the previous groups to solve the current group within the same iteration step. This strategy enables our algorithm to achieve fast convergence and high performance, which are further improved by the restart technique.  Finally, we conduct numerical experiments on both grayscale images and high altitude aerial images to confirm the superiority of the proposed FLGSR and algorithm.
	%The paper foucs on  low rank matrix recovery problem based on group sparsity reformulation. Firstly, we formulate the low rank matrix recovery problem as a group sparse optimizaiton problem. We further use a penalty method to reconsider the constraints in the group sparse optimization problem.  We give exact continuous relaxation problems for the reformulated group sparse optimization problem and its penalty problem. Secondly, we propose a smoothing penalty algorithm, combining with the accelarated proximal gradient algorithm, to solve the group sparse problem. We show that the algortihm converges to a $d$-stationary point.  Finally, simulation results illustrate the superiority of the group sparse formulation and the proposed algorithm.
\end{abstract}


\subjclass[2020]{}

\maketitle

\section{Introduction}
The recovery of an unknown low rank matrix $ C \in \R^{m\times n} $ from very limited information has arisen in many applications, such as optimal control \cite{FHB04,FHB01}, image classification \cite{ClTCB15,LLTX15}, multi-task learning \cite{ABEV06,AFSU07}, image inpainting \cite{Kom,MLH17,YZ22}, and so on. The problem is formulated as the following low rank matrix recovery (LRMR) problem:
\begin{equation}\label{Q:p}
\mathop{\min} \limits_{C \in \R^{m\times n}} \rank(C), \quad  \mbox{\rm s.t.} \quad \left\|\mathscr{A}\left(C \right)-\bm{b}\right\|_2\le\sigma,  		
\end{equation}
where $ C $ is the decision variable, and the linear transformation $\mathscr{A}:\R^{m\times n} \to \R^{p}$ and vector $ \bm{b} \in \R^p $ are given.

Problem \reff{Q:p} is NP-hard because of the combinatorial property of the rank function. To solve problem \reff{Q:p}, the rank function is relaxed by various spectral functions, such as the nuclear norm, the truncated nuclear norm, the Schatten-$q$ quasi-norm, etc.
Under mild conditions, the low rank matrix can be exactly recovered from most sampled entries by minimizing the nuclear norm of the matrix \cite{CR09}. Therefore, the nuclear norm minimization has been widely studied for LRMR problem \cite{CR09,CCS10,RFP10,MGC11}, which leads to a convex optimization problem. Numerical methods for nuclear norm minimization problem have strong theoretical guarantees under  some conditions, which cannot be satisfied in some practical applications \cite{CR09,CT10}. In other words, the nuclear norm is not the best approximation of the rank function. In this regard, the nonconvex relaxations are tighter than the nuclear norm relaxation to the rank function. Some popular nonconvex relaxation include truncated nuclear norm \cite{DXG18,LL16,SWKT19}, capped-$ l_1 $ function \cite{YZ22}, truncated $ l_{1-2} $ metric \cite{MLH17} and Schatten-$q$ quasi-norm \cite{MS12,NWC12}. Note that all of these methods have to compute singular value decompositions (SVD) in each iteration, which leads to high computational cost. 
To cut down the computation cost and running time, \cite{EvdH12,KMO10,WYZ12,XYWZ12} adopted a low rank matrix factorization to preserve the low rank structure of matrix such that \reff{Q:p} is relaxed as
\begin{equation}\label{eq:fac}
	\mathop{\min} \limits_{X\in\R^{m\times r},Y\in\R^{n\times r},C\in\R^{m\times n}} \frac{1}{2}\left\|XY^T-C\right\|_F^2, \quad  \mbox{\rm s.t.} \quad \left\|\mathscr{A}\left(C \right)-\bm{b}\right\|_2\le\sigma,  		
\end{equation}
where $ r $ is a preestimated matrix rank. For problem \reff{eq:fac}, the running time of numerical methods is cut down dramatically while the performance of the methods based on matrix factorization is not satisfactory \cite{FDCU19,YKWL19}.

Recently, group sparse regularizer as a surrogate for the matrix rank has attracted more and more interest as it scales well to large-scale problems. For example, Fan et al. \cite{FDCU19} proposed a new class of factor group-sparse regularizers (FGSR) as a surrogate for the matrix rank. To solve the matrix recovery problem associated with the proposed FGSR, they used the alternating direction method of multipliers (ADMM) with linearization. Jia et al. \cite{JFWZ21} proposed a generalized unitarily invariant gauge (GUIG) function for LRMR problem and solved it using the accelerated block prox-linear (ABPL) algorithm.
Although the above group sparse regularizer based matrix recovery methods have achieved satisfactory performance, they still suffer from the following limitations: 1) They treated each column of the matrix as a group and designed their algorithms to linearize the entire matrix instead of each column, which would otherwise dramatically increase the running time as the matrix size grows. However, a drawback of this approach is that they could not use the information of the previous groups to solve the current group in the same iteration step. This resulted in poor solution quality. 2) Although they established the relationship between the proposed group sparse regularizer and the matrix rank, they lacked the equivalence analysis of the related problem.

%Such two concerns are not well solved in existing group sparse regularizer based LRMR method. 
In this paper, we introduce a flexible group sparse regularizer (FLGSR) for the LRMR problem. The proposed FLGSR is partly based on the group sparse regularizer studied in \cite{FDCU19,JFWZ21}, but generalizes it to the flexible group setting. Our method, based on FLGSR, outperforms spectral functions, matrix decomposition, and previous group sparse regularizers in performance and efficiency for large-scale problems. It avoids computing SVD, estimating matrix rank, and allows flexible column grouping. In summary, our contributions can be summarized as follows:
\begin{itemize}
	\item[(1)] We prove that the matrix rank can be formulated equivalently as a FLGSR under some simple conditions. Our FLGSR model is more flexible than the FGSR model proposed in \cite{FDCU19} and the GUIG model proposed in \cite{JFWZ21} because it can group an arbitrary number of matrix columns as a unit, whereas FGSR and GUIG group each column as a unit.

	\item[(2)] We show that the LRMR problem based on matrix rank and the one based on FLGSR have the same global minimizers. Moreover, we prove the equivalence between the LRMR problem based on FLGSR with $\ell_0$-norm and its relaxed version, as well as the equivalence between the relaxed version and the corresponding penalty problem. Their links are summarized in Figure \ref{fig:link}.
	\begin{figure}[htbp]
		\centering
		\includegraphics[width=0.7\linewidth]{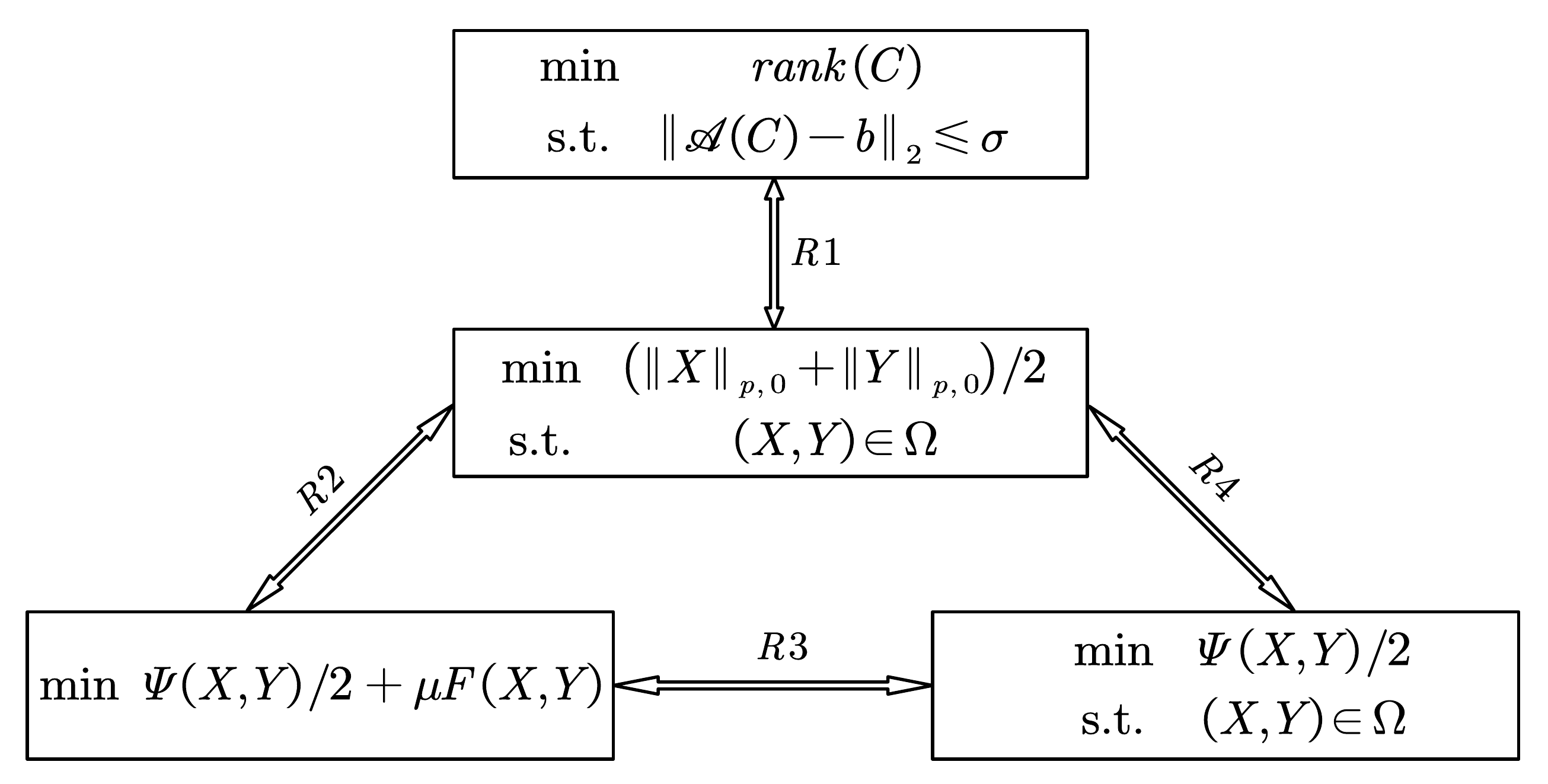}
		\caption{The relationships of global minimizers between problems \eqref{Q:p}, \eqref{Q:p0}, \eqref{Q:p_phi} and \eqref{Q:p_reg}.}
	\end{figure}
	\label{fig:link}
	\begin{center}
		R1 [Theorem \ref{th_p_p0}], R2 [Theorem \ref{th_p0_preg}], R3 [Theorem \ref{th_phi_preg}], R4 [Theorem \ref{th_p0_phi}].
	\end{center}	
	
	\item[(3)] We propose an inexact restarted augmented Lagrangian method, whose subproblem at each iteration is solved by an extrapolated linearized alternating minimization method. We also provide the convergence analysis of our method. In the update subproblem, we linearize each group, so that we can utilize the information from the groups that have been iterated before when we iterate the current group. This strategy enables our algorithm to achieve fast convergence and high performance, which are further enhanced by the restarted technique. %We also prove that the sequence generated by the proposed IRAL-ELAM algorithm converges to a Karush-Kuhn-Tucker (KKT) point of the proposed model. We propose a smoothing penalty method combining with accelerated proximal gradient algorithm to solve the group sparse optimization problem and show any accumulation point generated  by the algorithm  is a directional stationary point of \reff{Q:p_phi}. 
\end{itemize}

%The rest of the paper is organized as follows. Section 2 presents the links among problems (\ref{Q:p}),  (\reff{Q:p0}), \reff{Q:p_phi} and \reff{Q:p_reg}. In Section 3, we propose a smoothing penalty algorithm for \reff{Q:p_phi} and analyze its convergence. In Section 4, numerical examples are reported to illustrate the proposed algorithm. Section 5 draws conclusions.

%It is known that, while the matrix factorization models can generally be solved faster than the nuclear norm minimization-based approaches, it needs to know the matrix rank $ r $ in advance.

\textbf{Notation.} We introduce some notations that will be used throughout this paper. We denote by $[n]$ the set $\left\lbrace 1,2,\ldots,n\right\rbrace$, where $n$ is a positive integer. Scalars, vectors and matrices are denoted by lowercase letters (e.g., $x$), boldface lowercase letters (e.g., $\bm{x}$), and uppercase letters (e.g., $X$), respectively. The notation $\left\|\bm{x}\right\|_2:=\sqrt{\sum_i x_i^2}$ denotes the $\ell_2$-norm of a vector $\bm{x}$. For a matrix $X$, we denote by $\sigma(X):=\left(\sigma^1(X),\sigma^2(X),\ldots \right) $  the singular value vector of $X$ arranged in a nonincreasing order. The Frobenius norm, the $\ell_p$-norm, the nuclear norm and the spectral norm of a matrix $X$ are defined as $\left\|X\right\|_F:=\sqrt{\sum_i\sum_j X_{ij}^2}$, $\left\|X\right\|_p := \sqrt[p]{\sum_i\sum_j \left|X_{ij}\right|^p}$, $\left\|X\right\|_*:=\sum_i\sigma^i(X)$ and $\left\|X\right\|_2:=\sigma^1(X)$, respectively. 
Let $X$ be partitioned into $s$ disjoint groups as $ X:=[X_1,\ldots, X_s] \in \R^{m\times n} $ such that $X_{i} \in \mathbb{R}^{m\times n_{i}}$ for all $i\in [s]$. Here, $ n_1,\ldots, n_s $ are positive integers satisfying $ \sum_{i=1}^{s} n_{i}=n $.
Then we denote the $\ell_{p,0}$-norm of $ X $ as $\left\|X\right\|_{p,0}:=\sum_{i=1}^sn_i\left\|X_i\right\|_{p}^0$ (adopting the convenience that $0^0 = 0$).
We also denote the group support set of $ X $ as
\[ \Gamma(X):=\left\lbrace i \mid\left\|X_{i}\right\|_p \neq 0, \, i=1, \ldots, n\right\rbrace=\Gamma_{1}(X) \cup \Gamma_{2}(X), \]
\[ \Gamma_{1}\left(X\right):=\left\lbrace i \mid \left\|X_{i}\right\|_p<\nu, i \in \Gamma\left(X\right)\right\rbrace, \, \Gamma_{2}\left(X\right):=\left\lbrace i \mid \left\|X_{i}\right\|_p \geq \nu, i \in \Gamma\left(X\right)\right\rbrace. \]
The $ p $-distance from $X$ to a closed set $\mathbb{S} \subseteq \mathbb{R}^{n}$ is defined by $\operatorname{dist}_p(X, \mathbb{S}):=\inf \{\|X-Y\|_p: Y \in \mathbb{S}\}$.

\section{Flexible group sparse regularizer}
%In this section, we propose a flexible group sparse regularizer (FLGSR) as a surrogate to the matrix rank. For the convex relaxation of the matrix rank function, namely the nuclear norm, we also introduce the corresponding generalized matrix group nuclear norm. \textcolor{blue}{Moreover, we present the relationships between the two proposed concepts and the matrix rank and nuclear norm, respectively.}
Following \cite{FDCU19}, the rank of a matrix $C\in\R^{m\times n}$ can then be written as
$$
\operatorname{rank}\left(C\right) =\min_{C=XY^T} \operatorname{nnzc}\left(X\right)=\min_{C=XY^T} \operatorname{nnzc}\left(Y\right)=\min_{C=XY^T} \frac{1}{2}\left(\operatorname{nnzc}\left(X\right)+\operatorname{nnzc}\left(Y\right)\right),
$$
where $\operatorname{nnzc}\left(\cdot\right)$ denote the number of nonzero columns of the matrix. However, directly solving the model corresponding to the above decomposition is difficult due to its non-smoothness. Therefore, \cite{FDCU19} proposed the following factor group sparse regularizers (FGSR):
\begin{equation*}
	\begin{aligned}
		\operatorname{FGSR}_{1/2}\left(C\right) & :=\frac{1}{2} \min_{C=XY^T}\left(\left\|X\right\|_{2,1}+\left\|Y\right\|_{2,1}\right), \\
		\operatorname{FGSR}_{2 / 3}\left(C\right) & :=\frac{2}{3\alpha^{1/3}} \min _{C=XY^T}\left(\left\|X\right\|_{2,1}+\frac{\alpha}{2}\left\|Y\right\|_F^2\right),
	\end{aligned}
\end{equation*}
where $\alpha>0$, $X\in \mathbb{R}^{m\times d}$ and $ Y\in \mathbb{R}^{n\times d} $ with
$\rank(C)\le d \le \min\left\lbrace m,n \right\rbrace $. $ \left\| X \right\|_{2,1}: = \sum_{j=1}^d\left\|X(:,j)\right\|_2$ and the same to $Y$.
Furthermore, \cite{JFWZ21} gave a generalized unitarily invariant gauge function (GUIG):
\begin{equation*}
	\operatorname{GUIG}_g\left(C\right)=\min_{C=XY^T} \sum_{j=1}^d g_1\left(\left\|X(:,j)\right\|_2\right)+\sum_{j=1}^d g_2\left(\left\|Y(:,j)\right\|_2\right),
\end{equation*}
where $g_1,g_2:\R\to\R$ are two functions.

We observe that the grouping methods of FGSR and GUIG are not flexible enough, since they both consider each column of the matrix as a group. This has the following disadvantages:
\begin{itemize}
	\item When the algorithm applies alternating minimization to the whole matrix, it updates $X^k$, $Y^k$, $X^{k+1}$, $Y^{k+1}$, etc. sequentially. However, this implies that it cannot utilize the information from $X^{k+1}(:,l), l<j$ when updating $X^{k+1}(:,j)$, or the information from $Y^{k+1}(:,l), l<j$ when updating $Y^{k+1}(:,j)$. This results in poor effectiveness.
	\item When the algorithm applies alternating minimization to each group of the matrix, it updates $X^k(:,1)$, \ldots, $X^k(:,d)$, $Y^k(:,1)$, \ldots, $Y^k(:,d)$, $X^{k+1}(:,1)$, \ldots, $X^{k+1}(:,d)$, etc. in turn. The computational cost grows significantly with the number of columns of the matrix.
\end{itemize}
Therefore, in order to balance the efficiency and effectiveness, we propose a flexible grouping scheme, and extend the group $\ell_2$-norm to a more general group $\ell_p$-norm with $p>0$, named flexible group sparse regularizer (FLGSR), defined as follows:

%We first give the definition of our generalized group sparse regularizer (GGSR):
\begin{defi}\label{defi:FLGSR}
	Let $\phi_1,\phi_2:\mathbb{R} \to \mathbb{R}$ be functions. The FLGSR of a matrix $C\in\R^{m\times n}$, denoted by $G_p^{\phi_1,\phi_2}(\cdot): \mathbb{R}^{m \times n} \to \mathbb{R}$,  is defined as:
	$$ G_{p}^{\phi_1,\phi_2}\left(C\right)=\mathop{\min} \limits_{C=XY^T} \frac{1}{2}\sum_{i=1}^{s}n_i\left(\phi_1\left(\left\|X_i\right\|_p\right)+\phi_2\left(\left\|Y_i\right\|_p\right)\right),  $$
	where $\sum_{i=1}^sn_i=n$, $ X:=[X_1,\ldots, X_s] \in \R^{m\times n} $ with $X_{i} \in \mathbb{R}^{m\times n_{i}}$ and  $ Y:=[Y_1,\ldots, Y_s] \in \R^{n\times n} $ with $Y_{i} \in \mathbb{R}^{n\times n_{i}}$ for all $i\in [s]$.
\end{defi}
\begin{remark}
	Thanks to the flexible grouping in FLGSR, we can select large $n_i$ values when $n$ is large, making $s$ much smaller than $n$. This enables us to design an algorithm that applies alternating minimization to each group of the matrix, improving both computational efficiency and effectiveness. Please see Subsection \ref{sub:fg} for detailed comparison of different number of groups.
\end{remark}

For simplicity, when $\phi_1=\phi_2$, we denote $ G_p^{\phi_1,\phi_2}\left(C\right) $ by $ G_p^{\phi}\left(C\right) $, where $\phi=\phi_1=\phi_2$.
In what follows, we present two theorems that reveal the relationships between $G_p^{\left|\cdot\right|^0}\left(C\right)$ and $\rank\left(C\right)$, and between $G_2^{\phi}\left(C\right)$ and the spectral function of matrix $C$, respectively.
\begin{theorem}\label{Thm:rank}
	Let $ C=[C_1,\ldots, C_s] \in \R^{m\times n} $ be a matrix of rank $r$, where $C_i \in\R^{m\times n_i}$ with $\sum_{i=1}^sn_i=n$. If there are $ n_{i_1}, \ldots, n_{i_p} \in \left\lbrace n_1,\ldots, n_s \right\rbrace  $ such that $ \sum_{j=1}^{p} n_{i_j}=r $, then $ G_p^{\left|\cdot\right|^0}\left(C\right)=\rank\left(C\right) $.
\end{theorem}
\begin{proof}
The proof can be found in Appendix \ref{App-Thm:rank}.
\end{proof}

\begin{theorem}\label{Thm:nuclear}
	For any given matrix $ C \in \R^{m\times n} $, there exists an absolutely symmetric function $g:\mathbb{R}^{\min\left\lbrace m,n\right\rbrace}\to \mathbb{R}$, such that: $ G_2^{\phi}\left(C\right)=g\circ \sigma(C):=g\left(\sigma(C)\right) $.
\end{theorem}
\begin{proof}
	The definition of an absolutely symmetric function and the proof of the theorem are given in Appendix \ref{App-Thm:nuclear}.
\end{proof}

\begin{remark}
When $n_1=\ldots=n_s=1$, $ G_2^{\phi}\left(C\right)=\sum_{i}\phi\left(\sigma^i(C)\right) $ if $\phi$ satisfies certain conditions \cite{JFWZ21}. Unfortunately, the relationship between $\phi$ and $g$ cannot be established in other cases.
\end{remark}

\begin{remark}
	According to Theorem \ref{Thm:rank}, the matrix rank can be equivalently formulated as a FLGSR $ G_p^{\left|\cdot\right|^0}\left(C\right)$ under some easy conditions. Moreover, by applying Theorem \ref{Thm:nuclear}, there exists an absolutely symmetric spectral function $g\circ \sigma(C)$ corresponding to our proposed FLGSR function $G_2^{\phi}\left(C\right)$. This verifies that the proposed FLGSR is a good rank surrogate.
	Based on our proposed FLGSR, we do not need to calculate the SVD of the matrix, which is a costly operation, unlike the spectral function of the matrix. It also does not need to estimate the rank of the matrix beforehand, which is difficult and can affect the model performance if it is too high or too low, unlike the matrix decomposition. Moreover, our grouping is more flexible than other group sparse regularizers, allowing us to design a faster algorithm with better results.
\end{remark}

%From the above theorem, we can obtain the relationship between our proposed FLGSR and the rank of a matrix. 

\section{Efficient LRMR with FLGSR: Equivalence analysis}
In this section, we propose a novel FLGSR model based on matrix factorization for the low rank matrix recovery (LRMR) problem \reff{Q:p}, as follows:
\begin{equation}\label{Q:p0}
	G_p^{\left|\cdot\right|^0}\left(X{Y^T}\right)=\mathop{\min} \limits_{X\in\R^{m\times n},Y\in\R^{n\times n}} \frac{1}{2}\left(\left\| X \right\|_{p,0}+\left\| Y \right\|_{p,0}\right), \quad  \mbox{\rm s.t.} \quad \left\|\mathscr{A}\left( {X{Y^T}} \right)-\bm{b}\right\|_2\le\sigma.		
\end{equation}
Then,  we consider the following relaxation problem of \reff{Q:p0}:
\begin{equation}\label{Q:p_phi}
	G_p^{\phi}\left(X{Y^T}\right)=\mathop{\min} \limits_{X\in\R^{m\times n},Y\in\R^{n\times n}}\frac{1}{2}\sum_{i=1}^{s}n_i\left(\phi\left(\left\|X_i\right\|_p\right)+\phi\left(\left\|Y_i\right\|_p\right)\right) , \quad  \mbox{\rm s.t.} \quad \left\|\mathscr{A}\left( {X{Y^T}} \right)-\bm{b}\right\|_2\le\sigma,  		
\end{equation}
and its penalty problem:
\begin{equation}\label{Q:p_reg}
	\mathop{\min} \limits_{X\in\R^{m\times n},Y\in\R^{n\times n}} \frac{1}{2}\sum_{i=1}^{s}n_i\left(\phi\left(\left\|X_i\right\|_p\right)+\phi\left(\left\|Y_i\right\|_p\right)\right)
	+\mu \max\left\lbrace \left\|\mathscr{A}\left( {X{Y^T}} \right) - \bm{b}\right\|_2^2-\sigma^2,0\right\rbrace.	
\end{equation}

Here function $\phi(\cdot): \mathbb{R}_{+} \rightarrow \mathbb{R}_{+}$ is a capped folded concave function that satisfies the following two conditions with a fixed parameter $\nu>0$:
\par (1) $\phi$ is continuous, increasing and concave in $[0, \infty)$ with $\phi(0)=0$;
\par (2) there is a $\nu>0$ such that $\phi$ is differentiable in $(0, \nu), \phi_{-}^{\prime}(\nu):=\lim _{t \uparrow \nu} \phi^{\prime}(t)>0$ and $\phi(t)=1$ for $t \in[\nu, \infty)$.

Some capped folded concave functions of these two assumptions can be found in \cite{PC21}, and we omit them here.
For simplicity, we denote
\begin{equation*}
	\begin{array}{rl}\Omega&=\left\lbrace \left( X,  Y\right) \in \R^{m\times n} \times \R^{n\times n} \mid \left\|\mathscr{A}\left( {X{Y^T}} \right)-\bm{b}\right\|_2\le\sigma \right\rbrace,\\
		F(X,Y)&=\left(\left\|\mathscr{A}\left( {X{Y^T}} \right) - \bm{b}\right\|_2^2-\sigma^2\right)_+,\quad   (z)_+=\max\{0,z\}, \,\forall z\in \R, \\
		\Phi\left(X\right)&=\sum\limits_{i=1}^{s}n_i\phi\left(\left\|X_i\right\|_p\right),\, \tilde\Phi\left(Y\right)=\sum\limits_{i=1}^{s}n_i\phi\left(\left\|Y_i\right\|_p\right),\,\Psi\left(X,Y\right)=\Phi\left(X\right)+\tilde\Phi\left(Y\right).		
	\end{array}
\end{equation*}

Next, we present some relationships between \eqref{Q:p}, \eqref{Q:p0}, \eqref{Q:p_phi} and \eqref{Q:p_reg}. %Without specific explanation, Assumptions 1 is assumed throughout the paper.

%In this paper, we introduce a noval group sparse optimization  model based on matrix factorization to \reff{Q:p}. Suppose that $C=XY^T$ with $X\in \R^{m\times r}$ and $Y\in \R^{n\times r}$. Let  $ X:=[X, 0] \in \R^{m\times n}$ and $ Y:=[Y, 0] \in\R^{n\times n} $. Partition $X\,(Y)$ into $s$ disjoint groups as $ X=[X_1,\ldots, X_s]\,(Y=[Y_1,\ldots, Y_s]) $ with $X_{i} \in \mathbb{R}^{m\times n_{i}}\,(Y_{i} \in \mathbb{R}^{n\times n_{i}})$, $i=1, \ldots, s$. Here, $ n_1,\ldots, n_s $ are positive integer with $ \sum_{i=1}^{s} n_{i}=n $. Then we do not need to pre-specify the rank $ r $ and the new matrices $ X $ and $ Y $ have sparse groups structure.  Motivated by this observation and the idea of group sparse, we introduce the following low rank matrix recovery model via group sparse minimization: where $\|X\|_{p,0}=\sum_{i=1}^sn_i\left\|X_i\right\|_{p}^0,\;p\ge1$ (adopting the convenience that $0^0 = 0$) is a group cardinality function.

%In this paper, we consider the following penalty problem with $ \mu>0 $

%In this section, we present some relationships between \eqref{Q:p}, \eqref{Q:p0}, \eqref{Q:p_phi} and \eqref{Q:p_reg}. Without specific explanation, Assumptions 1 is assumed throughout the paper.

%We first present two lemmas that can be used for theoretical analysis.

\subsection{Link between problems \eqref{Q:p} and \eqref{Q:p0}}
\begin{theorem}\label{th_p_p0}
	Problem \eqref{Q:p} and \eqref{Q:p0} are equivalent. Moreover, they have the same optimal values.
\end{theorem}
\begin{proof}
	First, let ${\tilde C}\in \R^{m\times n} $ be a global minimizer of \eqref{Q:p} with ${\tilde r}:=\rank(\tilde C) $. Then there exist $ \left({\tilde X},{\tilde Y}\right) \in \R^{m\times n} \times \R^{n\times n}  $  such that $ {\tilde C}={\tilde X}{\tilde Y}^T$ and $ \|{\tilde X}\|_{p,0}=\|{\tilde Y}\|_{p,0}={\tilde r} $. Thus, 
	\begin{equation}\label{p-p0-1}
		{\tilde r}=\frac{1}{2}\left(\| {\tilde X} \|_{p,0}+\| {\tilde Y} \|_{p,0}\right)\ge  r^*,	
	\end{equation}
	where $  r^* $ is the global minimum of \eqref{Q:p0}.
	
	Next, suppose $ \left( X^*,Y^*\right) \in \R^{m\times n} \times \R^{n\times n}  $ is a global minimizer of \eqref{Q:p0}, then  
	\begin{equation}\label{p-p0-2}
	{\tilde r}\le \rank\left( X^* \left( Y^*\right)^T\right) \le \frac{1}{2}\left(\rank\left(X^*\right) +\rank\left( Y^*\right) \right) \le \frac{1}{2}\left(\left\| X^* \right\|_{p,0}+\left\| Y^* \right\|_{p,0}\right)=r^*.	
	\end{equation}
	Hence, using \eqref{p-p0-1} and \eqref{p-p0-2}, we ensure that problem \eqref{Q:p} and \eqref{Q:p0} are equivalent. Moreover, they have the same optimal values.
\end{proof}

\subsection{Link between problems \eqref{Q:p0} and \eqref{Q:p_phi}}
In this subsection, we first give the nature of the feasible solution in problem \eqref{Q:p0}.
\begin{lemma}\label{lem:xyk}
If $ \left( X^*, Y^*\right) \in \R^{m\times n} \times \R^{n\times n}  $ is a global minimizer of \eqref{Q:p0} with 
$ \left\| X^* \right\|_{p,0}+\left\|  Y^* \right\|_{p,0}=2k $, then, for any $  \left( X,  Y\right) \in \Omega  $, we have   $ \left\| X \right\|_{p,0} \ge k $ and $ \left\| Y \right\|_{p,0} \ge k $. Thus, $ \left\| X^* \right\|_{p,0} =\left\| Y^* \right\|_{p,0}= k $.
\end{lemma}
\begin{proof}
Assume on the contrary $ \min\left\lbrace\left\| X \right\|_{p,0},\left\| Y \right\|_{p,0}\right\rbrace =r < k $. Let $ C=XY^T $, then we obtain $ \rank\left(C\right)\le r $ by
$$ \rank\left(C\right)\le\rank\left(X\right)\le\left\| X \right\|_{p,0},\,\rank\left(C\right)\le\rank\left(Y\right)\le\left\| Y \right\|_{p,0}.  $$
Utilizing Theorem \ref{th_p_p0}, we get $ \rank\left(C\right)\ge k $. This contradicts $ \rank\left(C\right)\le r<k $. Hence $ \left\| X \right\|_{p,0} \ge k $ and $ \left\| Y \right\|_{p,0} \ge k $.
\end{proof}

For integers $s$ and $ t $ with $0 \leq s,\,t \leq n$, denote 
\[ Q_{X}^s:=\left\{X \in \mathbb{R}^{m\times n}:\|X\|_{p,0} \leq s\right\},\,Q_{Y}^t:=\left\{Y \in \mathbb{R}^{n\times n}:\|Y\|_{p,0} \leq t\right\} \]
and 
\[ \operatorname{dist}_p\left(\Omega_X, Q_{X}^s\right):=\inf_{X \in \Omega_X} \left\{\operatorname{dist}_p\left(X, Q_{X}^s\right) \right\},\,\operatorname{dist}_p\left(\Omega_Y, Q_{Y}^t\right):=\inf_{Y \in \Omega_Y} \left\{\operatorname{dist}_p\left(Y, Q_{Y}^t\right)\right\}, \]
where 
\[ \Omega_X=\left\lbrace X\mid  \exists\, Y\; \mbox{\rm s.t.}\; \left(X,Y\right) \in \Omega\right\rbrace\; \text{and}\;\Omega_Y=\left\lbrace Y\mid  \exists\, X\; \mbox{\rm s.t.}\; \left(X,Y\right) \in \Omega\right\rbrace. \]
Recall that the global minimum of \eqref{Q:p0} is a positive integer $k$. Then the feasible set $\Omega$ of \eqref{Q:p0} does not have $ \left(X,Y\right) $ with $ \min\left\lbrace\left\| X \right\|_{p,0},\left\| Y \right\|_{p,0}\right\rbrace < k $ by Lemma \ref{lem:xyk}, which means $ \bar{\nu}=\min\left\lbrace\bar{\nu}_X,\bar{\nu}_Y\right\rbrace>0  $ with
\begin{equation*}
	\begin{aligned}
	\bar{\nu}_X=\min\left\lbrace\frac{1}{2K} \operatorname{dist}_p\left(\Omega_X, Q_X^{k-K}\right): K=1, \ldots, k\right\rbrace,\\ \bar{\nu}_Y=\min\left\lbrace\frac{1}{2K} \operatorname{dist}_p\left(\Omega_Y, Q_Y^{k-K}\right): K=1, \ldots, k\right\rbrace.	
	\end{aligned}
\end{equation*}

In the following, we show that problems \eqref{Q:p0} and \eqref{Q:p_phi} have the same global optimal solutions for any $\nu \in(0, \bar{\nu})$.

\begin{theorem}\label{th_p0_phi}
For any capped folded concave function $\phi$ satisfying $0<\nu<\bar{\nu}$ and $\phi^{\text {CapL1 }}(t)$ $\leq \phi(t)<|t|^{0}$ for $t \in(0, \nu)$, problems \eqref{Q:p0} and \eqref{Q:p_phi} have same global minimizers and same optimal values.
\end{theorem}
\begin{proof}
$(i)$ Let $ \left( X^*,  Y^*\right) \in \R^{m\times n} \times \R^{n\times n}  $ be a global minimizer of \eqref{Q:p0} with 
$ \left\| X^* \right\|_{p,0}+\left\|  Y^* \right\|_{p,0}=2k $. We prove $\left( X^*,  Y^*\right)$ is also a global minimizer of \eqref{Q:p_phi} for any $0<\nu<\bar{\nu}$. Since the global optimality of \eqref{Q:p0} yields $ \left\| X \right\|_{p,0} \ge k $ and $ \left\| Y \right\|_{p,0} \ge k $ for any $  \left( X,  Y\right) \in \Omega  $ by Lemma \ref{lem:xyk}, we show the conclusion by two cases.

Case 1. $\|X\|_{p,0}=k$. It is easy to see that for any $i \in \Gamma(X)$,
\begin{equation*}
	\begin{aligned}
	\left\|X_{i}\right\|_p &\geq \left\|X_{j_0}\right\|_p=\min \left\lbrace \left\|X_{j}\right\|_p>0: j=1, \ldots, s\right\rbrace\\
	&=\operatorname{dist}_p\left(X, Q_X^{k-n_{j_0}}\right) \geq \operatorname{dist}\left(\Omega_X, Q_X^{k-1}\right) \geq \bar{\nu}>\nu,	
	\end{aligned}
\end{equation*}
which means that $\Phi(X)=k=\Phi\left(X^{*}\right)$.

Case 2. $\|X\|_{p,0}=r>k$. Without loss of generality, assume $\left\|X_{1}\right\|_p, \ldots,\left\|X_{i_1}\right\|_p \in(0, \nu)$, $\left\|X_{i_1+1}\right\|_p, \ldots,\left\|X_{i_2}\right\|_p \in[\nu, +\infty)$ and $\left\|X_{i_2+1}\right\|_p=\ldots=\left\|X_{s}\right\|_p=0$. If $r^{\prime}:=n_{i_1+1}+\ldots+n_{i_2}\ge k$, from $\phi(t)>0$ for $t>0$, we have $\Phi(X)>k$. Now assume $r^{\prime}<k$, we know that 
$$\frac{1}{k-r^{\prime}} \operatorname{dist}_p\left(\Omega_X, Q_X^{r^{\prime}}\right) \geq \bar{\nu}.$$ 
Together with
\begin{equation*}
	\begin{aligned}
	n_1\left\|X_{1}\right\|_p+\cdots+n_{i_1}\left\|X_{i_1}\right\|_p 
	&\geq \sqrt[\uproot{7}p]{n_1\left\|X_{1}\right\|_p^{p}+\cdots+n_{i_1}\left\|X_{i_1}\right\|_p^{p}} \\
	&\geq \operatorname{dist}_p\left(X, Q_X^{r^{\prime}}\right) \geq \operatorname{dist}_p\left(\Omega_X, Q_X^{r^{\prime}}\right),	
	\end{aligned}
\end{equation*}
we get 
\begin{equation}
	\begin{aligned}
		\Phi(X) &= n_1\phi\left(\left\|X_{1}\right\|_p\right)+\cdots+n_{i_1}\phi\left(\left\|X_{i_1}\right\|_p\right)+\cdots+n_{i_2}\phi\left(\left\|X_{i_2}\right\|_p\right) \\
		& \geq n_1\phi^{\mathrm{CapL} 1}\left(\left\|X_{1}\right\|_p\right)+\cdots+n_{i_1}\phi^{\mathrm{CapL} 1}\left(\left\|X_{i_1}\right\|_p\right)+r^{\prime} \\
		&=\frac{1}{\nu}\left(n_1\left\|X_{1}\right\|_p +\cdots+n_{i_1}\left\|X_{i_1}\right\|_p\right)+r^{\prime} \\
		& \geq \frac{1}{\nu} \operatorname{dist}_p\left(\Omega_X, Q_X^{r^{\prime}}\right)+r^{\prime}  \\ &\geq \frac{1}{\nu}\left(k-r^{\prime}\right) \bar{\nu}+r^{\prime} \\
		&>\frac{1}{\bar{\nu}}\left(k-r^{\prime}\right) \bar{\nu}+r^{\prime}=k.
	\end{aligned}
\end{equation}
The above two cases imply that $\Phi(X) \geq k=\Phi\left(X^{*}\right)$. Using similar ways in the proof for $ X $
above, we will obtain $\tilde\Phi(Y) \geq k=\tilde\Phi\left(Y^{*}\right)$. Hence $ \left( X^*,  Y^*\right) $ is also a global minimizer of \eqref{Q:p_phi}. Moreover, we have $\left\|X^{*}\right\|_{p,0}+\left\|Y^{*}\right\|_{p,0}=\Phi\left(X^{*}\right)+\tilde\Phi\left(Y^{*}\right)$ for each global minimizer $ \left( X^*,  Y^*\right) $ of \eqref{Q:p_phi}.

$(ii)$ Let $ \left( {\hat X},  {\hat Y}\right) $ be a global minimizer of \eqref{Q:p_phi} with $0<\nu<\bar{\nu}$. Assume on the contrary $ \left( {\hat X},  {\hat Y}\right) $ is not a solution of \eqref{Q:p0}. Let $ \left( X^*,  Y^*\right) $ be a global minimizer of \eqref{Q:p0}, that is, $ \left\| X^* \right\|_{p,0}=\left\|  Y^* \right\|_{p,0}=k $. By $\phi^{\mathrm{CapL} 1}(t) \leq \phi(t) \leq|t|^{0}$, we have $\Phi\left(X^*\right)\le \left\|X^*\right\|_{p,0}$ and $\tilde\Phi\left( Y^*\right)\le \left\| Y^*\right\|_{p,0}$. We may assume that $ \left\|  {\hat X} \right\|_{p,0}>k $, since $ \max\left\lbrace\left\| {\hat X} \right\|_{p,0},\left\|  {\hat Y}\right\|_{p,0} \right\rbrace>k  $ is obvious when $ \left( {\hat X}, {\hat Y}\right) $ is not a solution of \eqref{Q:p0}. Using similar ways in the proof for Case 2 above, we will obtain $\Phi\left({\hat X}\right)>k=\left\| X^*\right\|_{p,0} \geq \Phi\left( X^*\right)$ for any $0<\nu<\bar{\nu}$. Thus,
\[ \Phi\left({\hat X}\right)+\tilde\Phi\left({\hat Y}\right)>2k \ge \Phi\left( X^*\right)+\tilde\Phi\left( Y^*\right). \]
This contradicts the global optimality of $ \left( {\hat X}, {\hat Y}\right) $ for \eqref{Q:p_phi}. Hence $ \left({\hat X},  {\hat Y}\right) $ is a global minimizer of \eqref{Q:p0}.

Therefore, whenever $0<\nu<\bar{\nu}$, \eqref{Q:p0} and \eqref{Q:p_phi} have the same global minimizers and optimal values.
\end{proof}

\subsection{Link between problems \eqref{Q:p_phi} and \eqref{Q:p_reg}}

\begin{theorem}\label{th_phi_preg}
Suppose that $\phi$ is globally Lipschitz continuous on $[0, \nu]$. Then it holds that
\begin{itemize}
	\item[(1)] there exists a $\mu^{*}>0$ such that any global minimizer of \eqref{Q:p_phi} is a global minimizer of \eqref{Q:p_reg} whenever $\mu \geq \mu^{*}$;
	\item[(2)] if $\left({\bar X},{\bar Y}\right)$ is a global minimizer of \eqref{Q:p_reg} for some $\mu>\mu^{*}$, then $\left({\bar X},{\bar Y}\right)$ is a global minimizer of \eqref{Q:p_phi}.
\end{itemize}
\end{theorem}

\begin{proof}
Since $\phi$ is globally Lipschitz continuous on $[0, \nu]$, $\Psi\left(X,Y\right)$ is globally Lipschitz continuous on $[0, \nu]$. Similar to \cite[Lemma 3.1]{CLP16}, we can easily obtain (1) and (2).	
\end{proof}

\subsection{Link between problems \eqref{Q:p0} and \eqref{Q:p_reg}}

\subsubsection{stationary point of \eqref{Q:p_reg}.}
Let $f: \mathbb{R}^{m\times n} \rightarrow \mathbb{R}$ be locally Lipschitz continuous and directionally differentiable at point $X \in \mathbb{R}^{m\times n}$. The directional derivative of $f$ along a matrix $W \in \mathbb{R}^{m\times n}$ at $X$ is defined by
$$
f^{\prime}(X ; W):=\lim _{\tau \downarrow 0} \frac{f(X+\tau W)-f(X)}{\tau}.
$$
If $f$ is differentiable at $X$, then $f^{\prime}(X ; W)=\langle\nabla f(X), W\rangle$. Denote $\bar Z:=({\bar X},{\bar Y})$, by simple computation, we have
$$
F^{\prime}\left({\bar Z} ;Z-{\bar Z}\right)= \begin{cases}0, & \text { if }\left\|\mathscr{A}\left({\bar X} {\bar Y}^T\right) -\bm{b}\right\|_F^{2}<\sigma^{2}, \\
	\max \left\{0,\Delta_1+\Delta_2\right\}, & \text { if }\left\|\mathscr{A}\left({\bar X} {\bar Y}^T\right) -\bm{b}\right\|_F^{2}=\sigma^{2}, \\
	\Delta_1+\Delta_2, & \text { otherwise},\end{cases}
$$
where
\begin{equation*}
	\begin{aligned}
		\Delta_1&=2\left\langle\mathscr{A}\left({\bar X} {\bar Y}^T\right) -\bm{b}, \mathscr{A}\left(\left(X-{\bar X}\right){\bar Y}^T\right)\right\rangle,
		\\ \Delta_2&=2\left\langle\mathscr{A}\left({\bar X} {\bar Y}^T\right) -\bm{b}, \mathscr{A}\left({\bar X}\left(Y-{\bar Y}\right)^T\right)\right\rangle.		
	\end{aligned}
\end{equation*}
\iffalse
Next, we consider the directional derivative of $L_{2}$ norm and $F(X,Y)$. Denote $l(X):=\|X\|_2$. By simple computation, 
$$
l^{\prime}\left(X^{*} ; X-X^{*}\right)= \begin{cases}\|X\|_2, & \left\|X^{*}\right\|_2=0, \\ \frac{\left\langle X^{*}, X-X^{*}\right\rangle}{\left\|X^{*}\right\|_2}, & \left\|X^{*}\right\|_2\neq0.\end{cases}
$$
Then by [41, Exercise 8.31], the directional derivative of $F$ at $\bar Z:=(\bar X,\bar Y)$ has the form
\fi
\begin{defi}\label{def:sp}
We say that $ {\bar Z}:=\left( {\bar X},  {\bar Y}\right) \in \R^{m\times n} \times \R^{n\times n} $ is a stationary point of \eqref{Q:p_reg} if
$$
\Phi^{\prime}\left({\bar X} ; X-{\bar X}\right)+\tilde\Phi^{\prime}\left({\bar Y}; Y-{\bar Y}\right)+\mu F^{\prime}\left({\bar Z}; Z-{\bar Z}\right) \geq 0, \quad \forall \left( X,  Y\right) \in \R^{m\times n} \times \R^{n\times n}.
$$
\end{defi}

\subsubsection{Characterizations of lifted stationary points of \eqref{Q:p_reg}}
We obtain from \cite[Theorem 3.1]{LL94} that there exists a $\beta>0$ such that for all $Z=\left( X,  Y\right) \in \R^{m\times n} \times \R^{n\times n}$, we have
\begin{equation}\label{eq:beta}
\operatorname{dist}_p(Z, \Omega) \leq \beta\left(\left\|\mathscr{A}\left( {X{Y^T}} \right) - \bm{b}\right\|_F^2-\sigma^2\right)_+=\beta F(X,Y).	
\end{equation} 
Let $L\left(X,Y\right)=2\|\mathscr{A}^*\|_q\|\mathscr{A}\left( {X{Y^T}} \right) - \bm{b}\|_q\|Y\|_q $ with 
$$
\frac{1}{p}+\frac{1}{q}=1,\quad \|\mathscr{A}^*\|_q:=\operatorname{sup}\limits _{\left\|x\right\|_q=1}\left\|\mathscr{A}^*\left(x\right)\right\|_q.
$$

Since $\phi_{-}^{\prime}(\nu) \rightarrow \infty$ as $\nu \rightarrow 0$, for any $\Upsilon>\sigma$ and $\mu>0$, there are $\left({\hat X}, {\hat Y}\right)$ and a sufficiently small $\nu>0$ such that 
$\left\|\mathscr{A}\left( {\hat X}{\hat Y}^T \right) - \bm{b}\right\|_2 \geq \Upsilon$ and $\phi_{-}^{\prime}(\nu)>\mu L\left({\hat X}, {\hat Y}\right)$. In the rest of this paper, we choose $\Upsilon, \nu, \mu$ and $\left({\hat X}, {\hat Y}\right)$ satisfying
$$
\left\|\mathscr{A}\left( {{\hat X}{{\hat Y}^T}} \right) - \bm{b}\right\|_2 \geq \Upsilon, \quad \mu>\frac{\beta}{\bar{\nu}}, \quad \phi_{-}^{\prime}(\nu)>\lambda L\left({\hat X}, {\hat Y}\right).
$$
We then show a lower bound property of the lifted stationary points of \eqref{Q:p_reg}.
\begin{lemma}\label{lem:bound}
Let $\left( {\bar X},  {\bar Y}\right) \in \R^{m\times n} \times \R^{n\times n}$ be a stationary point of \eqref{Q:p_reg} satisfying $\left\|\mathscr{A}\left( {{\bar X}{{\bar Y}^T}} \right) - \bm{b}\right\|_2 \leq \Upsilon$ and $\phi_{-}^{\prime}(\nu)>\lambda L\left({\hat X}, {\hat Y}\right)$. Then for $ i=1,\ldots,s $, we have
$$
\left\{ \begin{array}{l}
	{\left\| {{{\bar X}_i}} \right\|_p} \ge \nu\\
	{\left\| {{{\bar Y}_i}} \right\|_p} \ge \nu
\end{array} \right. \quad \text{or} \quad
\left\{ \begin{array}{l}
	{\left\| {{{\bar X}_i}} \right\|_p} =0\\
	{\left\| {{{\bar Y}_i}} \right\|_p} =0
\end{array}. \right.
$$
\end{lemma}
\begin{proof}
To prove this Lemma, we only need to show $\Gamma_{1}\left(\bar X\right)\cup \Gamma_{1}\left(\bar Y\right) =\emptyset$. Assume on contradiction that $\Gamma_{1}\left(\bar X\right)\cup \Gamma_{1}\left(\bar Y\right) \neq\emptyset$. Might as well set $\Gamma_{1}\left(\bar X\right)\neq\emptyset$. From Definition \ref{def:sp}, we have the following inequality for any $\left( X,  Y\right)$ satisfying $ Y=\bar Y $, $X_{j}={\bar X}_{j}$ for all $j \notin \Gamma_{1}\left(\bar X\right)$ and for which $\exists\, i \in \Gamma_{1}\left(\bar X\right)$ such that $X_{i}= {\bar X}_{i}- \frac{\varepsilon}{n_i}{\bar X}_{i}$ with $\varepsilon>0$,
\begin{equation*}
\begin{aligned}
	&\varepsilon\phi^{\prime}\left(\left\|{\bar X}_{i}\right\|_p\right)\left\|{\bar X}_{i}\right\|_p\le\mu F^{\prime}\left({\bar Z} ; Z-{\bar Z}\right)\le\mu \left|\Delta_{1}+\Delta_{2}\right|\\
	\le&2\mu\varepsilon\left\|\mathscr{A}^*\left(\mathscr{A}\left({\bar X} {\bar Y}^T\right) -\bm{b}\right){\bar Y}_i\right\|_q\left\|{\bar X}_i\right\|_p.
\end{aligned}
\end{equation*}
Thus,
\begin{equation*}
\begin{aligned}
\phi_{-}^{\prime}(\nu)\le\phi^{\prime}\left(\left\|{\bar X}_{i}\right\|_p\right)\le2\mu \left\|\mathscr{A}^*\left(\mathscr{A}\left({\bar X} {\bar Y}^T\right) -\bm{b}\right){\bar Y}_i\right\|_q \le\mu L\left({\hat X}, {\hat Y}\right),	
\end{aligned}
\end{equation*}
where the first inequality follows from $i \in \Gamma_{1}\left(\bar X\right)$.%, the second inequality follows from $ \Delta_1\le 2\varepsilon\left\|\mathscr{A}^*\left(\mathscr{A}\left({\bar X} {\bar Y}^T\right) -\bm{b}\right){\bar Y}_i\right\|_q\left\|{\bar X}_i\right\|_p $ and $ \Delta_2=0 $.

This contradicts the condition of $\phi_{-}^{\prime}(\nu)>\mu L\left({\hat X}, {\hat Y}\right)$. The proof is completed.
\end{proof}

\begin{theorem}\label{th_p0_preg}
Let $\mu>\frac{\beta}{\bar{\nu}}$ and $\phi_{-}^{\prime}(\nu)>\mu L\left({\hat X}, {\hat Y}\right)$. Then the following statements hold:
\begin{itemize}
	\item[(1)] If $\left( {\bar X},  {\bar Y}\right)$ is a global minimizer of \eqref{Q:p_reg} with $\left\|\mathscr{A}\left( {{\bar X}{{\bar Y}^T}} \right) - \bm{b}\right\|_2 \leq \Upsilon$, then $\left( {\bar X},  {\bar Y}\right)$ is a global minimizer of \eqref{Q:p0};
	\item[(2)] If $\left(X^*, Y^*\right)$ is a global minimizer of \eqref{Q:p0} and \eqref{Q:p_reg} has a global minimizer $\left( {\bar X},  {\bar Y}\right)$ with $\left\|\mathscr{A}\left( {{\bar X}{{\bar Y}^T}} \right) - \bm{b}\right\|_2 \leq \Upsilon$, then $\left(X^*, Y^*\right)$ is a global minimizer of \eqref{Q:p_reg}.
\end{itemize}
\end{theorem}
\begin{proof}
(1) Since $\left( {\bar X},  {\bar Y}\right)$ is a global minimizer of \eqref{Q:p_reg} and the objective function is locally Lipschitz continuous, $\left( {\bar X},  {\bar Y}\right)$ is a stationary point of $\eqref{Q:p_reg}$. From $\left\|\mathscr{A}\left( {{\bar X}{{\bar Y}^T}} \right) - \bm{b}\right\|_2 \leq\Upsilon$ and Lemma \ref{lem:bound}, $\Phi\left(\bar X\right)=\left\|\bar X\right\|_{p,0}$ and $\tilde\Phi\left(\bar Y\right)=\left\|\bar Y\right\|_{p,0}$. Assume now that $\left( \bar X,  \bar Y\right)$ is not a global minimizer of \eqref{Q:p0} and $\left( X^*,  Y^*\right)$ with $\left\|X^*\right\|_{p,0}+\left\|Y^*\right\|_{p,0}=2k$ is a global minimizer of \eqref{Q:p0}.		

We split the proof into two cases.
\begin{itemize}
	\item If $  \left( \bar X,  \bar Y\right) \in \Omega  $, then 
	$$\left\|X^*\right\|_{p,0}+\left\|Y^*\right\|_{p,0} <\left\|\bar X\right\|_{p,0}+\left\|\bar Y\right\|_{p,0}.$$ 
	Thanks to  $ F\left( X^*,  Y^*\right)=0  $, we have
	\begin{equation*}
	\begin{aligned}
	&\Phi\left(X^*\right)+\tilde\Phi\left(Y^*\right)+\mu F\left(X^*,Y^*\right) \leq\left\|X^*\right\|_{p,0}+\left\|Y^*\right\|_{p,0}+\mu F\left(X^*,Y^*\right)\\
	=&\left\|X^*\right\|_{p,0}+\left\|Y^*\right\|_{p,0}<\left\|\bar X\right\|_{p,0}+\left\|\bar Y\right\|_{p,0} = \Phi\left(\bar X\right)+\tilde\Phi\left(\bar Y\right)+\mu F\left(\bar X,\bar Y\right),		
	\end{aligned}
	\end{equation*}
     which contradicts the global optimality of $ \left( \bar X,  \bar Y\right) $ for \eqref{Q:p_reg}.
	
	\item If $  \left( \bar X,  \bar Y\right) \notin \Omega  $, then $ F\left(\bar X,\bar Y\right)>0 $. Then, we distinguish two cases.
	\begin{itemize}
		\item if $\left\|\bar X\right\|_{p,0}+\left\|\bar Y\right\|_{p,0}\ge2k$, we obtain that
			\begin{equation*}
			\begin{aligned}
				&\Phi\left(X^*\right)+\tilde\Phi\left(Y^*\right)+\mu F\left(X^*,Y^*\right)\\ \leq&\left\|X^*\right\|_{p,0}+\left\|Y^*\right\|_{p,0}+\mu F\left(X^*,Y^*\right)\\
				=&2k< \Phi\left(\bar X\right)+\tilde\Phi\left(\bar Y\right)+\mu F\left(\bar X,\bar Y\right),		
			\end{aligned}
		\end{equation*}
		which contradicts the global optimality of $ \left( \bar X,  \bar Y\right) $ for \eqref{Q:p_reg}.
		\item if $\left\|\bar X\right\|_{p,0}+\left\|\bar Y\right\|_{p,0}<2k$, then $ \min\left\lbrace\left\|\bar X\right\|_{p,0}, \left\|\bar Y\right\|_{p,0}\right\rbrace <k$, might as well set $$ \left\|\bar X\right\|_{p,0}=\min\left\lbrace\left\|\bar X\right\|_{p,0}, \left\|\bar Y\right\|_{p,0}\right\rbrace=k'<k. $$ Thus
		\begin{equation*}
		\begin{aligned}
		2\bar{\nu} \leq& \frac{1}{k-k^{\prime}} \operatorname{dist}_p\left(\Omega_{X}, Q_{\bar X}^{k^{\prime}}\right)
        \leq \frac{1}{k-k^{\prime}} \operatorname{dist}_p\left(\Omega_{ X}, \bar X\right)\\
        \leq &\frac{1}{k-k^{\prime}} \operatorname{dist}_p\left(\Omega, \left(\bar X,\bar Y\right)\right)
         \leq \frac{\beta F\left(\bar X,\bar Y\right)}{k-k^{\prime}},	
		\end{aligned}
		\end{equation*}
		where the first inequality follows from the definition of $ \bar{\nu} $ and the last inequality uses \eqref{eq:beta}.
		This together with $ \mu>\beta/\bar{\nu} $ implies that
		\begin{equation*}
		\begin{aligned}
		&\Phi\left(\bar X\right)+\tilde\Phi\left(\bar Y\right)+\mu F\left(\bar X,\bar Y\right)> 2k+\frac{\beta}{\bar{\nu}}F\left(\bar X,\bar Y\right)	\\
		\ge&2k=\left\|X^*\right\|_{p,0}+\left\|Y^*\right\|_{p,0}+\mu F\left(X^*,Y^*\right)\\
		\ge&\Phi\left(X^*\right)+\tilde\Phi\left(Y^*\right)+\mu F\left(X^*,Y^*\right).
		\end{aligned}
		\end{equation*}
		%where the last equality follows from $ \left\|\mathscr{A}\left( {X'{Y'^T}} \right) - b\right\|_2^2 \leq \sigma^{2}<\Upsilon^{2}$ and Lemma \ref{lem:bound}. 
		This contradicts the global optimality of $ \left( \bar X,  \bar Y\right) $ for \eqref{Q:p_reg}.
	\end{itemize}	
\end{itemize}
This shows that $ \left( \bar X,  \bar Y\right) $ is a global minimizer of \eqref{Q:p0}.

(2) Suppose that $ \left(X^*,Y^*\right) $ is a global minimizer of \eqref{Q:p0} but not a global minimizer of \eqref{Q:p_reg}. Since $\left( \bar X,  \bar Y\right)$ is a global minimizer of \eqref{Q:p_reg} with $\left\|\mathscr{A}\left( {\bar X{\bar Y^T}} \right) - \bm{b}\right\|_2 \leq \Upsilon$, from Lemma \ref{lem:bound} and (1), we have 
$$\Phi\left(\bar X\right)=\left\|\bar X\right\|_{p,0},\,\tilde\Phi\left(\bar Y\right)=\left\|\bar Y\right\|_{p,0},\,\left( \bar X,  \bar Y\right) \in \Omega.$$ 
Using this, we conclude that
\begin{equation*}
\begin{aligned}
\left\|\bar X\right\|_{p,0}+\left\|\bar Y\right\|_{p,0}	&\le	\left\|\bar X\right\|_{p,0}+\left\|\bar Y\right\|_{p,0}+\mu F\left( \bar X,  \bar Y\right)\\
&=\Phi\left(\bar X\right)+\tilde\Phi\left(\bar Y\right)+\mu F\left( \bar X,  \bar Y\right)\\
&<\Phi\left(X^*\right)+\tilde\Phi\left(Y^*\right)+\mu F\left( X^*, Y^*\right)\\
&\le \left\|X^*\right\|_{p,0}+\left\|Y^*\right\|_{p,0}, 
\end{aligned}
\end{equation*}
which leads to a contradiction with the global optimality of $ \left(X^*, Y^*\right) $ for \eqref{Q:p0}. Hence $ \left( X^*, Y^*\right) $ is a global minimizer of problem \eqref{Q:p_reg} and the proof is completed.

\end{proof}

\section{An Inexact Restarted Augmented Lagrangian Method with the Extrapolated Linearized Alternating Minimization}
By introducing the auxiliary variable $ C=XY^T $, problem \eqref{Q:p_phi} can be reformulated into the following problem:
\begin{equation}\label{GSMR} \min\limits_{X,\,Y,\,C } ~\Phi\left(X\right)+\tilde\Phi\left(Y\right)+l_{\Theta}\left(C\right),\quad
	\mbox{\rm s.t.}\quad  C=XY^T,
\end{equation}
where $l_{\Theta}\left(C\right)$ is an indicator function with $ \Theta = \left\lbrace C|\left\|\mathscr{A}\left( C \right) - \bm{b}\right\|_2\le\sigma \right\rbrace  $.

Problem \eqref{GSMR} is to minimize a nonsmooth nonconvex function with bilinear constraints. By exploring the structure of problem \eqref{GSMR}, we propose an  inexact restarted augmented Lagrangian (IRAL) framework in Subsection \ref{sub:IAL}. Next, we propose an extrapolated linearized alternating minimization (ELAM) algorithm to solve the augmented Lagrangian subproblem in Subsection \ref{sub:ELAL}. By putting together these two parts, we come up with the name IRAL-ELAM for our new algorithm. In Subsections \ref{sub:con_IAL} and \ref{sub:con_ELAL}, we prove that every sequence generated by IRAL-ELAM has at least one accumulation point and that each accumulation point satisfies the Karush-Kuhn-Tucker (KKT) conditions of problem \eqref{GSMR}.

\subsection{The proposed IRAL} \label{sub:IAL}
In this subsection, we propose an IRAL framework to solve problem \eqref{GSMR}. It is easy to deduce that the augmented Lagrangian function for problem \eqref{GSMR} is
\begin{equation}\label{GSMR:Lag}
	L_{\eta}\left(X,\,Y,\,C;\,S\right)=\Phi\left(X\right)+\tilde\Phi\left(Y\right)+l_{\Theta}\left(C\right)+\left\langle XY^T-C,S \right\rangle + \frac{\eta}{2}\left\|XY^T-C\right\|_F^2,
\end{equation}
where $ \eta>0 $ is the penalty parameter and $ S $ is the Lagrange multiplier matrix.

Based on the classical augmented Lagrangian method, we use the following subproblem to approximate problem \eqref{GSMR} at each outer iteration:%, where the dual variables are kept constant and the primal variables are updated:
\begin{eqnarray}\label{IAL}
	\min\limits_{X,Y,C } ~L_{\eta^k}\left(X,Y,C;S^k\right).
\end{eqnarray}
%where iterate tuple $\left(X^k,Y^k,C^k,S^k\right)$ denotes $\left(X,Y,C,S\right)$ in the $k$th iteration. 
At the $(k+1)$th iteration, we inexactly solve \eqref{IAL} to obtain an approximate solution $\left(X^{k+1},Y^{k+1},C^{k+1}\right)$ satisfying the following condition:
\begin{equation}\label{eq:I}
	\operatorname{dist}\left(0, \partial L_{\eta^k}\left(X^{k+1},Y^{k+1},C^{k+1};S^k\right) \right) \leq \epsilon_k.
\end{equation}

%The IAL framework for solving problem \eqref{GSMR} is presented as follows.
Now, we present the IRAL framework for solving problem \eqref{GSMR} as follows.
%According to the above derivation, the IRAL framework for solving problem \eqref{GSMR} is presented as follows.
%we develop an IAL framework to solve the problem \eqref{GSMR}, which is described in Algorithm \ref{alg:IAL}.
\begin{algorithm}
	\caption{IRAL algorithm for problem \eqref{GSMR}}
	\label{alg1}
	\begin{algorithmic}[]
		\REQUIRE Initial point $ X^0,Y^0,C^0,S^0$. Parameters $\eta^0>0$, $\rho_1,\rho_2,\rho_3,\in(0,1)$, $\vartheta \in \mathbb{N}_{+}$ and $\epsilon_0>0$.
		Let $k:=0$.
		\WHILE{a stopping criterion is not met}
		\STATE \bm{Step~ 1.} Solve problem \eqref{IAL} to obtain $ \left(X^{k+1},Y^{k+1},C^{k+1}\right)$ satisfying \eqref{eq:I}.
		\STATE \bm{Step~ 2.} \textbf{If} $k\le \vartheta$, set $S^{k+1}=0$, $\eta^{k+1} = \eta^k$ and $\epsilon_{k+1} = \epsilon_k$. \\
		\qquad\quad\;\;\;\textbf{Else if} $k> \vartheta$ and
		\begin{equation}\label{eq:min}
			\left\|X^{k+1}(Y^{k+1})^T-C^{k+1}\right\|_F \le \rho_1 \min_{t=k-\vartheta+1,\ldots,k}\left\|X^t(Y^t)^T-C^t\right\|_F,
		\end{equation}
		then set $S^{k+1}=0$, $\eta^{k+1} = \eta^k$ and $\epsilon_{k+1} = \sqrt{\rho_1}\epsilon_k$. \\
		\qquad\quad\;\;\; Otherwise, compute multipliers $ S^{k+1}$ by
		\begin{equation}\label{opt:multipliers}
			S^{k+1} =S^k+\eta^k\left(X^{k+1}(Y^{k+1})^T-C^{k+1}\right),
		\end{equation}
		and set
		\begin{equation}\label{eq:eta}
			\eta^{k+1} = \eta^k/\rho_2,\;\; \text{and}\;\; \epsilon_{k+1} = \rho_3\epsilon_k. 
		\end{equation}
		\STATE \bm{Step ~3.} Let $ k:=k+1 $ and go to $ \bm{Step~1}$.
		\ENDWHILE
		\ENSURE $ X^{k+1} $, $ Y^{k+1} $, $ C^{k+1} $.
	\end{algorithmic}
\end{algorithm}

\begin{remark}
	Different from the existing inexact augmented Lagrangian methods for nonsmooth nonconvex optimization problems \cite{CGLY17,LLC23,LLM19,LZ12}, we design a new rule for updating the Lagrangian penalty parameter and the Lagrange multiplier matrix.
\end{remark}

\subsection{The proposed ELAM} \label{sub:ELAL}
We now shall discuss how to solve subproblem \eqref{IAL}, which is a nonconvex problem. 
For convenience of notation, we use $\left(X^{(k+1, \jmath)}, Y^{(k+1, \jmath)}, C^{(k+1, \jmath)}\right)$ to denote the $\jmath$-th iterate of the ELAM algorithm and the $k+1$th iterate of Algorithm \ref{alg1}. For the superscript $k$ (and $k+1$), we further denote $\eta^k, S^k, X^{(k+1, \jmath)}, Y^{(k+1, \jmath)}, C^{(k+1, \jmath)}$ by $\eta, S, X^{(\jmath)}, Y^{(\jmath)}, C^{(\jmath)}$. We assume to be at the $\jmath$th iterate of the ELAM algorithm. %\textcolor{blue}{Below, we give the details of updating each variable of the update minimizing subproblem.}

%The notations $\left(X^{(k+1, \jmath)}, Y^{(k+1, \jmath)}, C^{(k+1, \jmath)}\right)$ stands for the $\jmath$-th iterate of the alternating minimization algorithm and the $k+1$th iterate of Algorithm \ref{alg1}. For brevity, we drop the superscript $k$ (and $k+1$) and abuse the notations $\eta, S, X^{(\jmath)}, Y^{(\jmath)}, C^{(\jmath)}$ to denote $\eta^k, S^k, X^{(k+1, \jmath)}, Y^{(k+1, \jmath)}, C^{(k+1, \jmath)}$, respectively. We assume to be at the $\jmath$th iterate of the alternating minimization algorithm.

1) Computing $ X_i^{\jmath+1} $: Fixing other variables except for $ X_i $ in \eqref{GSMR:Lag},  we update $ X_i^{\jmath+1} $ by the following sub-problem:
\begin{equation}\label{solve:X}
	X_i^{\jmath+1} \leftarrow \min _{X_i}~ n_i\phi\left(\left\| X_i \right\|_p\right) +\frac{\eta}{2}\left\|X_i{(Y_i^{\jmath})^T} - G_i^{\jmath}\right\|_F^2,
\end{equation}
where $$ G_i^{\jmath}=C^\jmath-\sum_{l=1}^{i-1}X_l^{\jmath+1}(Y_l^{\jmath+1})^T-\sum_{l=i+1}^{s}X_l^{\jmath}(Y_l^{\jmath})^T-\frac{S}{\eta}.$$

Next, we propose a new acceleration method to solve problem \eqref{solve:X}. First, we give an extrapolated point $ {\bar X}_i^{\jmath}= X_i^{\jmath}+w_{x_i}^{\jmath}\left(X_i^{\jmath}-X_i^{\jmath-1}\right) $. Then we solve $ X_i^{\jmath+1} $ by solving the following problem:
\begin{equation}\label{opt:X}
	\begin{aligned}
		X_i^{\jmath+1} &= \min _{X_i}~ n_i\phi\left(\left\| X_i \right\|_p\right) +\frac{\sigma_{X_i}^{\jmath}}{2}\left\|X_i-\left({\bar X}_i^{\jmath}-\left( {\bar X}_i^{\jmath}{(Y_i^{\jmath})^T} - G_i^{\jmath} \right)Y_i^{\jmath}  / \tau_{X_i}^{\jmath}\right) \right\|_F^2\\
		&= \operatorname{prox}_{n_i/\sigma_{X_i}^{\jmath}}^{\phi(\left\| \cdot \right\|_p)}\left({\bar X}_i^{\jmath}-\left( {\bar X}_i^{\jmath}{(Y_i^{\jmath})^T} - G_i^{\jmath} \right)Y_i^{\jmath}  / \tau_{X_i}^{\jmath}\right),
	\end{aligned}
\end{equation}
where $ \sigma_{X_i}^{\jmath} = \eta\tau_{X_i}^{\jmath} $. To solve \eqref{opt:X}, we need to introduce the following lemma\footnote{For simplicity, we only consider the case $p = 2$ in our algorithm.}.

\begin{lemma}\cite[Lemma 1]{ZZW22}\label{thm:lfp}
	For a positive numbers $ \lambda $, the proximal operator of $ \operatorname{prox}_{\lambda}^{\phi(\left\| \cdot \right\|_F)}(Z) $  has a closed-form solution, i.e.,
	\begin{equation*}
		\begin{aligned}
			X^\star=\operatorname{prox}_{\lambda}^{\phi(\left\| \cdot \right\|_F)}(Z):=\arg \min_{X}\left\{\lambda\phi\left(\left\|X\right\|_F\right)+\frac{1}{2}\|X-Z\|_F^2\right\}
			=\left\{ \begin{array}{l}
				\psi\left(\left\|Z\right\|_F\right)\frac{Z}{\left\|Z\right\|_F}, \quad Z \neq 0,\\
				0, \quad Z = 0,
			\end{array} \right.
		\end{aligned}
	\end{equation*}
	where 
	$$ \psi(z) \in \operatorname{prox}_{\lambda \phi}\left(z\right):=\arg\min_{x \in \mathbb{R}_+} \left\lbrace \lambda\phi\left(x\right)+\frac{1}{2}\left(x-z\right)^2 \right\rbrace.$$
\end{lemma}

2)  Computing $ Y_i^{\jmath+1} $: Fixing other variables except for $ Y_i $ in \eqref{GSMR:Lag},  we update $ Y_i^{\jmath+1} $ by the following sub-problem:
\begin{equation}\label{solve:Y}
	Y_i^{\jmath+1} \leftarrow \min _{Y_i} ~ n_i\phi\left(\left\| Y_i \right\|_p\right) +\frac{\eta}{2}\left\|X_i^{\jmath+1}{Y_i^T} - G_i^{\jmath}\right\|_2^2.
\end{equation}
Likewise, we can obtain $Y_i^{\jmath+1}$ through the linearization and the proximal algorithm:
\begin{equation}\label{opt:Y}
	\begin{aligned}
		Y_i^{\jmath+1} &= \min _{Y_i}~ n_i\phi\left(\left\| Y_i \right\|_p\right) +\frac{\sigma_{Y_i}^{\jmath}}{2}\left\|Y_i-\left({\bar Y}_i^{\jmath}-\left(  X_i^{\jmath+1}{({\bar Y}_i^{\jmath})^T} - G_i^{\jmath} \right)^TX_i^{\jmath+1}  / \tau_{Y_i}^{\jmath}\right) \right\|_F^2\\
		&= \operatorname{prox}_{n_i/\sigma_{Y_i}^{\jmath}}^{\phi(\left\| \cdot \right\|_p)}\left({\bar Y}_i^{\jmath}-\left(  X_i^{\jmath+1}{({\bar Y}_i^{\jmath})^T} - G_i^{\jmath} \right)^TX_i^{\jmath+1} / \tau_{Y_i}^{\jmath}\right),
	\end{aligned}
\end{equation}
where $ \sigma_{Y_i}^{\jmath} = \eta\tau_{Y_i}^{\jmath} $.

3)  Computing $ C^{\jmath+1} $: Fixing other variables except for $ C $ in \eqref{GSMR:Lag},  we update $ C^{\jmath+1} $ by the following sub-problem:
\begin{equation}\label{solve:C}
	C^{\jmath+1} \leftarrow \min _{C}~ \frac{\eta}{2}\left\|C-X^{\jmath+1}(Y^{\jmath+1})^T - \frac{S}{\eta}\right\|_F^2,\quad
	\mbox{\rm s.t.}\quad  \left\|\mathscr{A}\left( C \right) - \bm{b}\right\|_2\le\sigma.
\end{equation}
Thus, the optimal solution of \eqref{solve:C} is
\begin{equation}\label{opt:C}
	C^{\jmath+1} =\Pi_{\Theta}\left(X^{\jmath+1}(Y^{\jmath+1})^T + \frac{S}{\eta}\right),
\end{equation}
where $ \Pi_{\Theta} $ denotes the projection onto set $ \Theta $.

We summary the solving algorithm for subproblem \eqref{IAL} in Algorithm \ref{alg2}.

\begin{algorithm}
	\caption{ELAM algorithm for subproblem \eqref{IAL}}
	\label{alg2}
	\begin{algorithmic}[]
		\REQUIRE Initial point $ X^{0},\,Y^{0}$. 
		Let $ X^{-1}:=X^{0},\,Y^{-1}:=Y^{0}$. Let $\jmath:=0$.
		\FOR{$ i=1,\ldots,s $}
		\STATE \bm{Step~ 1.} Compute $ w_{x_i}^{\jmath} $ and $ {\bar X}_i^{\jmath}= X_i^{\jmath}+w_{x_i}^{\jmath}\left(X_i^{\jmath}-X_i^{\jmath-1}\right) $.
		\STATE \bm{Step ~2.} Compute $ X_i^{\jmath+1} $ by \eqref{opt:X}.
		\STATE \bm{Step~ 3.} Compute $ w_{y_i}^{\jmath} $ and $ {\bar Y}_i^{\jmath} = Y_i^{\jmath}+w_{y_i}^{\jmath}\left(Y_i^{\jmath}-Y_i^{\jmath-1}\right) $.
		\STATE \bm{Step ~4.} Compute $ Y_i^{\jmath+1} $ by \eqref{opt:Y}.
		\ENDFOR
		\STATE \bm{Step ~5.} Remove the zero columns of $ X^{\jmath+1} $ and $ Y^{\jmath+1}$. 
		\STATE \bm{Step ~6.} Compute $ C^{\jmath+1} $ by \eqref{opt:C}.
		\STATE \bm{Step ~7.} Let $ \jmath:=\jmath+1 $. If the stop criterion is not met, return to $ \bm{Step~1}$.
		\ENSURE $ X^{\jmath} $, $ Y^{\jmath} $, $ C^{\jmath} $.
	\end{algorithmic}
\end{algorithm}

\begin{remark}
	In this paper, we set $\tau_{X_i}^0=\tau_{Y_i}^0=1$, and for any $\jmath\ge 1$, let $ \tau_{X_i}^{\jmath}=\max\left\lbrace \gamma\left\|Y_i^\jmath\right\|^2,\varepsilon\right\rbrace  $ and $ \tau_{Y_i}^{\jmath}=\max\left\lbrace \gamma\left\|X_i^{\jmath+1}\right\|^2,\varepsilon\right\rbrace  $ with $ \gamma>1 $.
	In addition, we take	
	\begin{equation*}
		w_{x_i}^{\jmath}=\min \left(\hat{\gamma}_{i}^{\jmath}, \frac{\delta(\gamma-1)}{2(\gamma+1)} \sqrt{\frac{\tau_{X_{i}}^{\jmath-1}}{\tau_{X_i}^{\jmath}}}\right), \;w_{y_i}^{\jmath}=\min \left(\hat{\gamma}_{i}^{\jmath}, \frac{\delta(\gamma-1)}{2(\gamma+1)} \sqrt{\frac{\tau_{Y_{i}}^{\jmath-1}}{\tau_{Y_i}^{\jmath}}}\right),
	\end{equation*}	
	where $\delta<1$ and
	\begin{equation*}
		\begin{aligned}
			\hat{\gamma}_{i}^{\jmath} &=\frac{t_{i-1}^{\jmath}-1}{t_i^{\jmath}}, \quad
			t_{0}^{1} =1, t_{0}^{\jmath}=t_{s}^{\jmath-1}, \text { for } \jmath \geq 2, \\
			t_{i}^{\jmath} &=\frac{1}{2}\left(1+\sqrt{1+4\left(t_{i-1}^{\jmath}\right)^{2}}\right), \text { for } \jmath \geq 1, i=1, \ldots, s .
		\end{aligned}
	\end{equation*}	
\end{remark}

\iffalse
\begin{remark}
	Since $ \Phi\left(X\right) $ and $ \tilde\Phi\left(Y\right) $ regularizer are non-convex, we cannot use the method similar to \cite{Xu15} to prove the convergence of Algorithm \ref{alg2}. Recently, there is growing interest in the use of non-convex regularizers with acceleration method, such as nmAPG \cite{LL15} and FaNCL-acc \cite{YKWL19}. However, these algorithms have high computational complexity because they require more than one proximal step in each iteration. Thus, \cite{YKG17,YX22} propose an efficient proximal gradient algorithm that requires only one inexact proximal step in each iteration. The three methods need to determine whether to adopt the acceleration step in each iteration, so it is possible that only a few steps adopt the acceleration step, resulting in the acceleration effect is not obvious. Here we can not only ensure that each step is accelerated, but also give the convergence of the algorithm.
\end{remark}
\fi

\subsection{Convergence analysis of Algorithm \ref{alg1}} \label{sub:con_IAL}

The following theorem states the main result of our convergence analysis for the proposed Algorithm \ref{alg1}.
\begin{theorem}
	Suppose that the sequence $ \left\lbrace X^k,Y^k,C^k,S^k\right\rbrace_{k \in \mathbb{N}^+}$ generated by Algorithm \ref{alg1} is bounded. Then the following statements hold:
	\begin{itemize}
		\item[$(i)$] $\lim_{k\to \infty}\left\|X^k(Y^k)^T-C^k\right\|_F=0$;
		\item[$(ii)$] any accumulation point $\left(X^{\star},Y^{\star},C^{\star},S^{\star}\right)$ of $ \left\lbrace X^k,Y^k,C^k,S^k\right\rbrace_{k \in \mathbb{N}^+}$ is a stationary point of problem \eqref{GSMR}.
	\end{itemize}
\end{theorem}
\begin{proof}
$(i)$ We split the proof into two cases.

Case 1. The sequence $\left\lbrace \eta^k\right\rbrace $ is bounded. In this case, \eqref{eq:eta} happens finite times at most, which means that there exists $k_0>\vartheta$ such that $\eta^k = \eta^{k_0}$ and 
$$\left\|X^{k+1}(Y^{k+1})^T-C^{k+1}\right\|_F \le \rho_1 \min_{t=k-\vartheta+1,\ldots,k}\left\|X^t(Y^t)^T-C^t\right\|_F$$
for all $k>k_0$. We obtain from the above inequality that for all $k>k_0$, 
$$\left\|X^{k+1}(Y^{k+1})^T-C^{k+1}\right\|_F \le \rho_1 \left\|X^k(Y^k)^T-C^k\right\|_F.$$
Together this inequality with $ \rho_1 \in (0,1) $ implies that
$\lim_{k\to \infty}\left\|X^k(Y^k)^T-C^k\right\|_F=0$.
	
Case 2. The sequence $\left\lbrace \eta^k\right\rbrace $ is unbounded. In this case, the set 
\begin{eqnarray}\label{set:K}
	{\mathcal K} = \left\lbrace k: \eta^{k+1} = \eta^k/\rho_2 \right\rbrace 
\end{eqnarray}
is infinite. Thanks to $\rho_2 \in (0,1)$, \eqref{set:K} leads to $\lim_{k\to\infty,k\in{\mathcal K}}\eta^k=\infty$. Given $k>\vartheta$, let $t_k$ be the largest element in $\mathcal K$ satisfying $t_k \le k$. Subsequently, we demonstrate that
\begin{eqnarray}\label{eq:tk}
	\left\|X^{k+1}(Y^{k+1})^T-C^{k+1}\right\|_F \le \frac{\left\|S^{t_k}\right\|_F}{\eta^{t_k}} +   \frac{\left\|S^{t_k+1}\right\|_F}{\eta^{t_k}}.
\end{eqnarray}
It is clear that the inequality \eqref{eq:tk} holds when $t_k=k$. Therefore, we only need to consider $t_k<k$ in the following.  Combining \eqref{eq:min} and \eqref{opt:multipliers}, one has
$$
\left\|X^{k+1}(Y^{k+1})^T-C^{k+1}\right\|_F\le \left\|X^k(Y^k)^T-C^k\right\|_F\le \ldots \le\left\|X^{t_k+1}(Y^{t_k+1})^T-C^{t_k+1}\right\|_F\le \frac{\left\|S^{t_k}\right\|_F}{\eta^{t_k}} +   \frac{\left\|S^{t_k+1}\right\|_F}{\eta^{t_k}}.
$$
Together with the boundedness of $S^k$, $t_k \in {\mathcal K}$ and $\lim_{k\to\infty,k\in{\mathcal K}}\eta^k=\infty$, we know that
$$ \lim_{k\to\infty}\left\|X^{k+1}(Y^{k+1})^T-C^{k+1}\right\|_F\le\lim_{k\to\infty}\frac{\left\|S^{t_k}\right\|_F}{\eta^{t_k}}+\lim_{k\to\infty}\frac{\left\|S^{t_k+1}\right\|_F}{\eta^{t_k}}=0. $$
The above inequality yield that $\lim_{k\to \infty}\left\|X^k(Y^k)^T-C^k\right\|_F=0$. Thus, we complete the proof of statement $(i)$.

$(ii)$ From Bolzano-Weierstrass Theorem \cite{Bro12}, $ \left\lbrace X^k,Y^k,C^k,S^k\right\rbrace_{k \in \N^+}$ has at least one accumulation point $\left(X^{\star},Y^{\star},C^{\star}, S^{\star}\right)$ and there exists a subsequences that converges to this accumulation point. Without loss of generality, we assume that the sequence is $ \left\lbrace X^k,Y^k,C^k,S^k\right\rbrace$. 
Recalling result $(i)$ of this theorem, we obtain that this accumulation point is feasible.

Together with the inequality \eqref{eq:I} and the definition of $L_{\eta}\left(X,Y,C;S\right)$, there exists $\zeta^{k+1} \in \partial_Z L_{\eta}\left(X,Y,C;S\right)$ satisfying $\left\|\zeta^{k+1}\right\|_F\le \epsilon_k $ such that
\begin{eqnarray}\label{zeta}
\begin{aligned}
	\zeta^{k+1} \in& \partial_{Z}\left[ \Phi\left(X^{k+1}\right)+\tilde\Phi\left(Y^{k+1}\right)+l_{\Theta}\left(C^{k+1}\right)\right]+\nabla_{Z}\left[\left\langle h\left(Z^{k+1}\right),S^k \right\rangle + \frac{\eta^k}{2}\left\|h\left(Z^{k+1}\right)\right\|_F^2\right]\\
	=&\partial_{Z}\left[ \Phi\left(X^{k+1}\right)+\tilde\Phi\left(Y^{k+1}\right)+l_{\Theta}\left(C^{k+1}\right)\right]+\left\langle \nabla_{Z}h\left(Z^{k+1}\right),S^k+\eta^kh\left(Z^{k+1}\right) \right\rangle,
\end{aligned}
\end{eqnarray}
where $Z = (X,Y,C)$ and $h\left(Z\right) = XY^T-C $.

In view of \eqref{opt:multipliers} hold when $k\in {\mathcal K}$, we then obtain that
\begin{equation}\label{eq:Sk-1}
	\lim_{k\in {\mathcal K}, k\to \infty}S^k+\eta^kh\left(Z^{k+1}\right)=\lim_{k\in {\mathcal K},k\to \infty} S^{k+1}.
\end{equation}
If $k\notin {\mathcal K}$, by $S^{k+1} = 0,~ \eta^{k+1}=\eta^k$ and result $(i)$ of this Theorem, we get
\begin{equation}\label{eq:Sk-2}
	\lim_{k\notin {\mathcal K},k\to \infty}S^k+\eta^kh\left(Z^{k+1}\right)=\lim_{k\notin {\mathcal K},k\to \infty} S^{k+1}.
\end{equation}
Combining \eqref{eq:Sk-1} with \eqref{eq:Sk-2}, we obtain that
\begin{eqnarray}\label{eq:Sk}
	\lim_{k\to \infty}S^k+\eta^kh\left(Z^{k+1}\right)=\lim_{k\to \infty} S^{k+1}.
\end{eqnarray}

By the update rule of Algorithm \ref{alg1}, we have $\epsilon_k \leq \max \left\{\sqrt{\rho_1}, \rho_3\right\} \epsilon_{k-1}$ for all $k>\vartheta$. Combining this with the fact that $\rho_1,\rho_3 \in (0,1)$, we obtain $\lim _{k \to \infty} \epsilon_k=0$. Consequently, we can infer that $\lim _{k \to \infty} \zeta^{k+1}=0$, which together with \eqref{zeta} and \eqref{eq:Sk} implies that
\begin{eqnarray*}
	0 \in \partial_{Z}\left[ \Phi\left(X^{\star}\right)+\tilde\Phi\left(Y^{\star}\right)+l_{\Theta}\left(C^{\star}\right)\right]+\nabla_{Z}\left[\left\langle X^{\star}(Y^{\star})^T-C^{\star},S^{\star} \right\rangle\right].
\end{eqnarray*}
Hence, $\left(X^{\star},Y^{\star},C^{\star},S^{\star}\right)$ is a stationary point of problem \eqref{GSMR}, which completes the statement $(ii)$.

\end{proof}

\subsection{Convergence analysis of Algorithm \ref{alg2}} \label{sub:con_ELAL}
In this subsection, we prove the convergence of Algorithm \ref{alg2} for solving subproblem \eqref{IAL}. The main results are given in Theorem \ref{thm:local} below. We first give some lemmas.

To simply the notation, we denote $ f_{X_i}:=\frac{1}{2}\left\|X_i{(Y_i^{\jmath})^T} - G_i^{\jmath}\right\|_F^2 $ and $ f_{Y_i}:=\frac{1}{2}\left\|X_i^{\jmath+1}{Y_i^T} - G_i^{\jmath}\right\|_F^2 $ in this subsection.

\begin{lemma}\label{lem:mono}
	Let $ \left\lbrace X^\jmath,Y^\jmath,C^\jmath \right\rbrace  $ be the sequence generated by Algorithm \ref{alg2}, then
	\begin{equation*}
		\begin{aligned}
			&L_{\eta}\left(X^{\jmath+1},\,Y^{\jmath+1},\,C^{\jmath+1};S\right)- L_{\eta}\left(X^{\jmath},Y^{\jmath},\,C^{\jmath};S\right) \\
			\le& -\frac{\eta}{2}\left\|C^{\jmath+1}-C^{\jmath}\right\|_F^2+\frac{\gamma-1}{4\gamma}\sum_{i=1}^s\eta\tau_{X_i}^{\jmath-1}\delta^2\left\|X_i^{\jmath}-X_i^{\jmath-1}\right\|_F^2-\frac{\gamma-1}{4\gamma}\sum_{i=1}^s\eta\tau_{X_i}^{\jmath} \left\|X_i^{\jmath+1}-X_i^{\jmath}\right\|_F^2 \\ &+\frac{\gamma-1}{4\gamma}\sum_{i=1}^s\eta\tau_{Y_i}^{\jmath-1}\delta^2\left\|Y_i^{\jmath}-Y_i^{\jmath-1}\right\|_F^2-\frac{\gamma-1}{4\gamma}\sum_{i=1}^s\eta\tau_{Y_i}^{\jmath} \left\|Y_i^{\jmath+1}-Y_i^{\jmath}\right\|_F^2.
		\end{aligned}
	\end{equation*}
\end{lemma}
\begin{proof}
	From the Lipschitz continuity of $  \nabla_{X_i}f_{X_i} $ about $ X_i $, it holds that
	\begin{equation}\label{mono-1}
		f_{X_i^{\jmath+1}} \le f_{X_i^{\jmath}}+\left\langle \nabla_{X_i}f_{X_i^{\jmath}}, X_i^{\jmath+1}-X_i^{\jmath} \right\rangle+\frac{\tau_{X_i}^{\jmath}}{2\gamma}\left\|X_i^{\jmath+1}-X_i^{\jmath}\right\|_F^2.
	\end{equation}
	Since $ X_i^{\jmath+1} $ is the minimizer of \eqref{opt:X}, then
	\begin{equation}\label{mono-2}
		\begin{aligned}
			&n_i\phi\left(\left\| X_i^{\jmath+1} \right\|_p\right)+\frac{\sigma_{X_i}^{\jmath}}{2}\left\|X_i^{\jmath+1}-\bar X_i^{\jmath}\right\|_F^2+\eta\left\langle X_i^{\jmath+1}-\bar X_i^{\jmath}, \nabla_{X_i}f_{\bar X_i^{\jmath}}
			\right\rangle\\
			\le & n_i\phi\left(\left\| X_i^{\jmath} \right\|_p\right)+\frac{\sigma_{X_i}^{\jmath}}{2}\left\|X_i^{\jmath}-\bar X_i^{\jmath}\right\|_F^2+\eta\left\langle X_i^{\jmath}-\bar X_i^{\jmath}, \nabla_{X_i}f_{\bar X_i^{\jmath}} \right\rangle.
		\end{aligned}
	\end{equation}
	Summing \eqref{mono-1} and \eqref{mono-2}, we have
	\begin{equation}\label{mono-X}
		\begin{aligned}
			&n_i\phi\left(\left\| X_i^{\jmath+1} \right\|_p\right)+\eta f_{X_i^{\jmath+1}} - n_i\phi\left(\left\| X_i^{\jmath} \right\|_p\right)-\eta f_{X_i^{\jmath}} \\
			\le & \eta\left\langle \nabla_{X_i}f_{X_i^{\jmath}}-\nabla_{X_i}f_{\bar X_i^{\jmath}}, X_i^{\jmath+1}-X_i^{\jmath} \right\rangle+\frac{\sigma_{X_i}^{\jmath}}{2\gamma}\left\|X_i^{\jmath+1}-X_i^{\jmath}\right\|_F^2 +\frac{\sigma_{X_i}^{\jmath}}{2}\left\|X_i^{\jmath}-\bar X_i^{\jmath}\right\|_F^2-\frac{\sigma_{X_i}^{\jmath}}{2}\left\|X_i^{\jmath+1}-\bar X_i^{\jmath}\right\|_F^2\\
			=& \eta\left\langle \nabla_{X_i}f_{X_i^{\jmath}}-\nabla_{X_i}f_{\bar X_i^{\jmath}}, X_i^{\jmath+1}-X_i^{\jmath} \right\rangle+\sigma_{X_i}^{\jmath}\left\langle X_i^{\jmath}-X_i^{\jmath+1}, X_i^{\jmath}-\bar X_i^{\jmath} \right\rangle -\frac{\sigma_{X_i}^{\jmath}\left(\gamma-1\right)}{2\gamma} \left\|X_i^{\jmath+1}-X_i^{\jmath}\right\|_F^2\\
			\le & \eta\left\|X_i^{\jmath+1}-X_i^{\jmath}\right\|_F\left(\left\|\nabla_{X_i}f_{X_i^{\jmath}}-\nabla_{X_i}f_{\bar X_i^{\jmath}}\right\|_F+\tau_{X_i}^{\jmath}\left\|X_i^{\jmath}-\bar X_i^{\jmath} \right\|_F\right)-\frac{\sigma_{X_i}^{\jmath}\left(\gamma-1\right)}{2\gamma} \left\|X_i^{\jmath+1}-X_i^{\jmath}\right\|_F^2 \\
			\le & \frac{\sigma_{X_i}^{\jmath}\left(1+\gamma\right) }{\gamma}\left\|X_i^{\jmath+1}-X_i^{\jmath}\right\|_F\left\|X_i^{\jmath}-\bar X_i^{\jmath} \right\|_F-\frac{\sigma_{X_i}^{\jmath}\left(\gamma-1\right)}{2\gamma} \left\|X_i^{\jmath+1}-X_i^{\jmath}\right\|_F^2\\
			=&\frac{\sigma_{X_i}^{\jmath}w_{x_i}^{\jmath}\left(1+\gamma\right) }{\gamma}\left\|X_i^{\jmath+1}-X_i^{\jmath}\right\|_F\left\|X_i^{\jmath}- X_i^{\jmath-1} \right\|_F-\frac{\sigma_{X_i}^{\jmath}\left(\gamma-1\right)}{2\gamma} \left\|X_i^{\jmath+1}-X_i^{\jmath}\right\|_F^2\\
			\le&\frac{\eta\tau_{X_i}^{\jmath}\left(1+\gamma\right)^2 }{\gamma\left(\gamma-1\right) }\left(w_{x_i}^{\jmath}\right)^2\left\|X_i^{\jmath}-X_i^{\jmath-1}\right\|_F^2-\frac{\eta\tau_{X_i}^{\jmath}\left(\gamma-1\right)}{4\gamma} \left\|X_i^{\jmath+1}-X_i^{\jmath}\right\|_F^2 \\
			\le & \frac{\eta\tau_{X_i}^{\jmath-1}\left(\gamma-1\right)}{4\gamma}\delta^2\left\|X_i^{\jmath}-X_i^{\jmath-1}\right\|_F^2-\frac{\eta\tau_{X_i}^{\jmath}\left(\gamma-1\right)}{4\gamma} \left\|X_i^{\jmath+1}-X_i^{\jmath}\right\|_F^2.
		\end{aligned}
	\end{equation}
	Here, we have used Cauchy–Schwarz inequality in the second inequality, Lipschitz continuity of $  \nabla_{X_i}f_{X_i} $ about $ X_i $ in the third one, the Young’s inequality in the fourth one, and $ w_{x_i}^{\jmath}\le \frac{\delta(\gamma-1)}{2(\gamma+1)} \sqrt{\frac{\tau_{X_{i}}^{\jmath-1}}{\tau_{X_i}^{\jmath}}} $ to get the last inequality, the fact $ \bar X_i^{\jmath}= X_i^{\jmath}+w_{x_i}^{\jmath}\left(X_i^{\jmath}-X_i^{\jmath-1}\right) $ to have the second equality.
	Similarly, for $ Y_i^{\jmath+1} $, we have
	\begin{equation}\label{mono-Y}
		\begin{aligned}
			&n_i\phi\left(\left\| Y_i^{\jmath+1} \right\|_p\right)+\eta f_{Y_i^{\jmath+1}} - n_i\phi\left(\left\| Y_i^{\jmath} \right\|_p\right)-\eta f_{Y_i^{\jmath}} \\
			\le & \frac{\eta\tau_{Y_i}^{\jmath-1}\left(\gamma-1\right)}{4\gamma}\delta^2\left\|Y_i^{\jmath}-Y_i^{\jmath-1}\right\|_F^2-\frac{\eta\tau_{Y_i}^{\jmath}\left(\gamma-1\right)}{4\gamma} \left\|Y_i^{\jmath+1}-Y_i^{\jmath}\right\|_F^2.
		\end{aligned}
	\end{equation}
	Summing up \eqref{mono-X} and \eqref{mono-Y} over $ i $ from 1 to $ s $ gives
	\begin{equation}\label{mono-XY}
		\begin{aligned}
			&L_{\eta}\left(X^{\jmath+1},Y^{\jmath+1},C^{\jmath};S\right)-L_{\eta}\left(X^{\jmath},Y^{\jmath},C^{\jmath};S\right)\\
			=& \Phi\left(X^{\jmath+1}\right)+\tilde\Phi\left(Y^{\jmath+1}\right)+\frac{\eta}{2}\left\|X^{\jmath+1}\left(Y^{\jmath+1}\right)^T-C^{\jmath}+\frac{S}{\eta}\right\|_F^2
			-\Phi\left(X^{\jmath}\right)-\tilde\Phi\left(Y^{\jmath}\right)-\frac{\eta}{2}\left\|X^{\jmath}\left(Y^{\jmath}\right)^T-C^{\jmath}+\frac{S}{\eta}\right\|_F^2 \\
			=& \sum_{i=1}^s n_i\left[ \phi\left(\left\| X_i^{\jmath+1} \right\|_p\right) - \phi\left(\left\| X_i^{\jmath} \right\|_p\right) +\phi\left(\left\| Y_i^{\jmath+1} \right\|_p\right)- \phi\left(\left\| Y_i^{\jmath} \right\|_p\right)\right]+\eta\left(f_{Y_s^{\jmath+1}}-f_{X_1^{\jmath}}\right)     \\
			=& \sum_{i=1}^s n_i\left[ \phi\left(\left\| X_i^{\jmath+1} \right\|_p\right) - \phi\left(\left\| X_i^{\jmath} \right\|_p\right) +\phi\left(\left\| Y_i^{\jmath+1} \right\|_p\right)- \phi\left(\left\| Y_i^{\jmath} \right\|_p\right)\right]+\eta\left(f_{Y_s^{\jmath+1}}-f_{X_1^{\jmath}}\right)\\
			&+\eta\sum_{i=1}^{s-1}\left(f_{Y_i^{\jmath+1}}-f_{X_{i+1}^{\jmath}}\right)+\eta\sum_{i=1}^s\left(f_{X_i^{\jmath+1}}-f_{Y_i^{\jmath}}\right)     \\
			=& \sum_{i=1}^s \left[ n_i\phi\left(\left\| X_i^{\jmath+1} \right\|_p\right)+\eta f_{X_i^{\jmath+1}} - n_i\phi\left(\left\| X_i^{\jmath} \right\|_p\right)-\eta f_{X_i^{\jmath}}\right]\\
			&+ \sum_{i=1}^s \left[ n_i\phi\left(\left\| Y_i^{\jmath+1} \right\|_p\right)+\eta f_{Y_i^{\jmath+1}} - n_i\phi\left(\left\| Y_i^{\jmath} \right\|_p\right)-\eta f_{Y_i^{\jmath}}\right]                                       \\
			\le &\frac{\gamma-1}{4\gamma}\sum_{i=1}^s\eta\tau_{X_i}^{\jmath-1}\delta^2\left\|X_i^{\jmath}-X_i^{\jmath-1}\right\|_F^2-\frac{\gamma-1}{4\gamma}\sum_{i=1}^s\eta\tau_{X_i}^{\jmath} \left\|X_i^{\jmath+1}-X_i^{\jmath}\right\|_F^2 \\ &+\frac{\gamma-1}{4\gamma}\sum_{i=1}^s\eta\tau_{Y_i}^{\jmath-1}\delta^2\left\|Y_i^{\jmath}-Y_i^{\jmath-1}\right\|_F^2-\frac{\gamma-1}{4\gamma}\sum_{i=1}^s\eta\tau_{Y_i}^{\jmath} \left\|Y_i^{\jmath+1}-Y_i^{\jmath}\right\|_F^2,
		\end{aligned}
	\end{equation}
	where the third equality comes from $ f_{Y_i^{\jmath+1}}=f_{X_{i+1}^{\jmath}} $ for $ i\in[s-1] $ and $ f_{X_i^{\jmath+1}}=f_{Y_i^{\jmath}} $ for $ i\in[s] $.
	
	Since $ C^{\jmath+1} $ is optimal to \eqref{solve:C}, we have
	\begin{equation}\label{mono-C}
		L_{\eta}\left(X^{\jmath+1},Y^{\jmath+1},C^{\jmath+1};S\right)-L_{\eta}\left(X^{\jmath+1},Y^{\jmath+1},C^{\jmath};S\right) \le -\frac{\eta}{2}\left\|C^{\jmath+1}-C^{\jmath}\right\|_F^2.
	\end{equation}
	Summing up \eqref{mono-XY} and \eqref{mono-C} completes the proof.
\end{proof}
\iffalse
\begin{lemma}
	Suppose that $A_n$ for $n\in[N]$ are $N$ arbitrary matrices. Then
	$$\left\|A_1+\ldots+A_N\right\|_F^2 \le N\left(\left\|A_1\right\|_F^2+\ldots+\left\|A_N\right\|_F^2\right).$$
\end{lemma}
\begin{proof}
	Let $f(X)=\left\|X\right\|_F^2$, which is a convex function. By Jensen's inequality, one has
	$$
	f\left(\frac{1}{N}\sum_{n=1}^N A_n\right) \le \frac{1}{N}\sum_{n=1}^N f\left(A_n\right).
	$$
	This implies that $\left\|A_1+\ldots+A_N\right\|_F^2 \le N\left(\left\|A_1\right\|_F^2+\ldots+\left\|A_N\right\|_F^2\right)$.
\end{proof}
\fi

\begin{lemma}\label{lem:lim}
Let $ \left\lbrace X^\jmath,Y^\jmath,C^\jmath \right\rbrace  $ be the sequence generated by Algorithm \ref{alg2}, then
\begin{equation}\label{lim-1}
	\lim _{\jmath \to \infty}\left(X^{\jmath+1}-X^{\jmath}\right)=0,\quad\lim _{\jmath \to \infty}\left(Y^{\jmath+1}-Y^{\jmath}\right)=0,\quad \lim _{\jmath \to \infty}\left(C^{\jmath+1}-C^{\jmath}\right)=0.
\end{equation}
\end{lemma}
\begin{proof}
By Lemma \ref{lem:mono}, one has
\begin{equation}\label{bound-3}
	\begin{aligned}
		&L_{\eta}\left(X^{{\mathcal J}+1},Y^{{\mathcal J}+1},C^{{\mathcal J}+1};S\right)- L_{\eta}\left(X^0,Y^0,C^0;S\right)+\frac{\eta}{2}\sum_{\jmath=1}^{\mathcal J}\left\|C^{\jmath+1}-C^{\jmath}\right\|_F^2 \\
		\le& \frac{\gamma-1}{4\gamma}\sum_{\jmath=1}^{\mathcal J}\sum_{i=1}^s\eta\tau_{X_i}^{\jmath-1}\left(\delta^2-1\right)\left\|X_i^{\jmath}-X_i^{\jmath-1}\right\|_F^2+\frac{\gamma-1}{4\gamma}\sum_{i=1}^s\eta\tau_{X_i}^0 \left\|X_i^1-X_i^0\right\|_F^2 \\ &+\frac{\gamma-1}{4\gamma}\sum_{\jmath=1}^{\mathcal J}\sum_{i=1}^s\eta\tau_{Y_i}^{\jmath-1}\left(\delta^2-1\right)\left\|Y_i^{\jmath}-Y_i^{\jmath-1}\right\|_F^2+\frac{\gamma-1}{4\gamma}\sum_{i=1}^s\eta\tau_{Y_i}^0 \left\|Y_i^1-Y_i^0\right\|_F^2.
	\end{aligned}
\end{equation}
Note that 
\begin{equation}
\begin{aligned}
&L_{\eta}\left(X^{{\mathcal J}+1},Y^{{\mathcal J}+1},C^{{\mathcal J}+1};S\right)\\
=&\Phi\left(X^{{\mathcal J}+1}\right)+\tilde\Phi\left(Y^{{\mathcal J}+1}\right)+l_{\Theta}\left(C^{{\mathcal J}+1}\right)+ \frac{\eta}{2}\left\|X^{{\mathcal J}+1}\left(Y^{{\mathcal J}+1}\right)^T-C^{{\mathcal J}+1}+\frac{S}{\eta}\right\|_F^2-\frac{\left\|S\right\|_F^2}{2\eta}\\
\ge& -\frac{\left\|S\right\|_F^2}{2\eta}>-\infty,
\end{aligned}
\end{equation}
where the last inequality uses  the boundedness of $S$.
Combining this with \eqref{bound-3} and $ \tau_{X_i}^{\jmath-1}\ge \varepsilon $, $ \tau_{Y_i}^{\jmath-1}\ge \varepsilon $ gives
 \begin{equation*}
 	\sum_{\jmath=1}^\infty\left\|X^{\jmath+1}-X^{\jmath}\right\|_F^2<\infty,\quad \sum_{\jmath=1}^\infty\left\|Y^{\jmath+1}-Y^{\jmath}\right\|_F^2<\infty, \quad \sum_{\jmath=1}^\infty\left\|C^{\jmath+1}-C^{\jmath}\right\|_F^2<\infty.
 \end{equation*}
Therefore, $ \lim _{\jmath \to \infty}\left(X^{\jmath+1}-X^{\jmath}\right)=0 $, $ \lim _{\jmath \to \infty}\left(Y^{\jmath+1}-Y^{\jmath}\right)=0 $, $ \lim _{\jmath \to \infty}\left(C^{\jmath+1}-C^{\jmath}\right)=0 $.
\end{proof}

\begin{theorem}\label{thm:local}
Suppose that the sequence $ \left\lbrace X^{\jmath},Y^{\jmath},C^{\jmath}\right\rbrace$ generated by Algorithm \ref{alg2} is bounded. Then the sequence  $ \left\lbrace X^{\jmath},Y^{\jmath},C^{\jmath}\right\rbrace$ has at least one accumulation point, and any accumulation point $ \left\lbrace X^{\star},Y^{\star},C^{\star}\right\rbrace$ is a stationary point of the optimization problem \eqref{IAL}.
\end{theorem}

\begin{proof}
From Bolzano-Weierstrass Theorem \cite{Bro12}, Algorithm \ref{alg2} has at least one accumulation point $ \left\lbrace X^\star,Y^\star,C^\star\right\rbrace $ and there exists one sequences that converges to this accumulation point. Without loss of generality, we assume that the sequence is $ \left\lbrace X^{\jmath},Y^{\jmath},C^{\jmath}\right\rbrace $.
	
For the $ X_i $-subproblem, by first-order necessary optimality condition, we get
\begin{equation}\label{fo-X}
	0 \in n_i\partial\phi\left(\left\| X_i^{\jmath+1} \right\|_p\right)+\sigma_{X_i}^{\jmath}\left(X_i^{\jmath+1}-\bar X_i^{\jmath} \right)+\eta\left( \bar X_i^{\jmath}{(Y_i^{\jmath})^T} - G_i^{\jmath} \right)Y_i^{\jmath}.
\end{equation}
According to $ \lim_{\jmath \to\infty}\left(X^{\jmath+1}-X^{\jmath}\right)=0 $ and $ \lim_{\jmath \to\infty}\left(Y^{\jmath+1}-Y^{\jmath}\right)=0 $, we obtain that
\begin{equation}\label{KKT-X}
	0 \in n_i\partial\phi\left(\left\| X_i^\star \right\|_p\right)+\left(\eta X^\star {Y^\star}^T-\eta C^\star+S\right) Y_i^\star.
\end{equation}

For the $ Y_i $-subproblem, by first-order necessary optimality condition, we get
\begin{equation}\label{fo-Y}
	0 \in n_i\partial\phi\left(\left\| Y_i^{\jmath+1} \right\|_p\right)+\sigma_{Y_i}^{\jmath}\left(Y_i^{\jmath+1}-\bar Y_i^{\jmath} \right)+\eta\left(X_i^{\jmath+1}\left(\bar Y_i^{\jmath}\right)^T - G_i^{\jmath} \right)^TX_i^{\jmath+1}.
\end{equation}
According to $ \lim_{\jmath \to\infty}\left(X^{\jmath+1}-X^{\jmath}\right)=0 $ and $ \lim_{\jmath \to\infty}\left(Y^{\jmath+1}-Y^{\jmath}\right)=0 $, we obtain that
\begin{equation}\label{KKT-Y}
	0 \in n_i\partial\phi\left(\left\| Y_i^\star \right\|_p\right)+\left(\eta X^\star {Y^\star}^T-\eta C^\star+S\right)^T X_i^\star.
\end{equation}
	
For the $ C $-subproblem, by first-order necessary optimality condition, we get
\begin{equation*}
	0\in \eta\left(C^{\jmath+1}-X^{\jmath+1}(Y^{\jmath+1})^T - \frac{S}{\eta}\right)+N_\Theta\left(C^{\jmath+1}\right).
\end{equation*}
Thus,
\begin{equation}\label{KKT-C}
	0 \in \eta\left(C^{\star}-X^{\star}(Y^{\star})^T - \frac{S}{\eta}\right) + N_\Theta\left(C^{\star}\right).
\end{equation}
	
By \eqref{KKT-X}, \eqref{KKT-Y} and \eqref{KKT-C}, we know that $ \left\lbrace X^\star,Y^\star,C^\star\right\rbrace$ is a stationary point of the optimization problem \eqref{IAL}, which completes the proof.
\end{proof}

\section{Experimental Results}

In this section, we compare our performance to state-of-the-art matrix completion methods, including group sparsity-based methods FGSR \cite{FDCU19} and GUIG \cite{JFWZ21}, matrix factorization-based method NMFC \cite{XYWZ12}, and nuclear norm-based method SVT \cite{CCS10}.
In order to demonstrate the efficiency of the proposed method, we will present the results of matrix completion experiments on two typical types of matrix data, namely, grayscale image and high altitude aerial images. As part of 
our quantitative evaluation, we use four numerical metrics, namely peak signal-to-noise ratio (PSNR) \cite{WBSS04}, structural similarity index measure (SSIM) \cite{WBSS04} and the recovery computation time.
All experiments are conducted in Matlab R2020b under Windows 11 on a desktop of a 2.50 GHz CPU and 16 GB memory.

Parameter Settings: The parameters of compared algorithms are set as described in their papers, and we take the best results as the final results. In FLGSR, If not specified, $ \eta^0,~\vartheta,~\epsilon_0,~[\rho_1,~\rho_2,~\rho_3] $ are set as 1e-3, 10, 10 and $ [0.999, 0.5, 0.5] $, respectively.
The capped concave function $ \phi $ is set as $\phi^{\operatorname{CapLog}}(t)$. The matrix is divided into 32 groups. All matrix data are prescaled to $ \left[0,1\right]  $.

\subsection{Model analysis}

In this part, we analyze the effectiveness of flexible grouping and the restarted technique in the proposed algorithm FLGSR, and test it on four images from the USC-SIPI Image Database\footnote{https://sipi.usc.edu/database/database.php?volume=misc.}: ``Peppers", ``Sailboat", ``Bridge" and ``Mandrill". All images have a size of $512 \times 512$. The sampling rate (SR) is set to be 70\% in the experiments.

\subsubsection{Effects of flexible grouping}\label{sub:fg}
In this subsection, we compare the effects of different number of groups on our proposed algorithm FLGSR. The curves of PSNR, SSIM and running time with respect to different number of groups are shown in Figure \ref{fig:group}. From the recovery results, as the number of groups increases, the recovery effect of FLGSR first increases and then decreases, reaching the best near 16 groups. In terms of computational time, FLGSR takes significantly longer as the number of groups increases. When the number of groups reaches 512 (that is, each column is treated as a group like FGSR and GUIG), the time it consumes is about six times that of 16 groups. Therefore, the flexible grouping strategy of our proposed FLGSR method is effective in terms of both recovery quality and computational efficiency.

\begin{figure*}[htbp]
	\centering
	\begin{subfigure}[b]{1\linewidth}
		\begin{subfigure}[b]{0.333\linewidth}
			\centering
			\includegraphics[width=\linewidth]{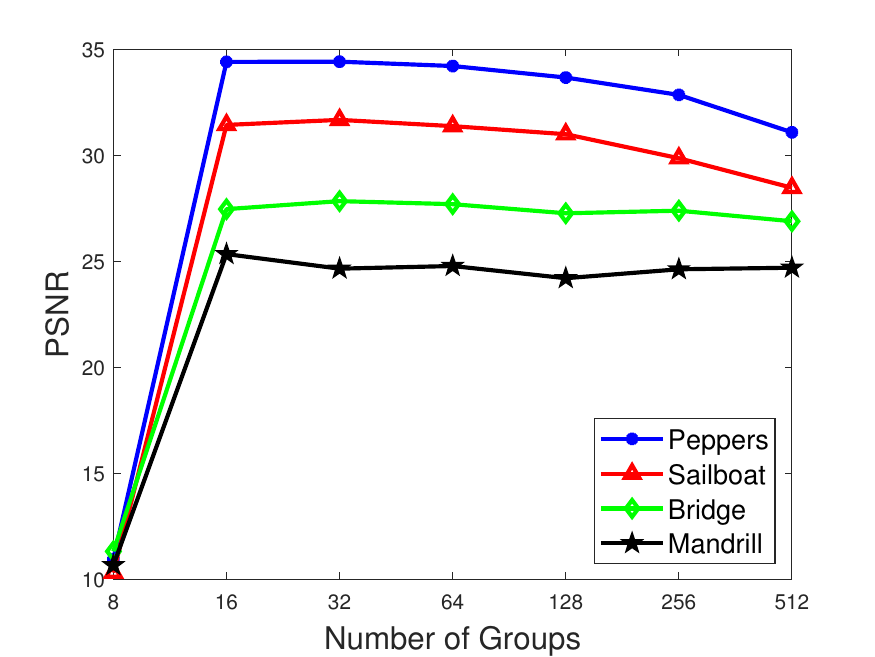}
%			\caption*{PSNR}
		\end{subfigure}   	
		\begin{subfigure}[b]{0.333\linewidth}
			\centering
			\includegraphics[width=\linewidth]{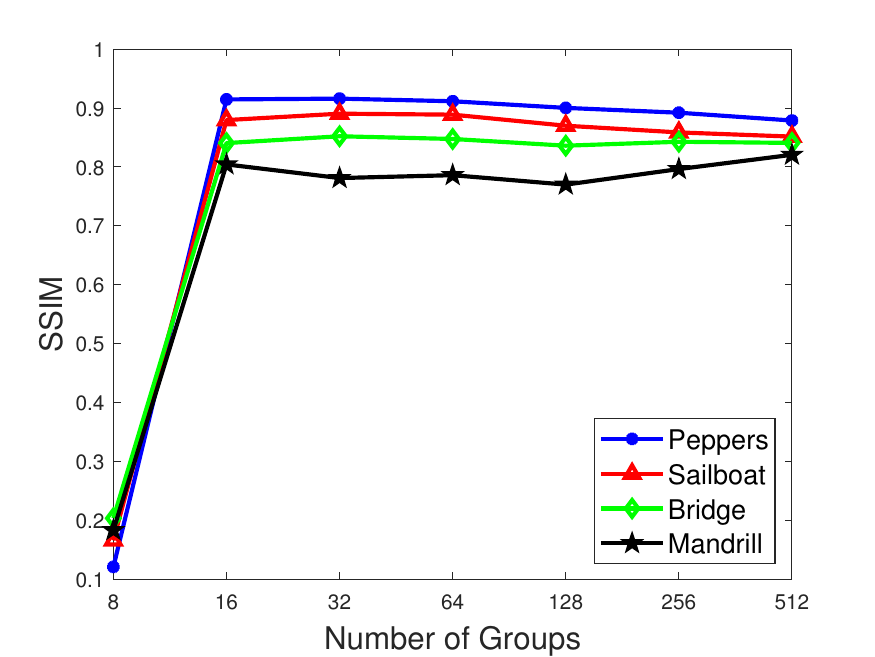}
%			\caption*{SSIM}
		\end{subfigure}
		\begin{subfigure}[b]{0.333\linewidth}
			\centering
			\includegraphics[width=\linewidth]{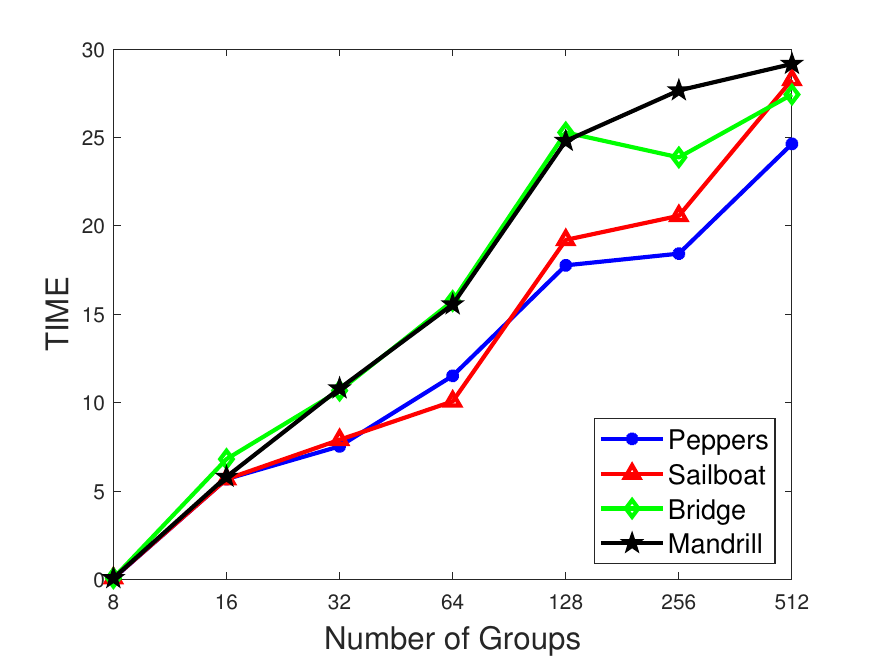}
%			\caption*{TIME}
		\end{subfigure}
	\end{subfigure}
	\vfill
	\caption{The PSNR, SSIM and running time with different number of groups.}
	\label{fig:group}
\end{figure*}

\subsubsection{Effects of the restarted technique}
In this subsection, we compare the effects of restarted technique on our proposed algorithm FLGSR. The bar charts of PSNR, SSIM and running time with respect to different iteration methods for $S$ are shown in Figure \ref{fig:Restarted}. From the PSNR and SSIM metrics, which measure the recovery effect, we can see that using the restarted technique on the Lagrange multiplier matrix ($S$) in our proposed algorithm FLGSR can enhance the recovery quality; from the running time, we can see that using the restarted technique on $S$ can significantly reduce the computational time, which is about four times faster on average than not using the restarted technique on $S$.

\begin{figure*}[htbp]
	\centering
	\begin{subfigure}[b]{1\linewidth}
		\centering
		\includegraphics[width=\linewidth]{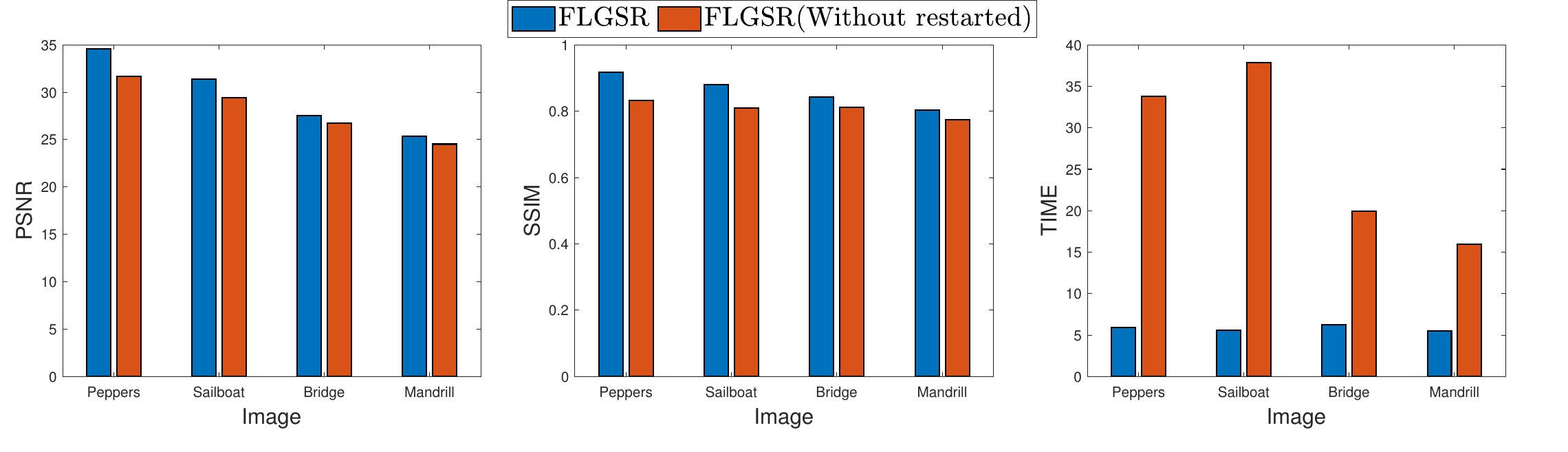}
	\end{subfigure}
	\vfill
	\caption{The PSNR, SSIM and running time with different iteration methods for $S$.}
	\label{fig:Restarted}
\end{figure*}

\subsection{Grayscale image inpainting}
In this subsection, we evaluate all the methods on the USC-SIPI Image Database\footnote{https://sipi.usc.edu/database/database.php?volume=misc.}. For testing, we randomly select 20 images of size $512 \times 512$ pixels from this database. We set the sampling rate $SR$ to 70\%.

We display the inpainting results of the four testing images (``Peppers", ``Sailboat", ``Bridge" and ``Mandrill") in Figure \ref{fig:grayscaleimage}. Enlarged views of parts of the recovered images clearly show the recovery differences. FLGSR recovers the details much better and preserves the surface of peppers, the sail on boat, the tree branch by bridge and the face of mandrill well. It can be seen that FLGSR is superior to other methods.

Table \ref{tab:image} lists the recovery PSNR, SSIM and the corresponding running times of different methods. The highest PSNR and SSIM results are shown in bold.
It can be seen from the table that FLGSR outperforms other methods on all metrics. FLGSR, FGSR and GUIG, which are based on group sparsity, have higher PSNR and SSIM values than NMFC, which is based on matrix factorization, and SVT, which is based on nuclear norm. 
In terms of time consumption, thanks to the idea of flexible grouping, FLGSR consumes much less time than the 1-column-based methods FGSR and GUIG.

In Figure \ref{fig:Box}, we report the PSNR, SSIM and the algorithm running time of different methods on the 20 images. Our method achieves the best performance with an average improvement of 1.9 dB in PSNR and 0.07 in SSIM over the respective second best methods on each image, further verifying its advantages and robustness. In terms of time consumption, our method FLGSR is much faster than other compared methods based on group sparsity. In conclusion, it not only achieves the best inpainting results but also runs very fast.

\begin{table}[htbp]
	\centering
	\caption{Grayscale image inpainting performance comparison: PSNR, SSIM and running time}
	\begin{tabular}{ccccccc}
		\toprule
		Image & Index & FLGSR & FGSR  & GUIG  & NMF   & SVT \\
		\midrule
		\multirow{3}{*}{Peppers} & PSNR  & \textbf{34.391 } & 31.019  & 31.816  & 27.696  & 27.883  \\
		& SSIM  & \textbf{0.905 } & 0.777  & 0.821  & 0.748  & 0.786  \\
		& TIME  & 4.900  & 8.886  & 10.285  & 1.101  & 4.455  \\
		\midrule
		\multirow{3}{*}{Sailboat} & PSNR  & \textbf{31.775 } & 29.850  & 29.170  & 24.724  & 26.004  \\
		& SSIM  & \textbf{0.895 } & 0.811  & 0.781  & 0.673  & 0.776  \\
		& TIME  & 5.295  & 9.598  & 11.335  & 0.339  & 3.773  \\
		\midrule
		\multirow{3}{*}{Bridge} & PSNR  & \textbf{27.455 } & 26.169  & 24.372  & 23.179  & 20.977  \\
		& SSIM  & \textbf{0.840 } & 0.778  & 0.725  & 0.603  & 0.634  \\
		& TIME  & 5.233  & 9.815  & 10.808  & 0.403  & 2.363  \\
		\midrule
		\multirow{3}{*}{Mandrill} & PSNR  & \textbf{25.301 } & 23.864  & 22.620  & 21.574  & 21.395  \\
		& SSIM  & \textbf{0.802 } & 0.717  & 0.647  & 0.523  & 0.657  \\
		& TIME  & 4.554  & 9.776  & 11.312  & 0.331  & 2.345  \\
		\bottomrule
	\end{tabular}%
	\label{tab:image}%
\end{table}%

\begin{figure*}[htbp]
	\centering
	\begin{subfigure}[b]{1\linewidth}
		\begin{subfigure}[b]{0.138\linewidth}
			\centering
			\includegraphics[width=\linewidth]{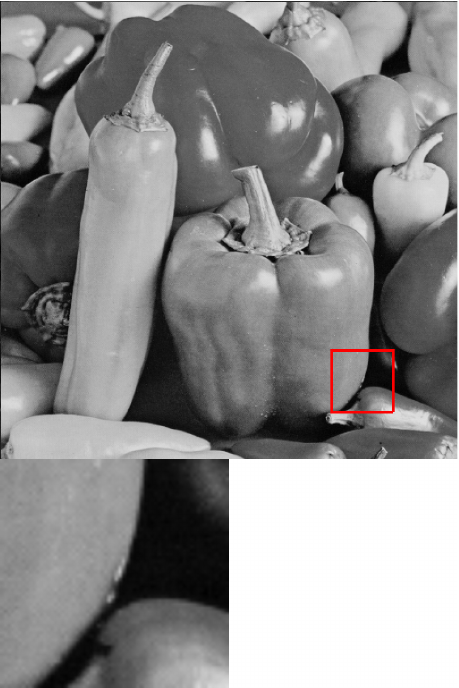}\vspace{0pt}
			\includegraphics[width=\linewidth]{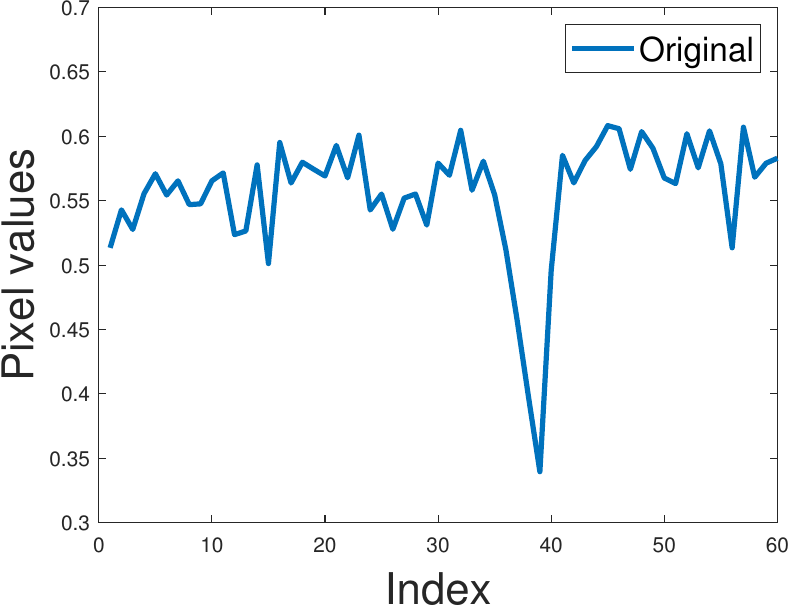}\vspace{0pt}
			\includegraphics[width=\linewidth]{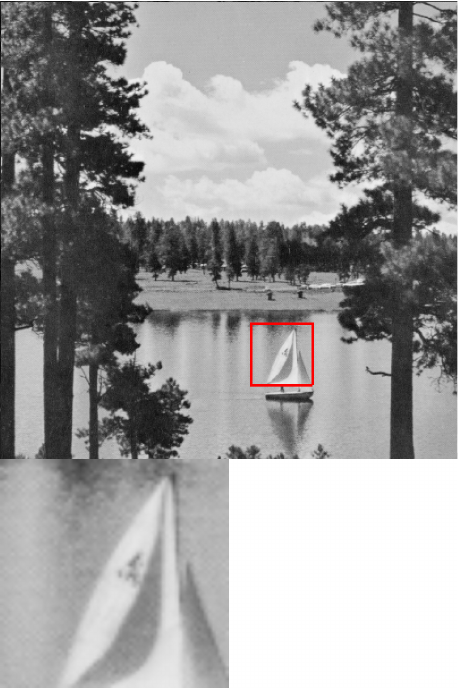}\vspace{0pt}
			\includegraphics[width=\linewidth]{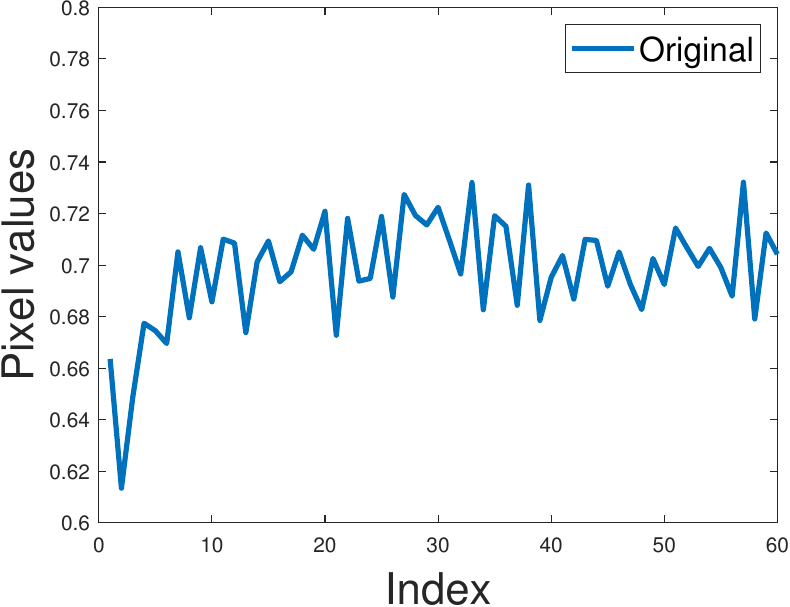}\vspace{0pt}
			\includegraphics[width=\linewidth]{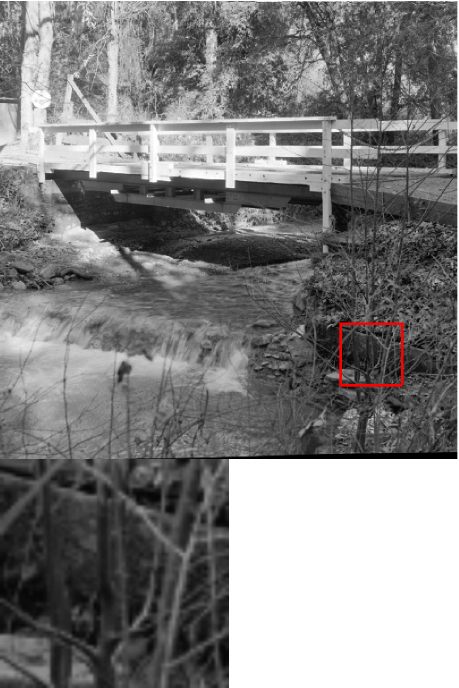}\vspace{0pt}
			\includegraphics[width=\linewidth]{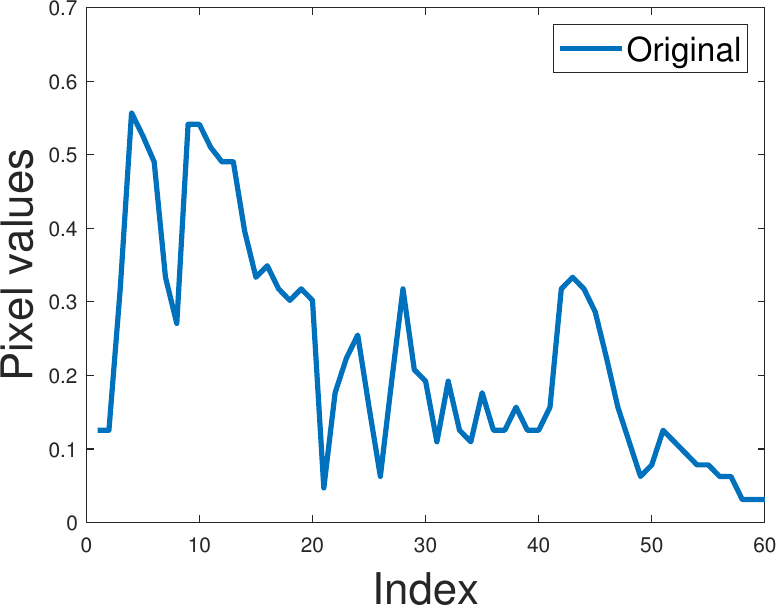}\vspace{0pt}
			\includegraphics[width=\linewidth]{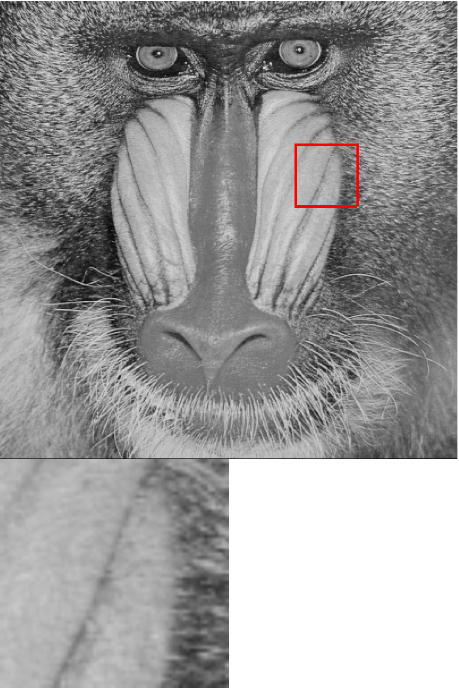}\vspace{0pt}
			\includegraphics[width=\linewidth]{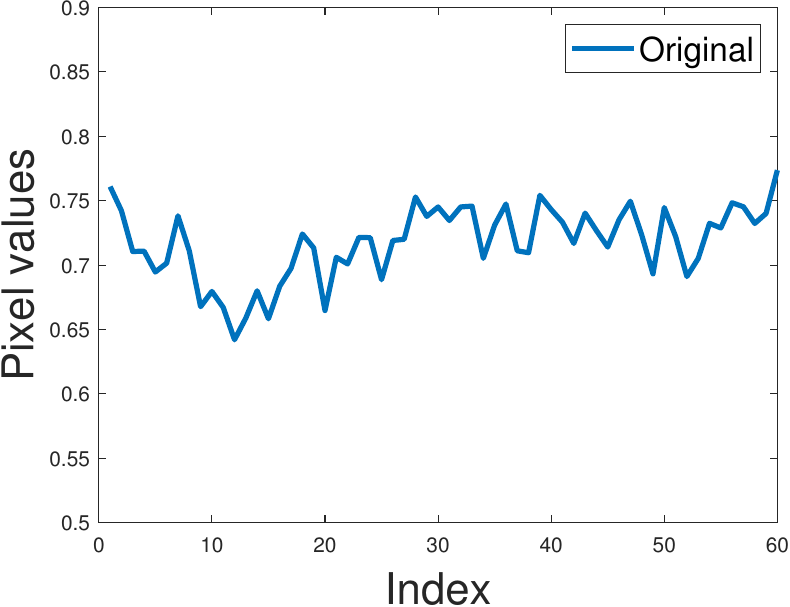}
			\caption*{Original}
		\end{subfigure}   	
		\begin{subfigure}[b]{0.138\linewidth}
			\centering
			\includegraphics[width=\linewidth]{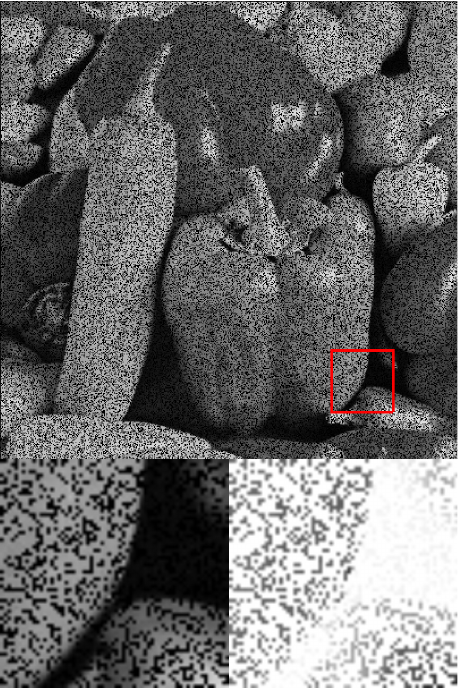}\vspace{0pt}
			\includegraphics[width=\linewidth]{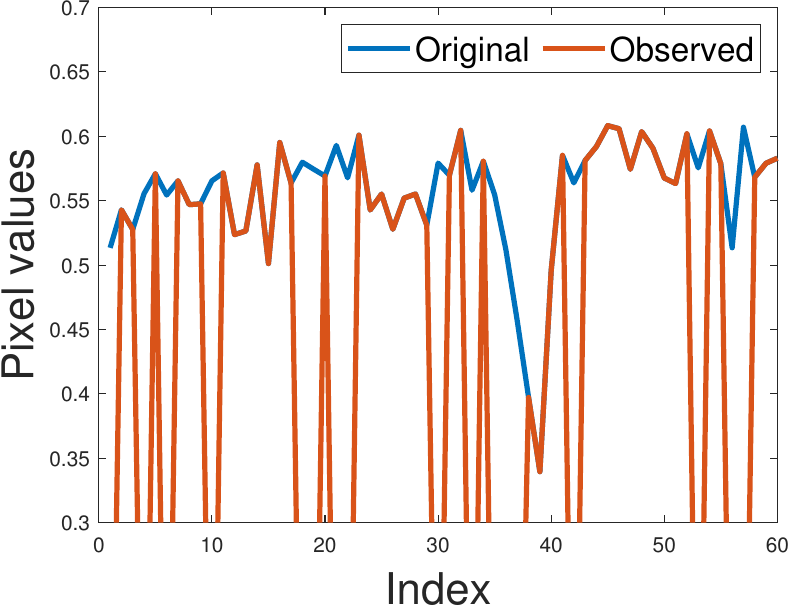}\vspace{0pt}
			\includegraphics[width=\linewidth]{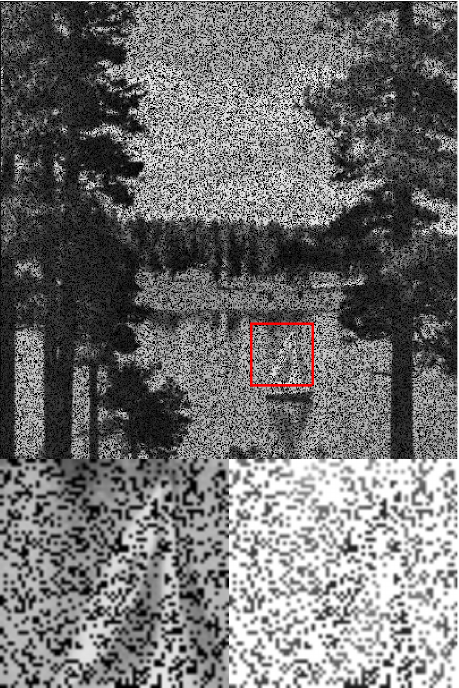}\vspace{0pt}
			\includegraphics[width=\linewidth]{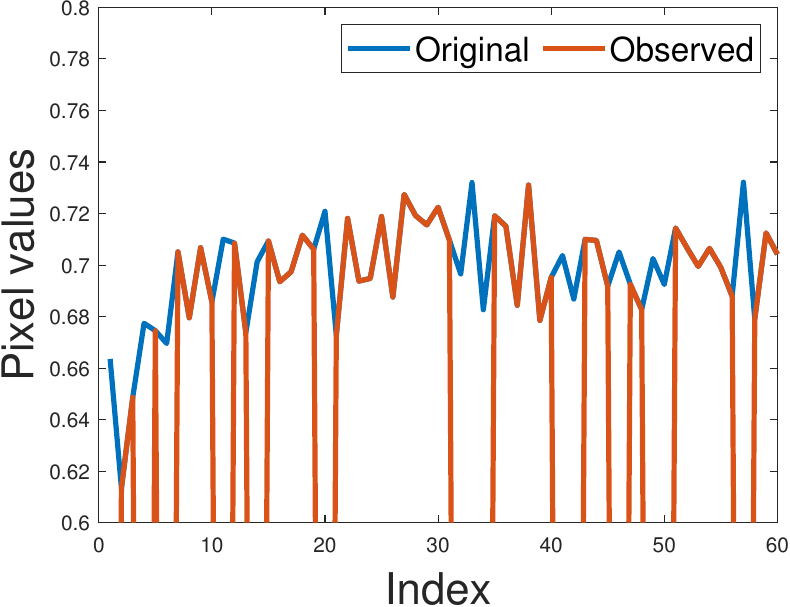}\vspace{0pt}
			\includegraphics[width=\linewidth]{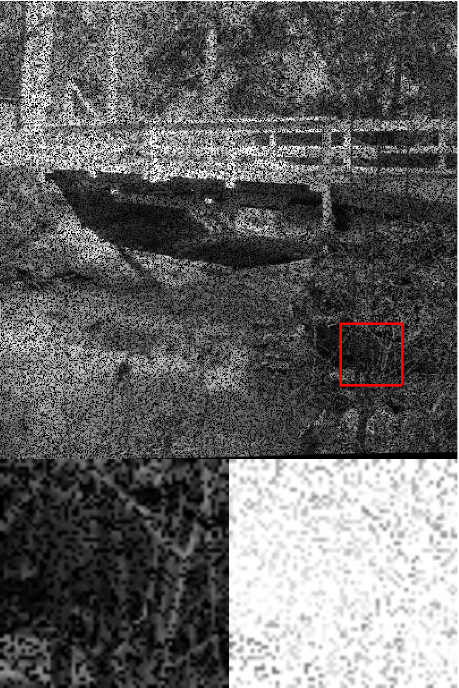}\vspace{0pt}
			\includegraphics[width=\linewidth]{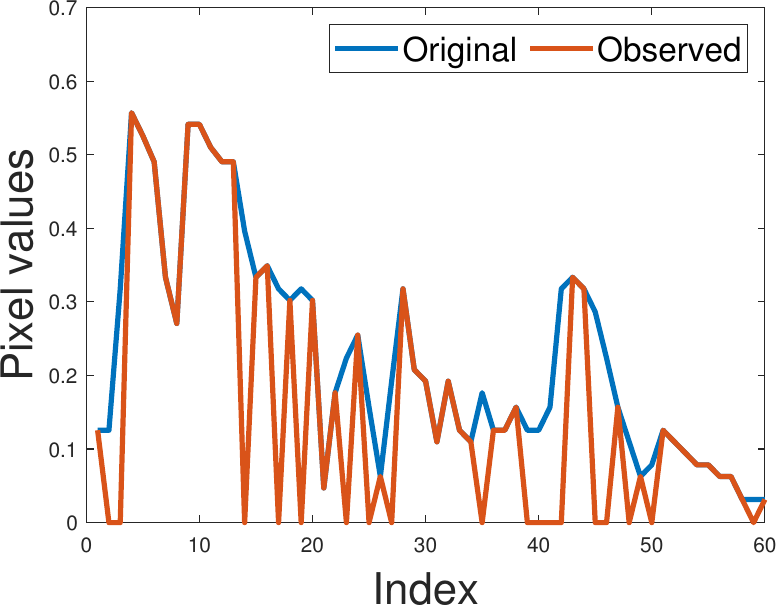}\vspace{0pt}
			\includegraphics[width=\linewidth]{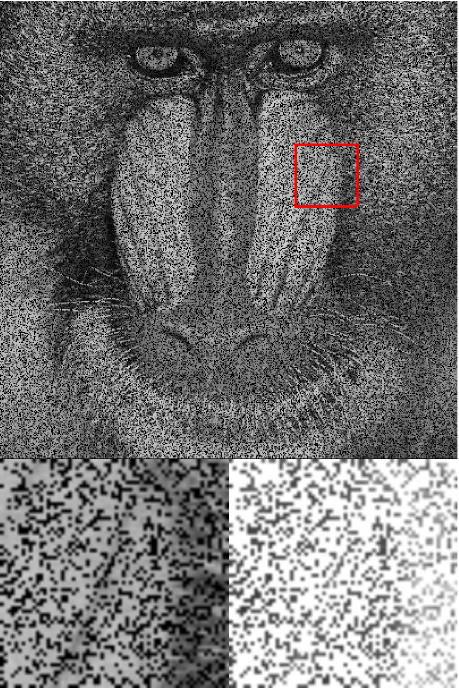}\vspace{0pt}
			\includegraphics[width=\linewidth]{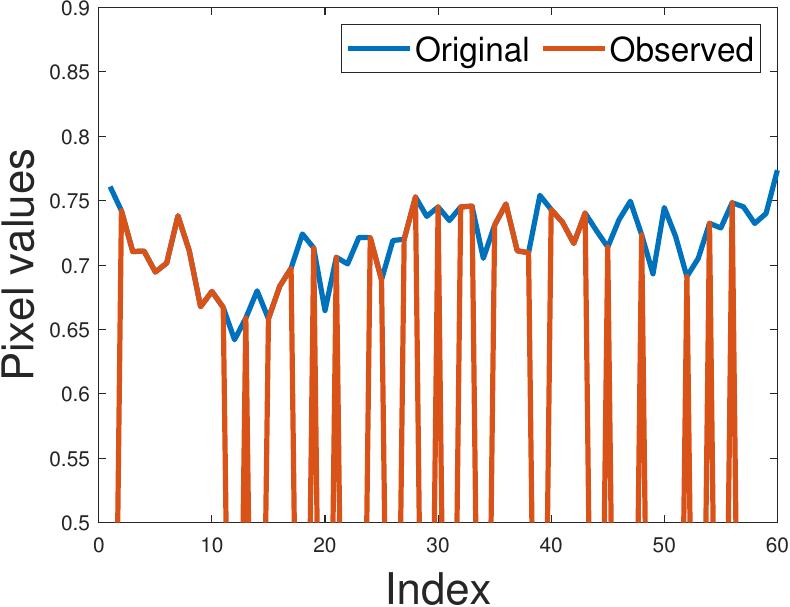}
			\caption*{Observed}
		\end{subfigure}
		\begin{subfigure}[b]{0.138\linewidth}
			\centering
			\includegraphics[width=\linewidth]{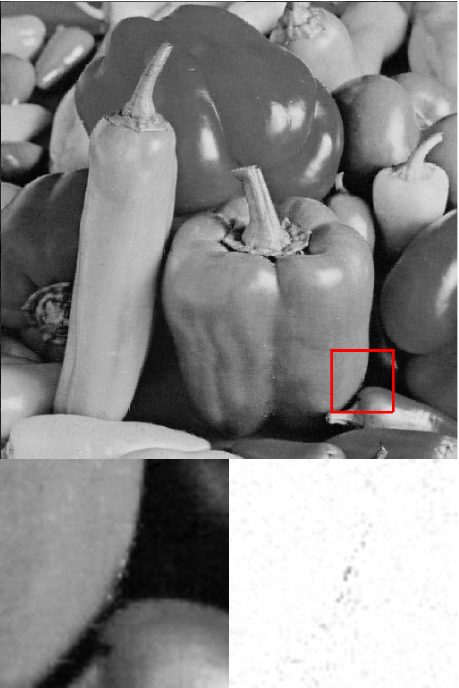}\vspace{0pt}
			\includegraphics[width=\linewidth]{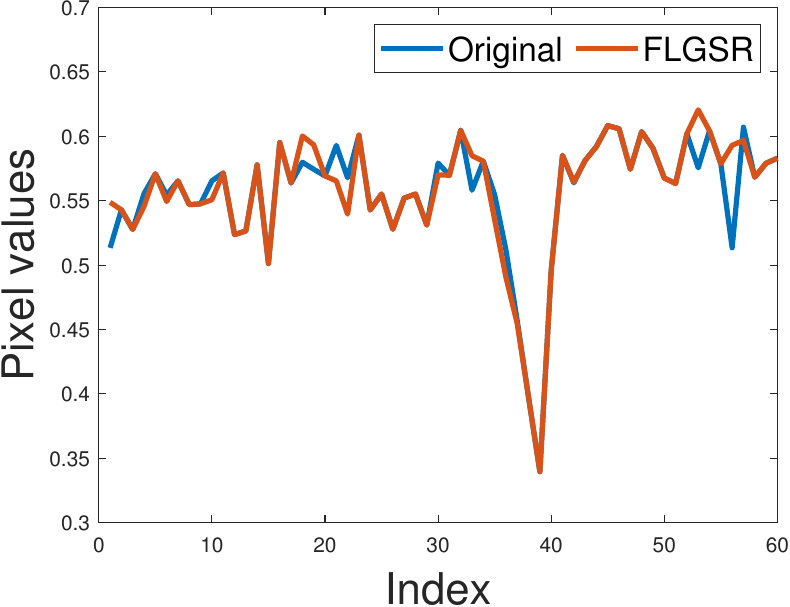}\vspace{0pt}
			\includegraphics[width=\linewidth]{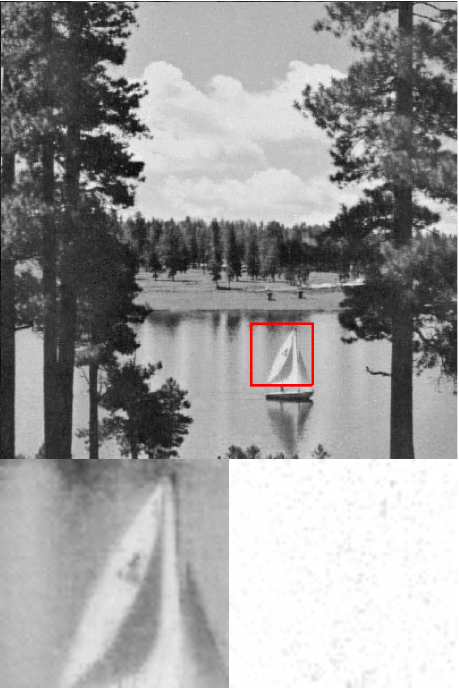}\vspace{0pt}
			\includegraphics[width=\linewidth]{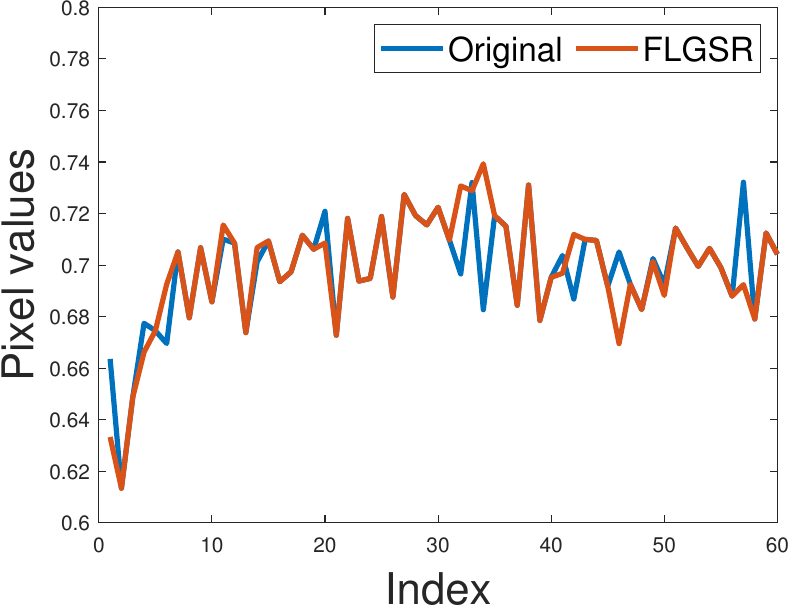}\vspace{0pt}
			\includegraphics[width=\linewidth]{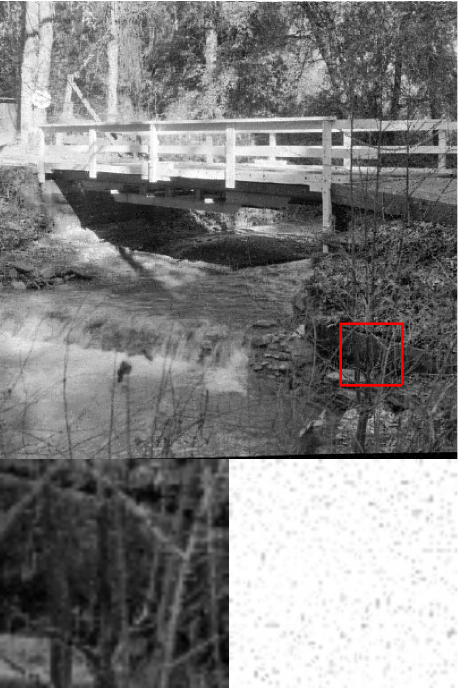}\vspace{0pt}
			\includegraphics[width=\linewidth]{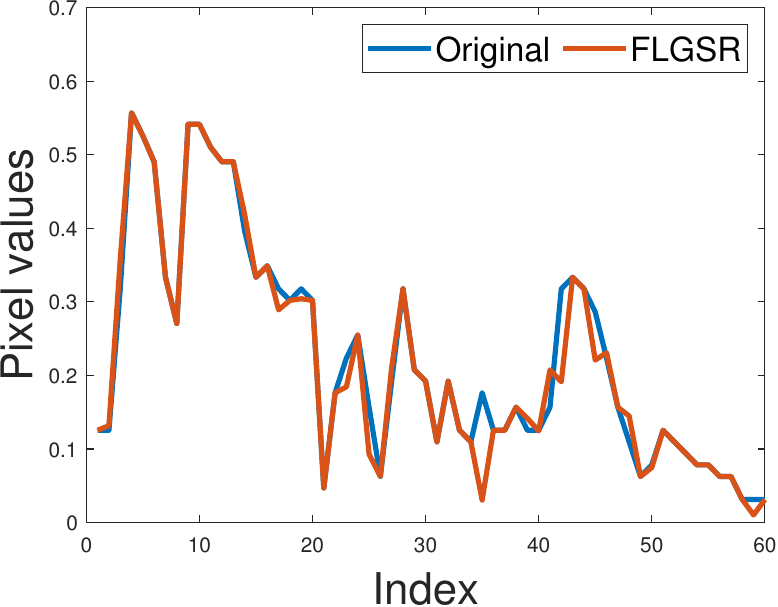}\vspace{0pt}
			\includegraphics[width=\linewidth]{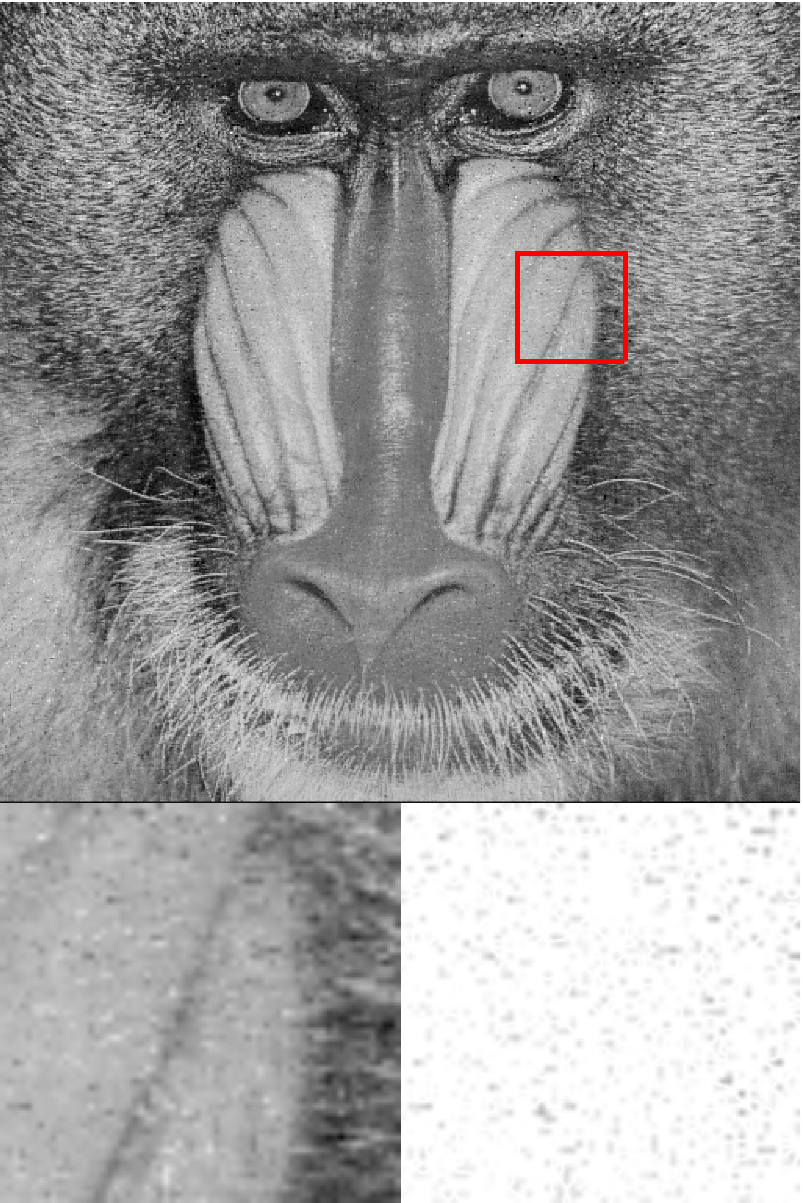}\vspace{0pt}
			\includegraphics[width=\linewidth]{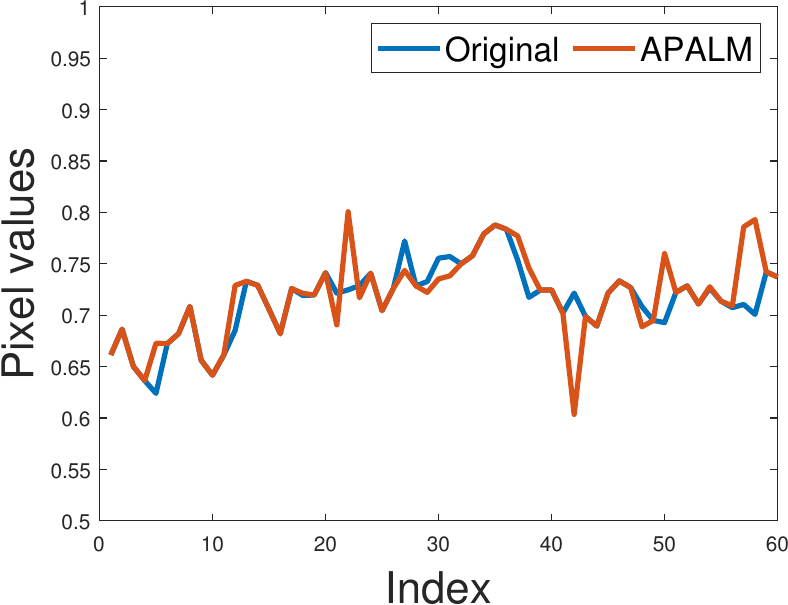}
			\caption*{FLGSR}
		\end{subfigure}
		\begin{subfigure}[b]{0.138\linewidth}
			\centering		
			\includegraphics[width=\linewidth]{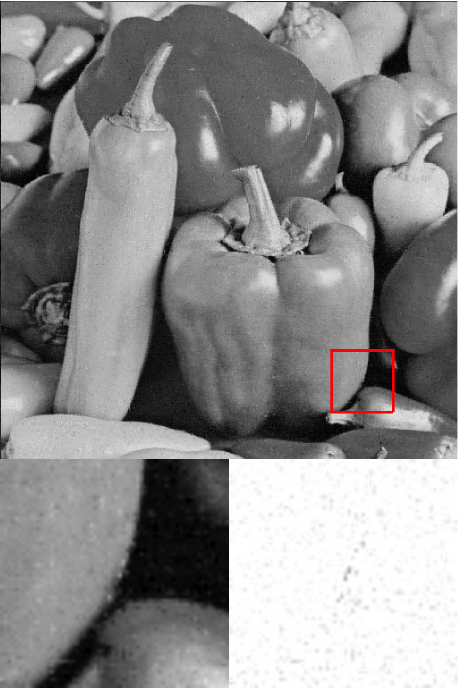}\vspace{0pt}
			\includegraphics[width=\linewidth]{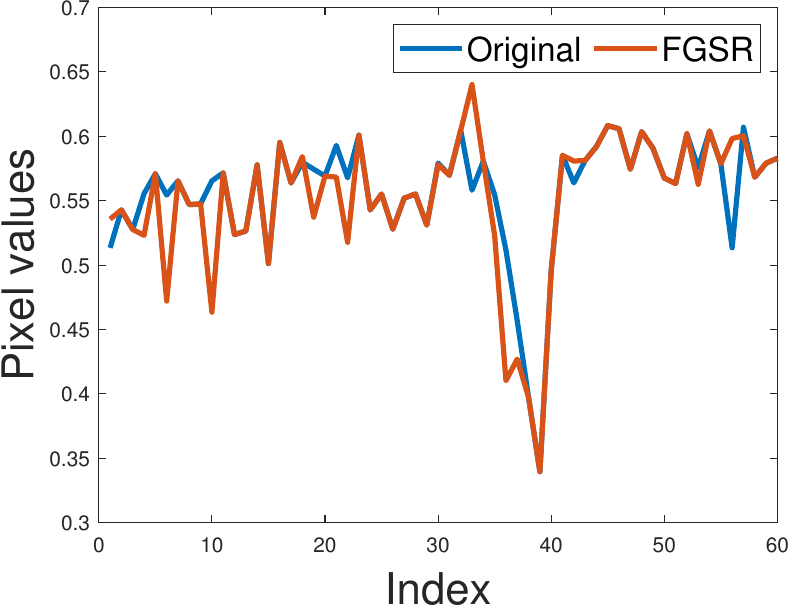}\vspace{0pt}
			\includegraphics[width=\linewidth]{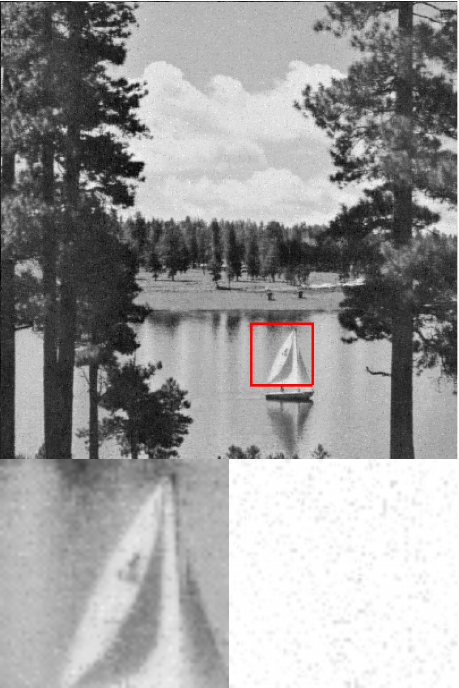}\vspace{0pt}
			\includegraphics[width=\linewidth]{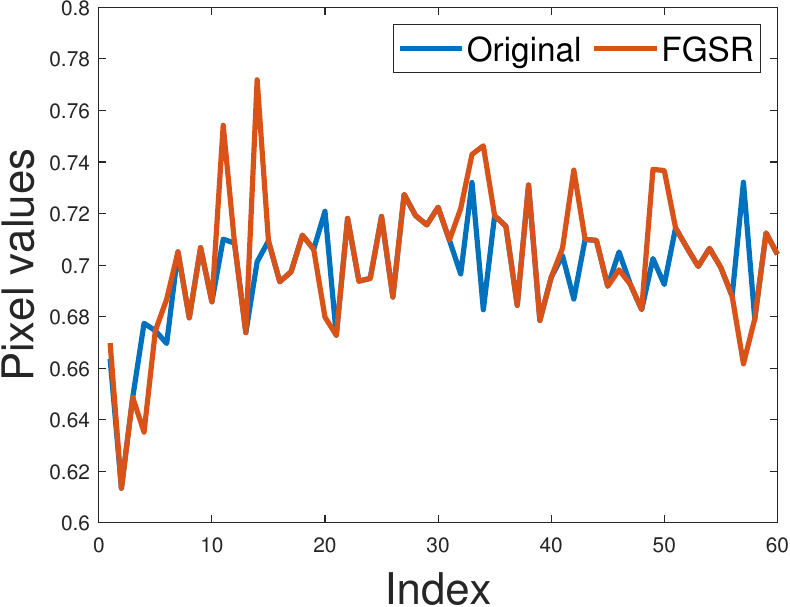}\vspace{0pt}
			\includegraphics[width=\linewidth]{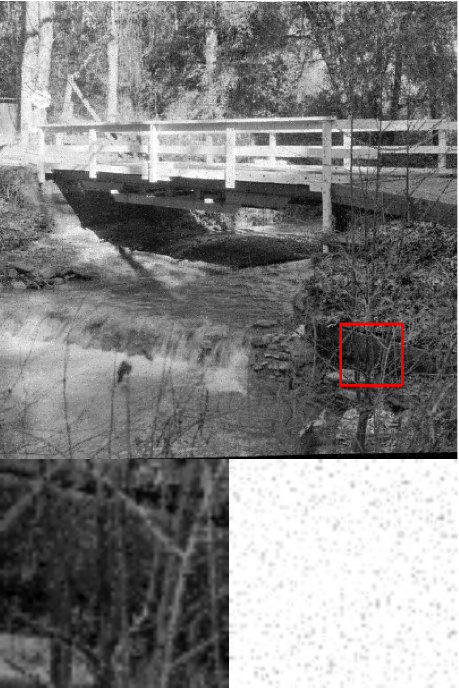}\vspace{0pt}
			\includegraphics[width=\linewidth]{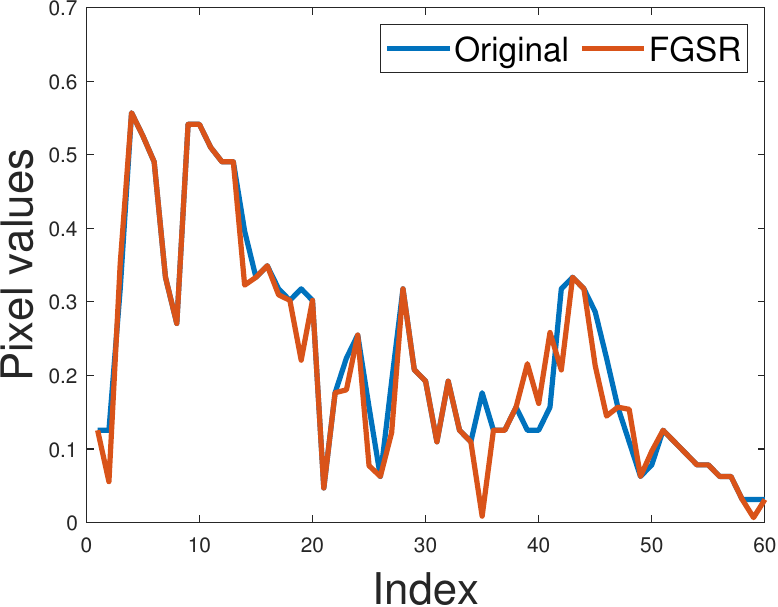}\vspace{0pt}
			\includegraphics[width=\linewidth]{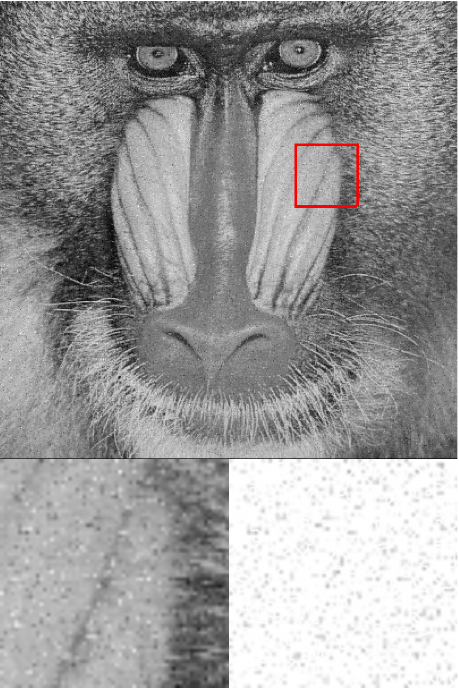}\vspace{0pt}
			\includegraphics[width=\linewidth]{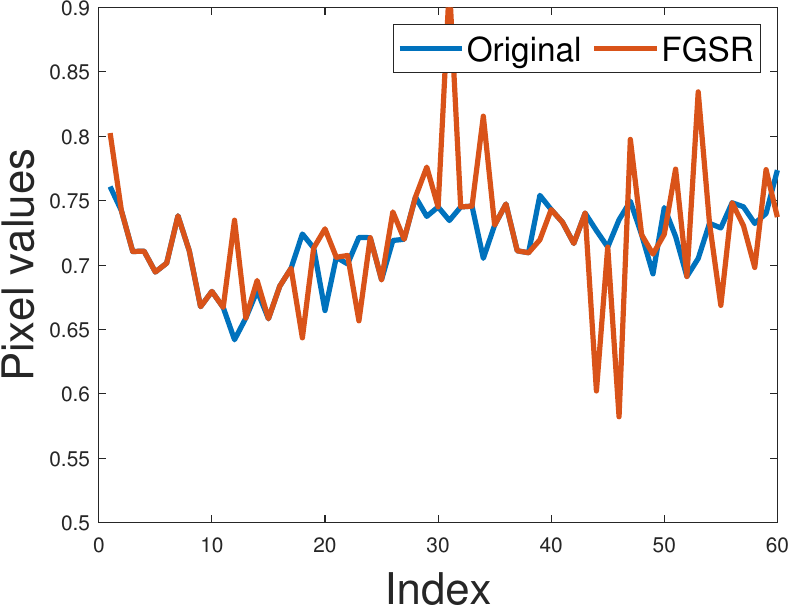}
			\caption*{FGSR}
		\end{subfigure}
		\begin{subfigure}[b]{0.138\linewidth}
			\centering
			\includegraphics[width=\linewidth]{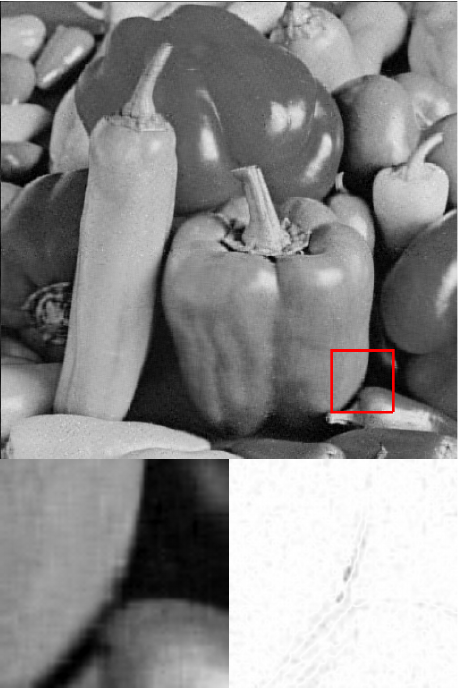}\vspace{0pt}
			\includegraphics[width=\linewidth]{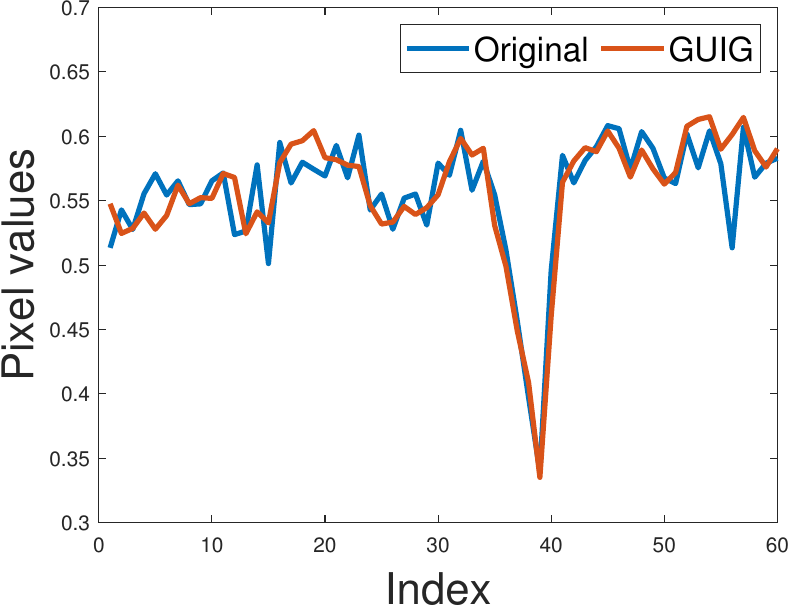}\vspace{0pt}
			\includegraphics[width=\linewidth]{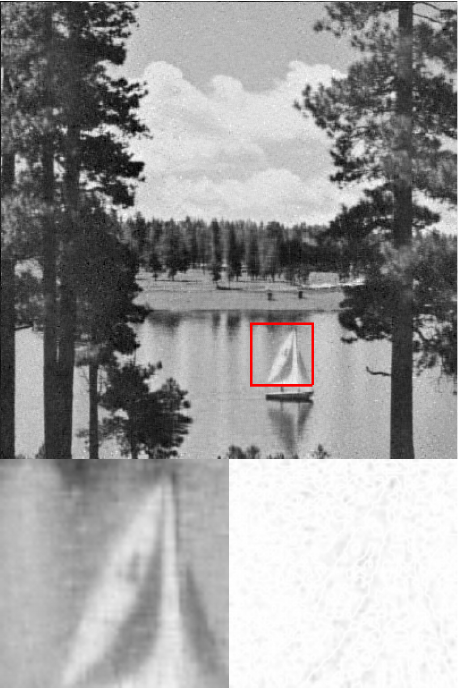}\vspace{0pt}
			\includegraphics[width=\linewidth]{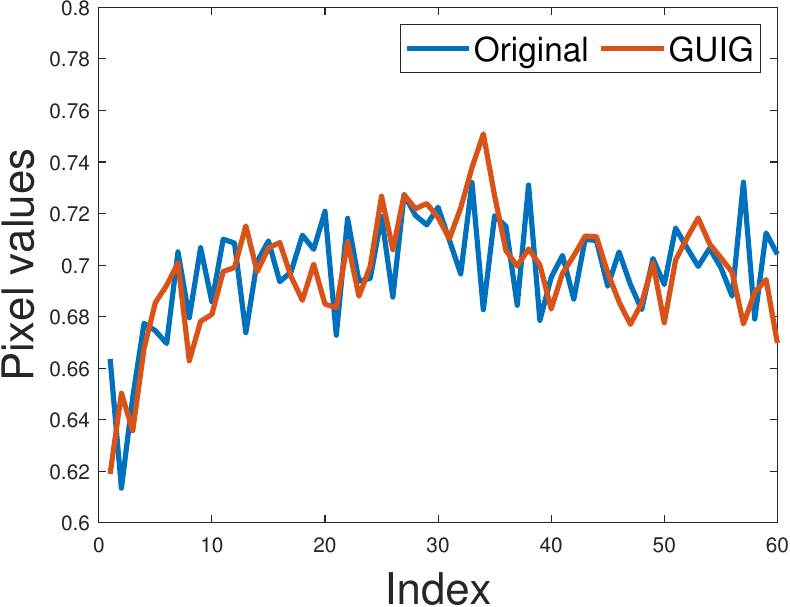}\vspace{0pt}
			\includegraphics[width=\linewidth]{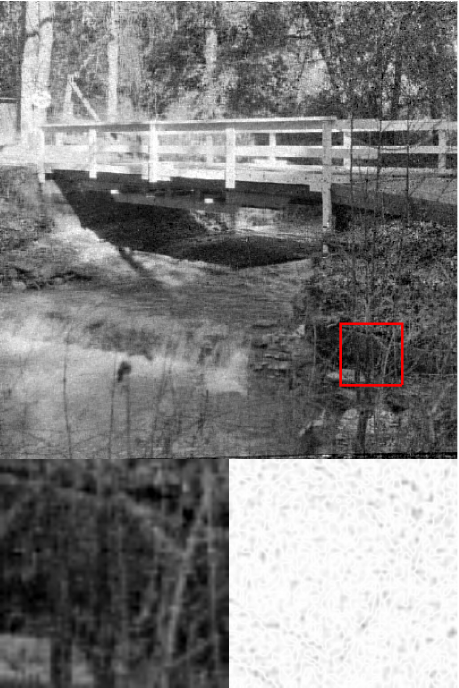}\vspace{0pt}
			\includegraphics[width=\linewidth]{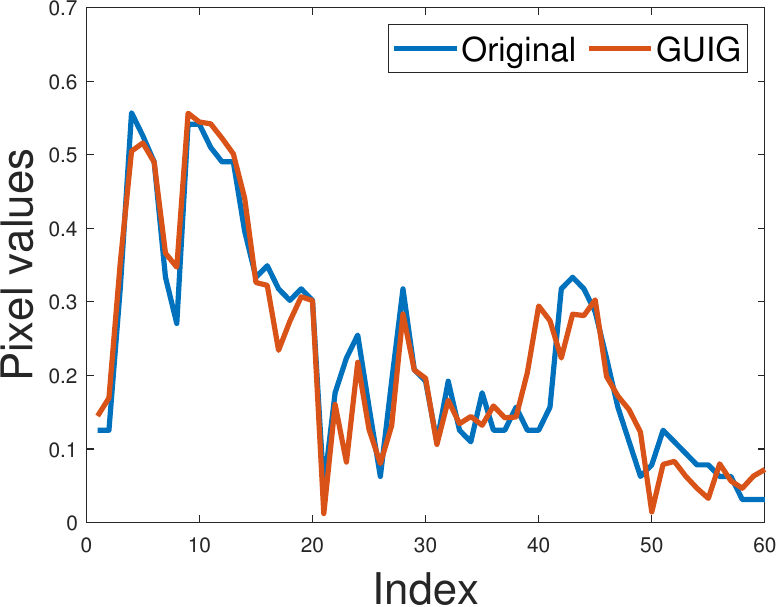}\vspace{0pt}
			\includegraphics[width=\linewidth]{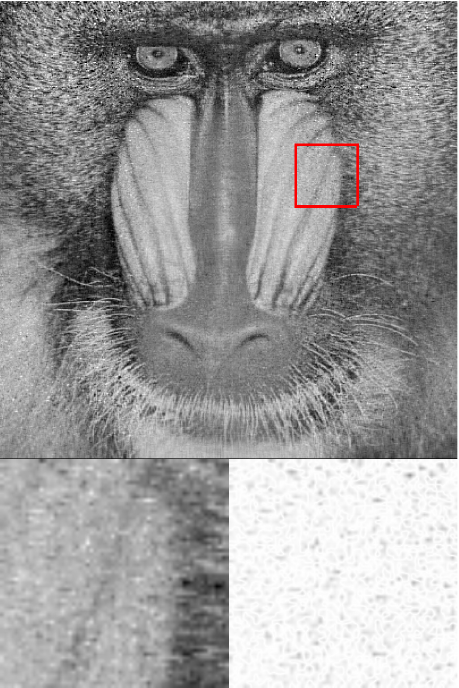}\vspace{0pt}
			\includegraphics[width=\linewidth]{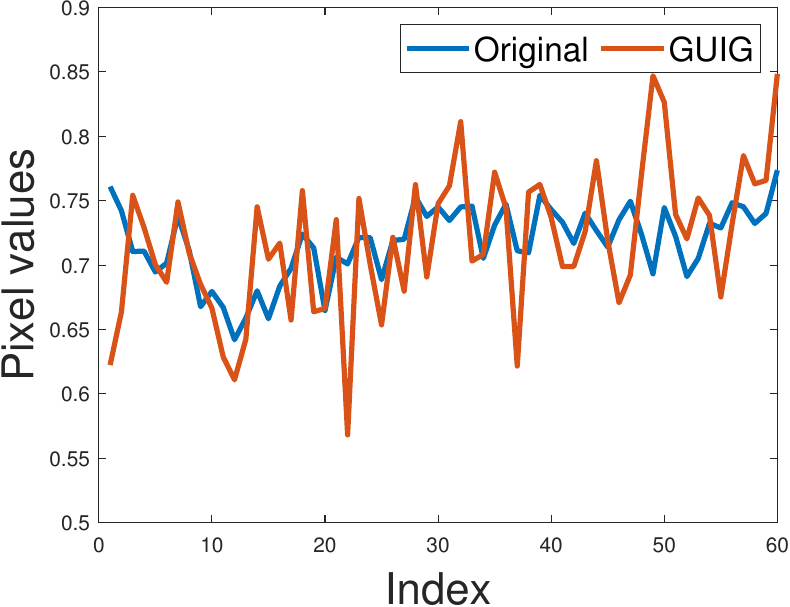}
			\caption*{GUIG}
		\end{subfigure}	
		\begin{subfigure}[b]{0.138\linewidth}
			\centering			
			\includegraphics[width=\linewidth]{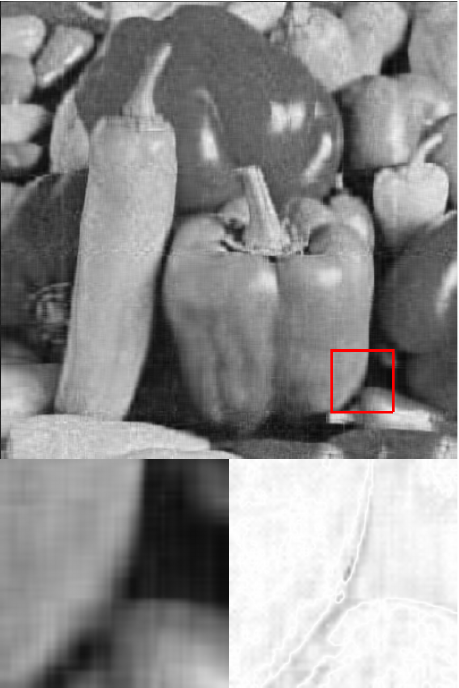}\vspace{0pt}
			\includegraphics[width=\linewidth]{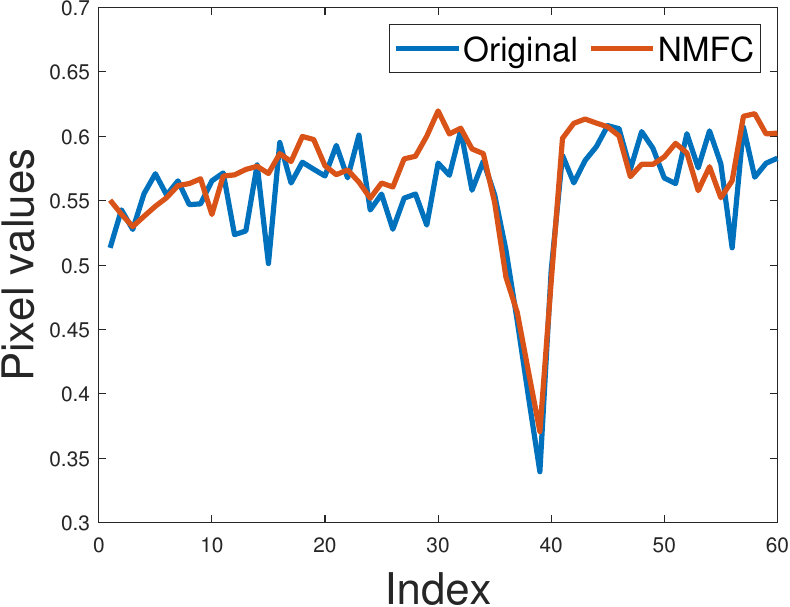}\vspace{0pt}
			\includegraphics[width=\linewidth]{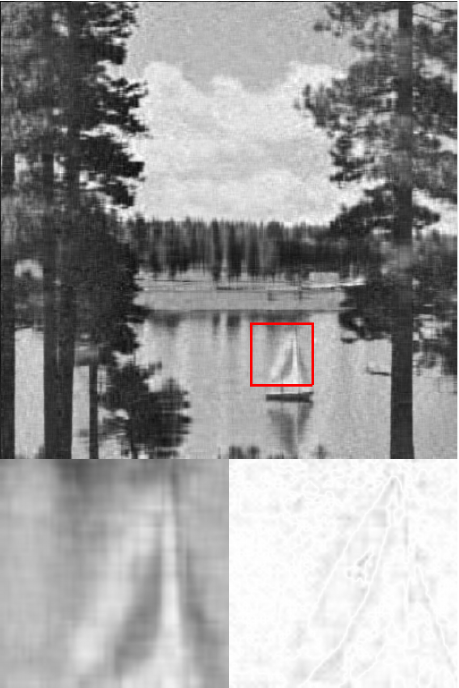}\vspace{0pt}
			\includegraphics[width=\linewidth]{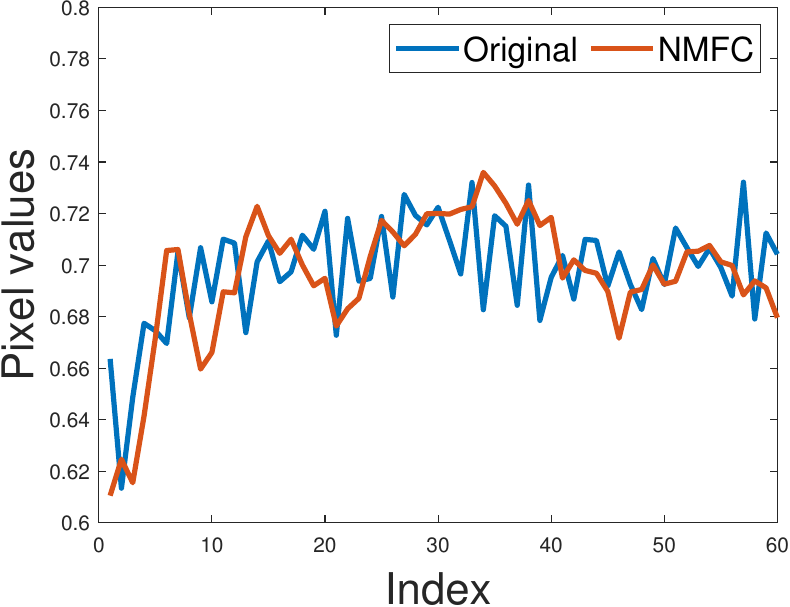}\vspace{0pt}
			\includegraphics[width=\linewidth]{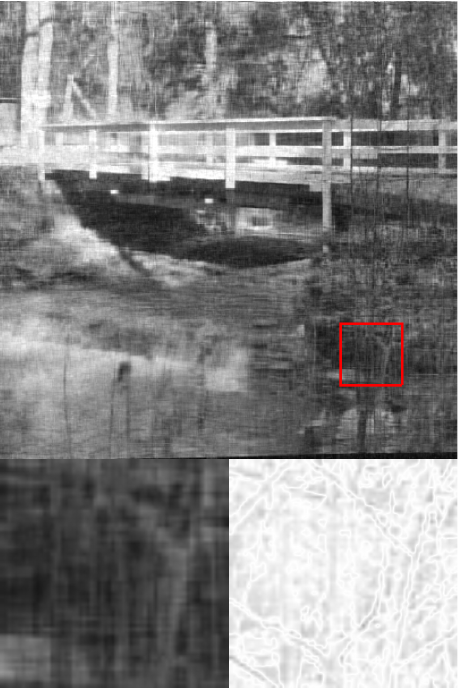}\vspace{0pt}
			\includegraphics[width=\linewidth]{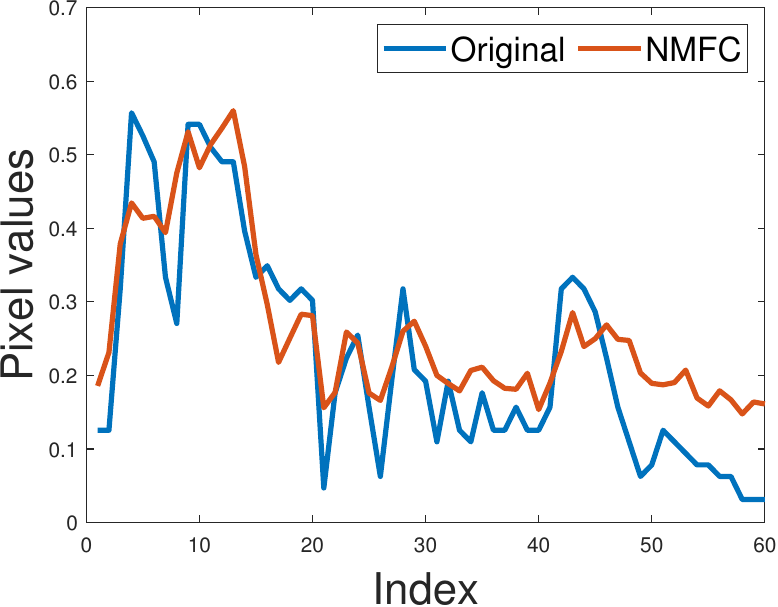}\vspace{0pt}
			\includegraphics[width=\linewidth]{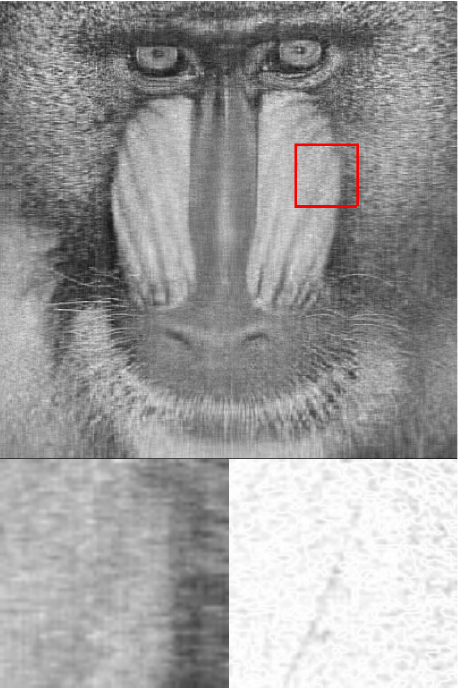}\vspace{0pt}
			\includegraphics[width=\linewidth]{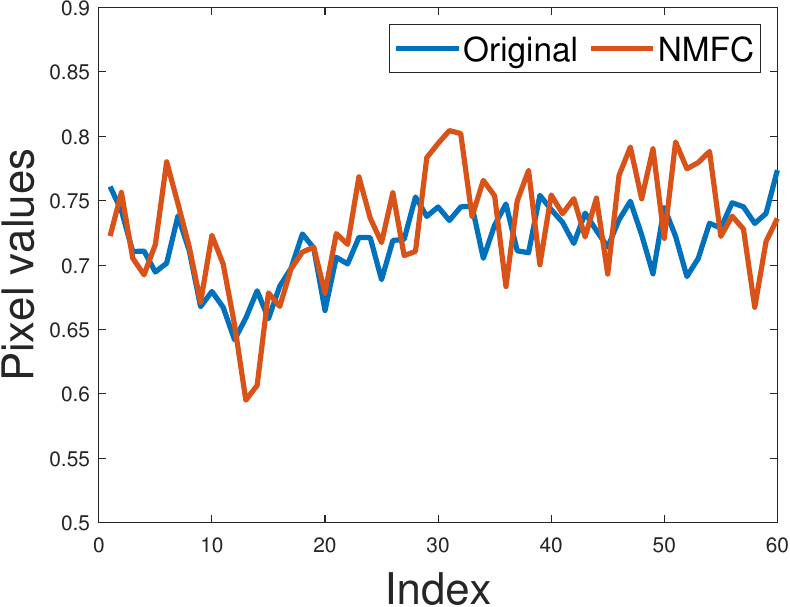}
			\caption*{NMFC}
		\end{subfigure}
			\begin{subfigure}[b]{0.138\linewidth}
		\centering			
		\includegraphics[width=\linewidth]{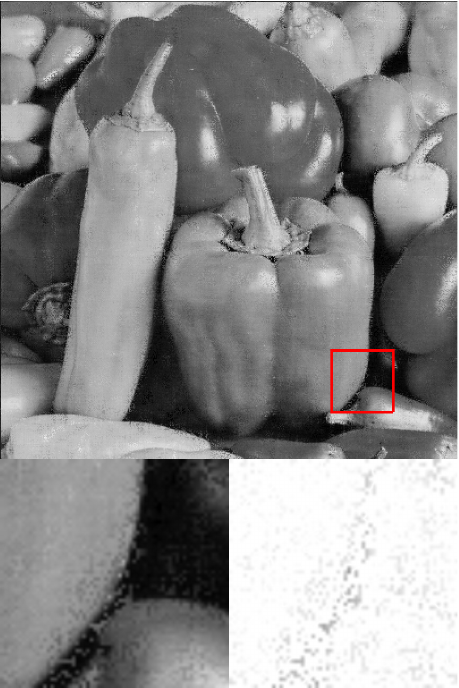}\vspace{0pt}
		\includegraphics[width=\linewidth]{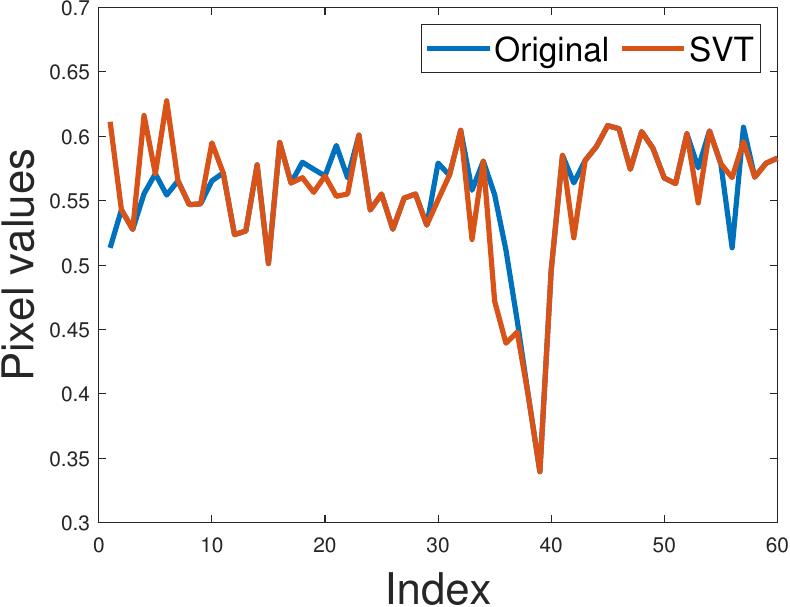}\vspace{0pt}
		\includegraphics[width=\linewidth]{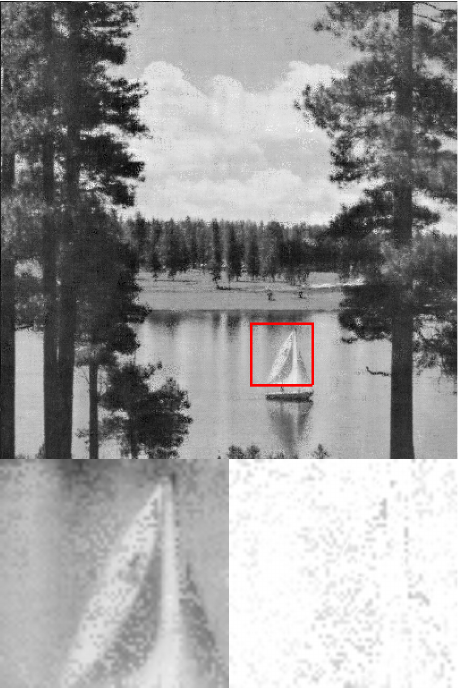}\vspace{0pt}
		\includegraphics[width=\linewidth]{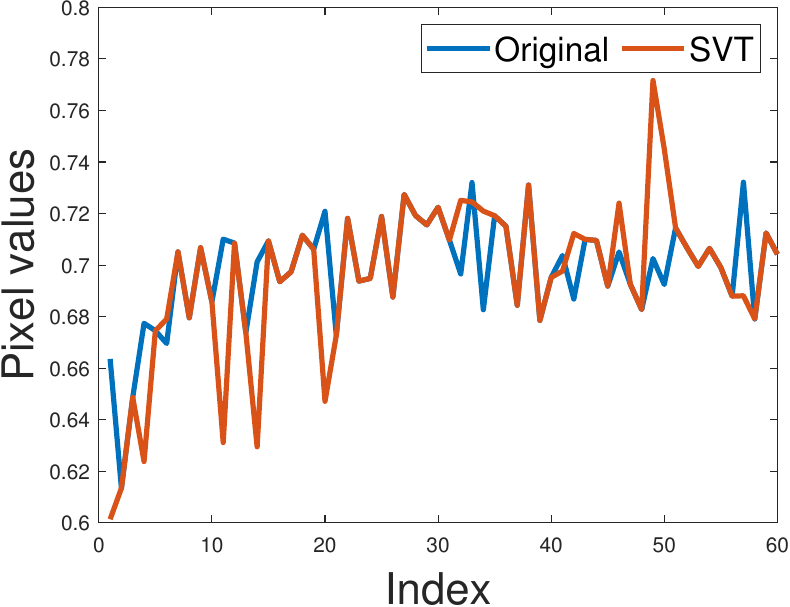}\vspace{0pt}
		\includegraphics[width=\linewidth]{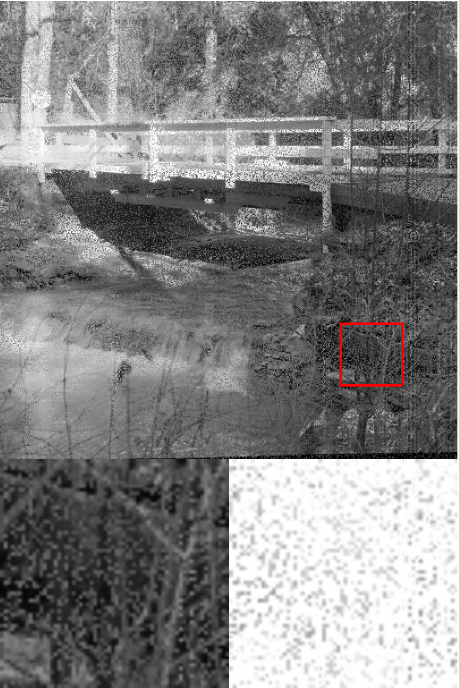}\vspace{0pt}
		\includegraphics[width=\linewidth]{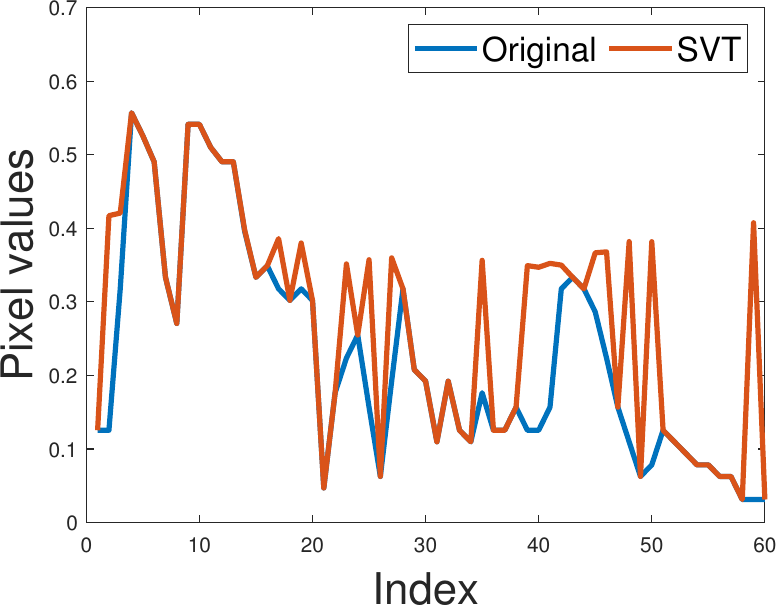}\vspace{0pt}
		\includegraphics[width=\linewidth]{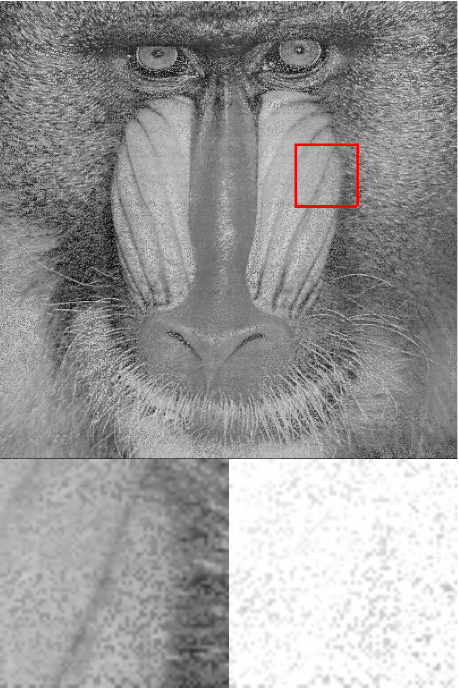}\vspace{0pt}
		\includegraphics[width=\linewidth]{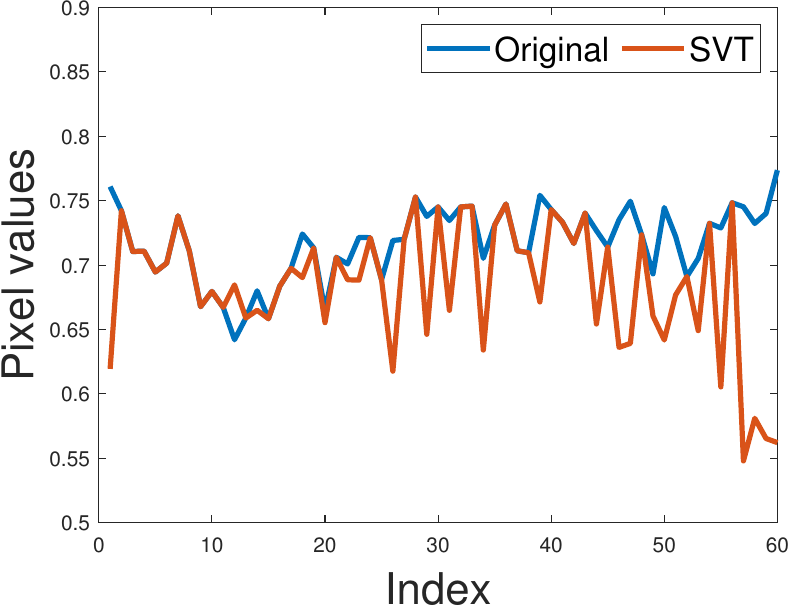}
		\caption*{SVT}
	\end{subfigure}						
	\end{subfigure}
	\vfill
	\caption{Examples of grayscale image inpainting. From top to bottom are respectively corresponding to ``Peppers", ``Sailboat", ``Bridge" and ``Mandrill". For better visualization, we show the zoom-in region and the corresponding partial residuals of the region. Under each image, we show enlargements of a demarcated patch and the corresponding error map (difference from the Original). Error maps with less color information indicate better restoration performance.}
	\label{fig:grayscaleimage}
\end{figure*}

\begin{figure*}[htbp]
	\centering
	\begin{subfigure}[b]{1\linewidth}
		\centering
		\includegraphics[width=\linewidth]{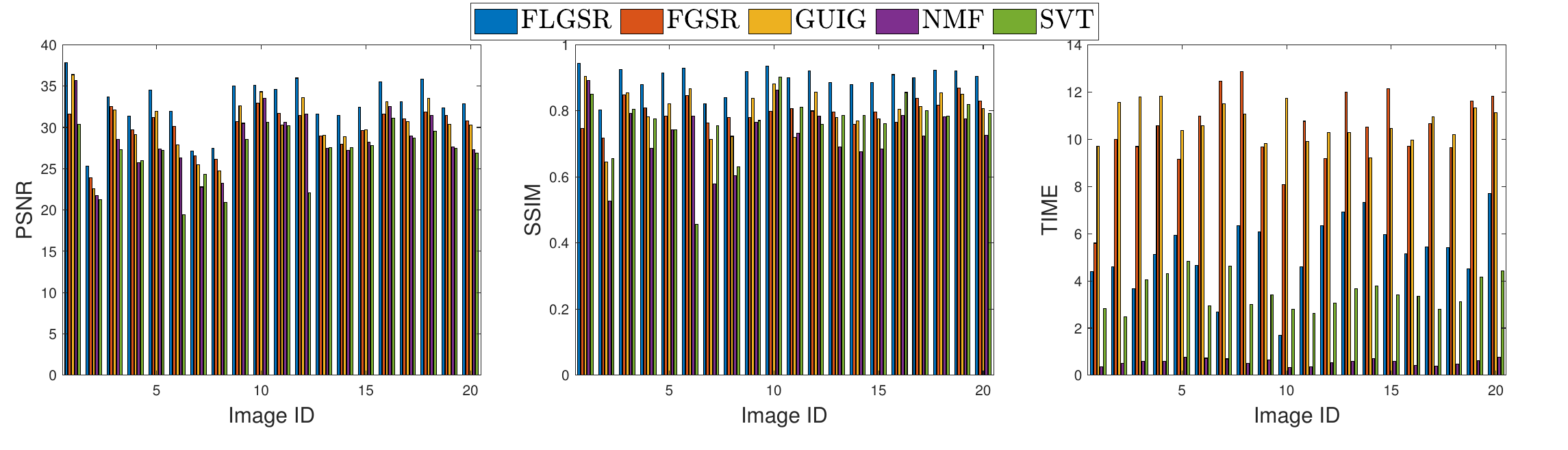}
	\end{subfigure}
	\vfill
	\caption{Comparison of the PSNR, SSIM and the running time on the randomly selected 20 images.}
	\label{fig:Box}
\end{figure*}

\subsection{High altitude aerial image inpainting}
In this subsection, we test high altitude aerial (HAA) image set of size $ 512 \times 512 $ pixels. The sampling rate $ SR $ is set to $ 70\% $. Table \ref{fig:HAAimage} summarizes the PSNR, SSIM values and the corresponding running times of the compared algorithms. The highest PSNR and SSIM results are shown in bold. As observed, FLGSR consistently achieves the highest values in terms of all evaluation indexes, e.g., the proposed method achieves an approximately 1.66 dB gain in PSNR over the respective second best methods on each image. Figure \ref{fig:HAAimage} shows a visualized comparison of the recovery images. As can be seen, FLGSR, GUIG and FGSR, which all rely on group sparsity, generate the best visual results. In addition, it can be seen that NMF and SVT still contain a certain amount of noise. The high altitude aerial image inpainting results are also consistent with the grayscale image inpainting results and all these demonstrate that our FLGSR results are much better than other methods, both in visual quality and in terms of PSNR, and SSIM.

\begin{figure*}[htbp]
	\centering
	\begin{subfigure}[b]{1\linewidth}
		\begin{subfigure}[b]{0.138\linewidth}
			\centering
			\includegraphics[width=\linewidth]{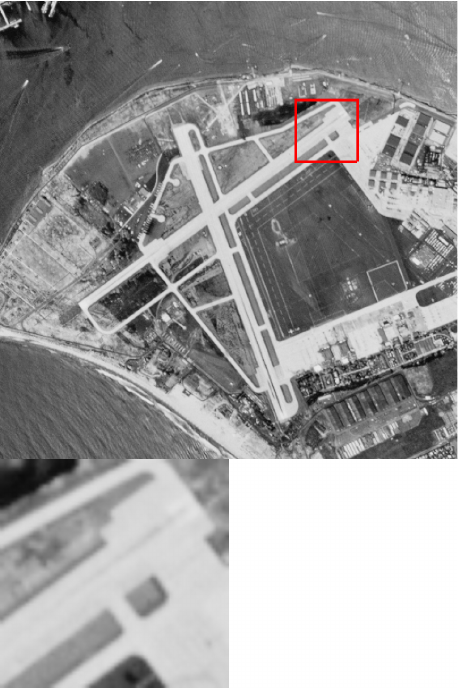}\vspace{0pt}
			\includegraphics[width=\linewidth]{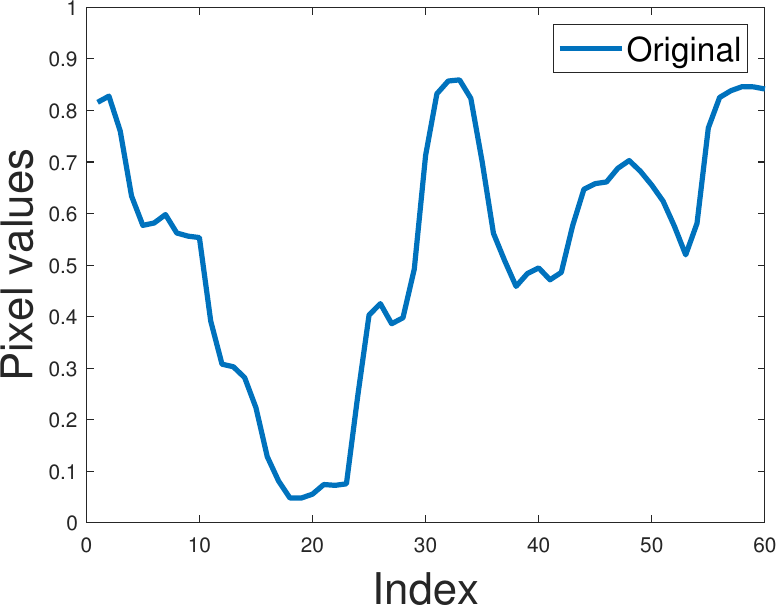}\vspace{0pt}
			\includegraphics[width=\linewidth]{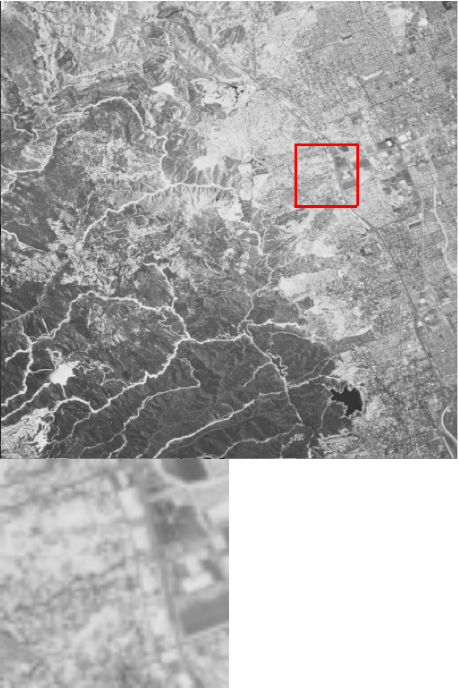}\vspace{0pt}
			\includegraphics[width=\linewidth]{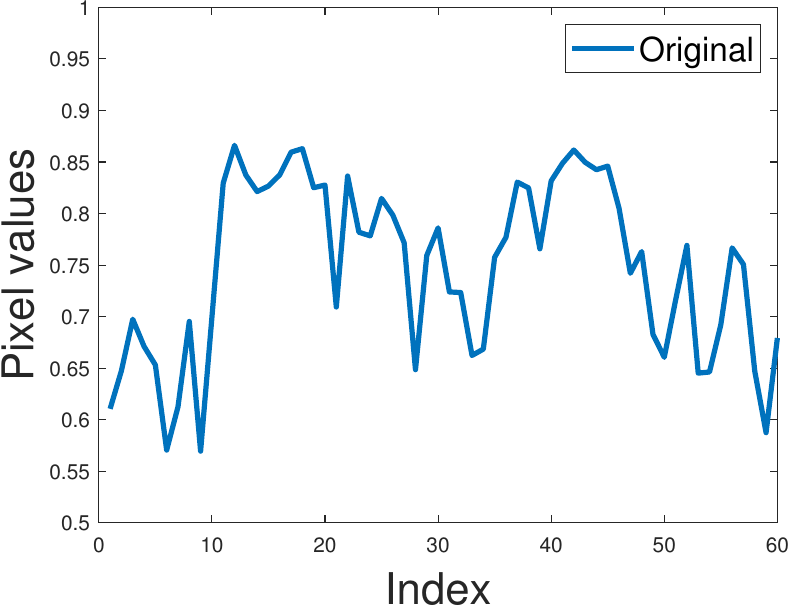}
			\caption*{Original}
		\end{subfigure}   	
		\begin{subfigure}[b]{0.138\linewidth}
			\centering
			\includegraphics[width=\linewidth]{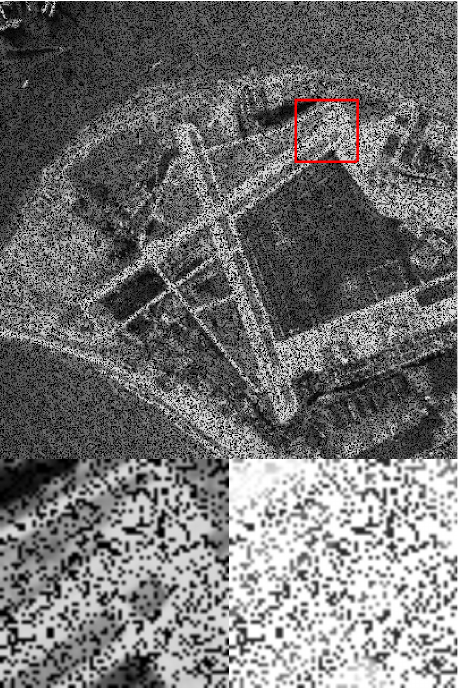}\vspace{0pt}
			\includegraphics[width=\linewidth]{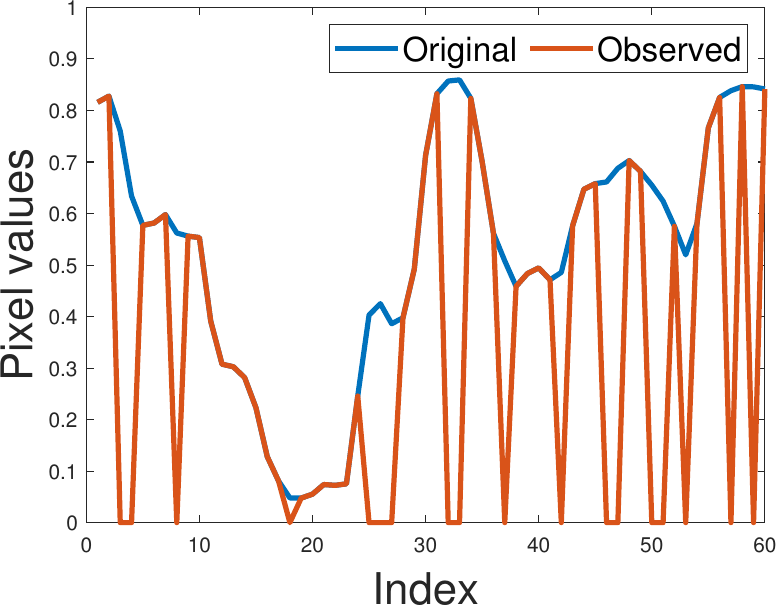}\vspace{0pt}
			\includegraphics[width=\linewidth]{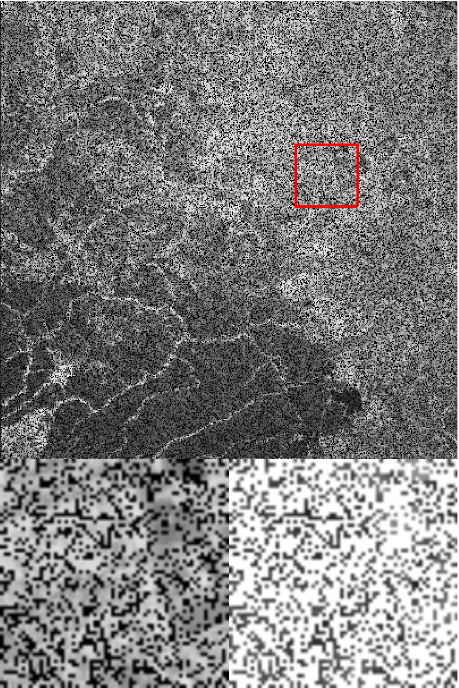}\vspace{0pt}
			\includegraphics[width=\linewidth]{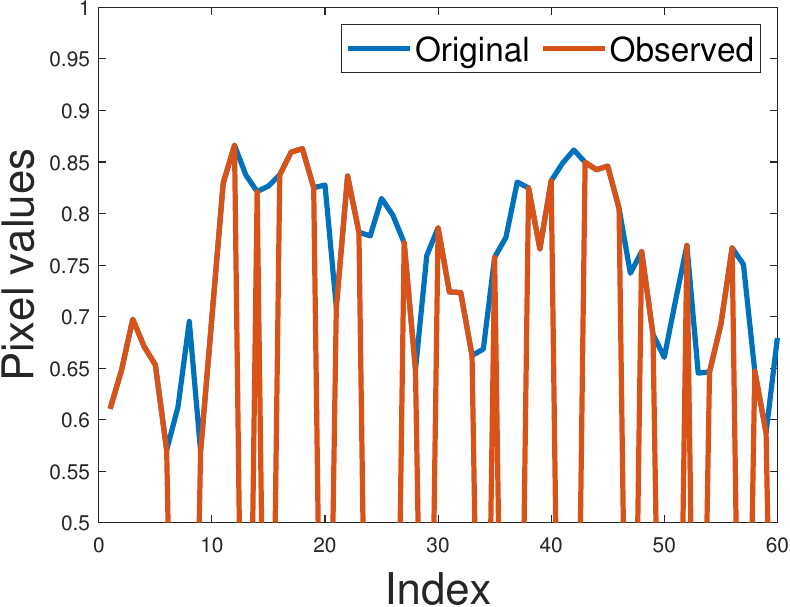}
			\caption*{Observed}
		\end{subfigure}
		\begin{subfigure}[b]{0.138\linewidth}
			\centering
			\includegraphics[width=\linewidth]{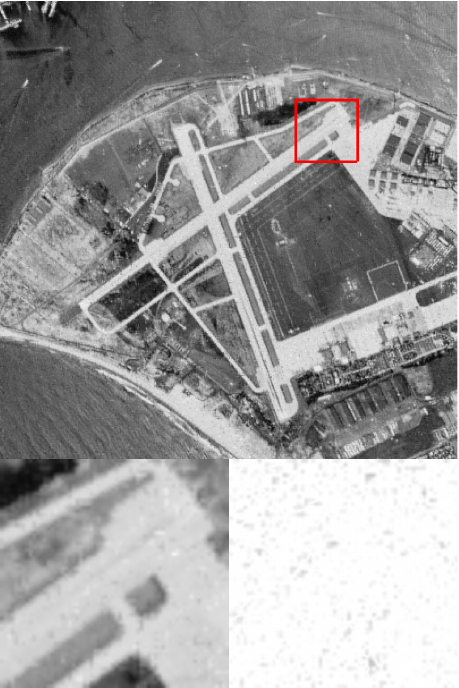}\vspace{0pt}
			\includegraphics[width=\linewidth]{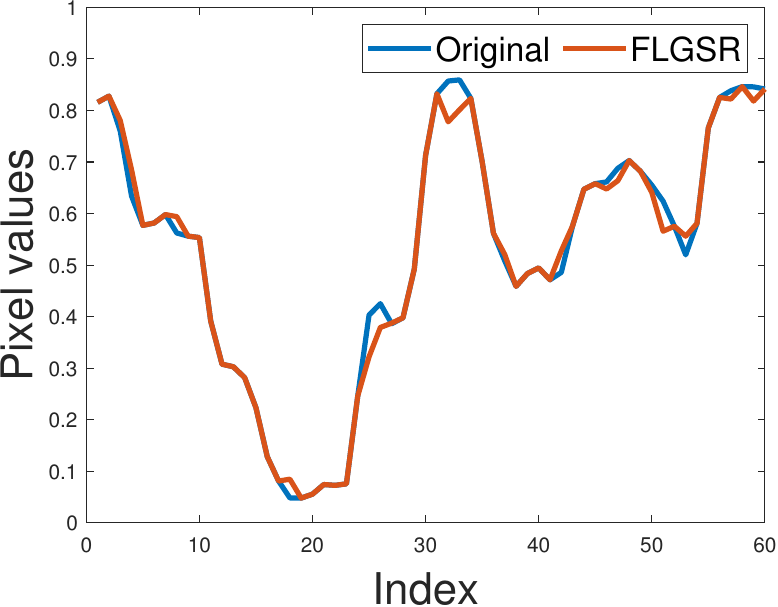}\vspace{0pt}
			\includegraphics[width=\linewidth]{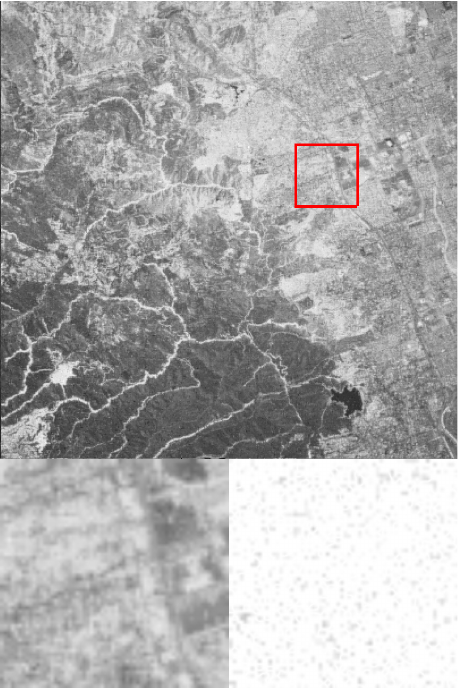}\vspace{0pt}
			\includegraphics[width=\linewidth]{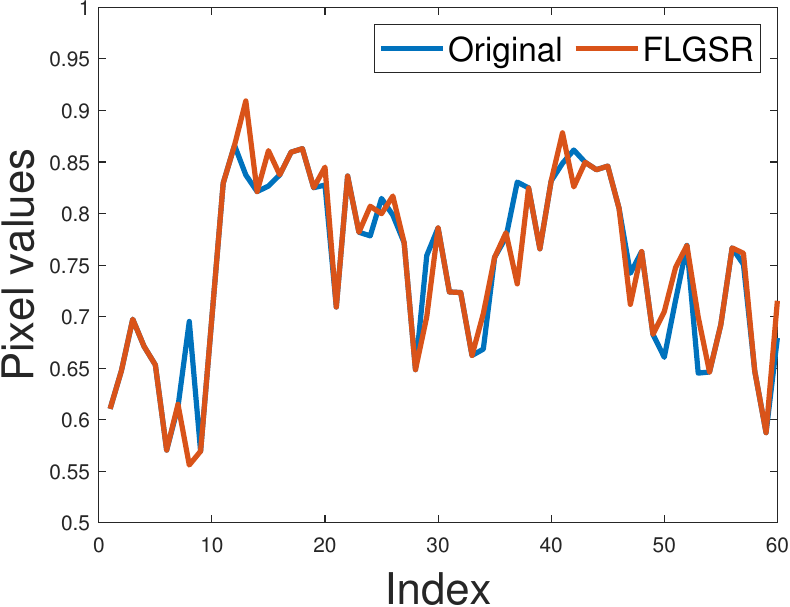}
			\caption*{FLGSR}
		\end{subfigure}
		\begin{subfigure}[b]{0.138\linewidth}
			\centering		
			\includegraphics[width=\linewidth]{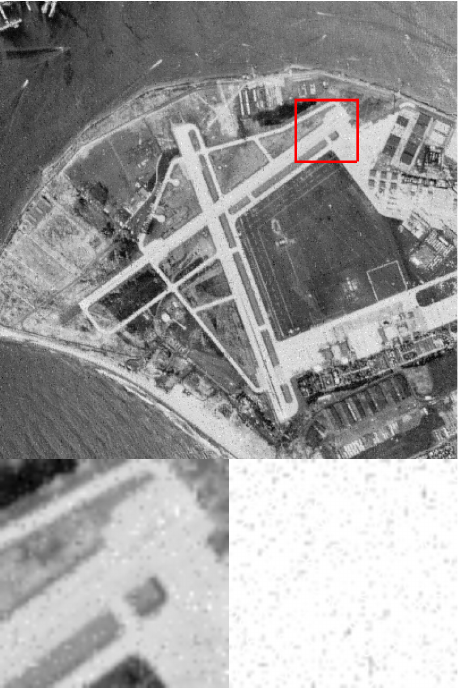}\vspace{0pt}
			\includegraphics[width=\linewidth]{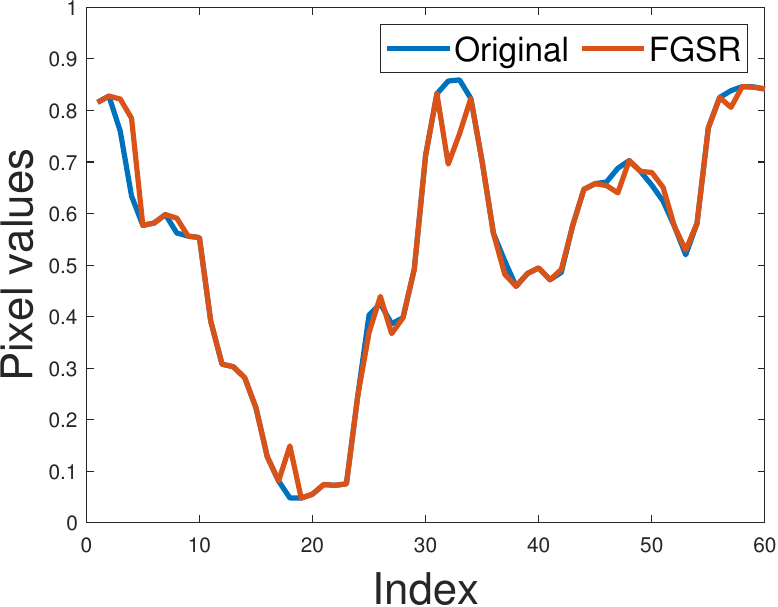}\vspace{0pt}
			\includegraphics[width=\linewidth]{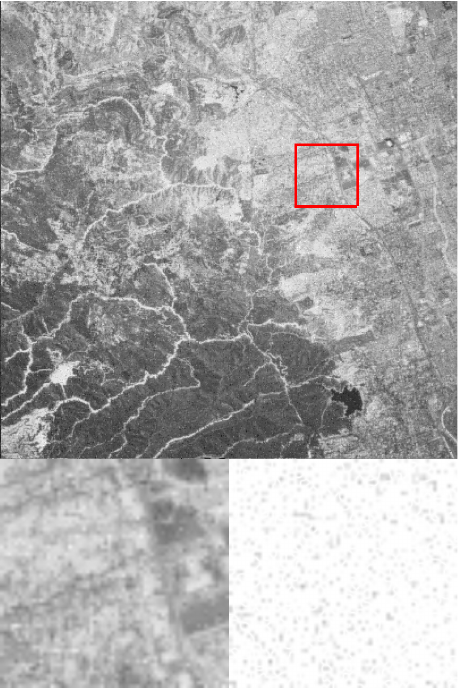}\vspace{0pt}
			\includegraphics[width=\linewidth]{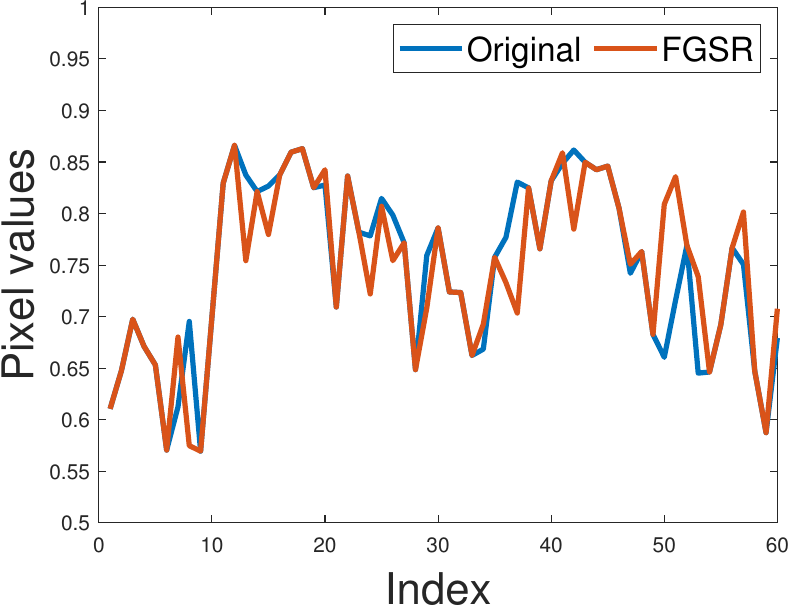}
			\caption*{FGSR}
		\end{subfigure}
		\begin{subfigure}[b]{0.138\linewidth}
			\centering
			\includegraphics[width=\linewidth]{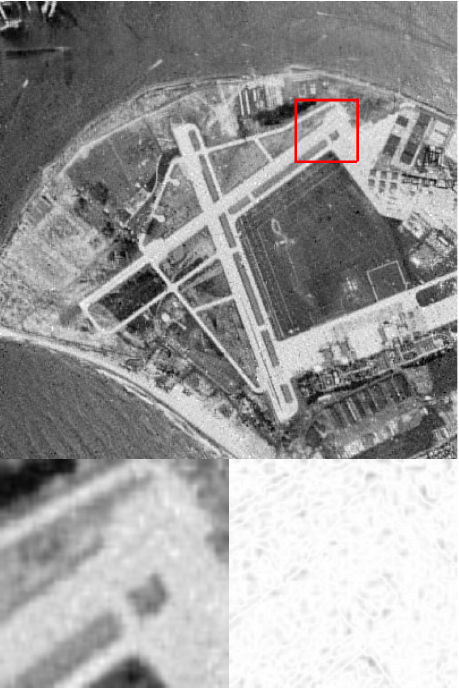}\vspace{0pt}
			\includegraphics[width=\linewidth]{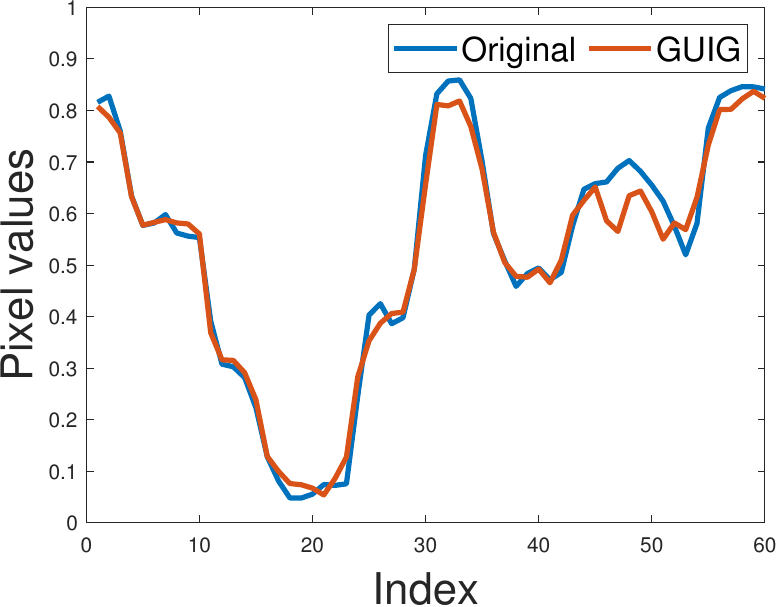}\vspace{0pt}
			\includegraphics[width=\linewidth]{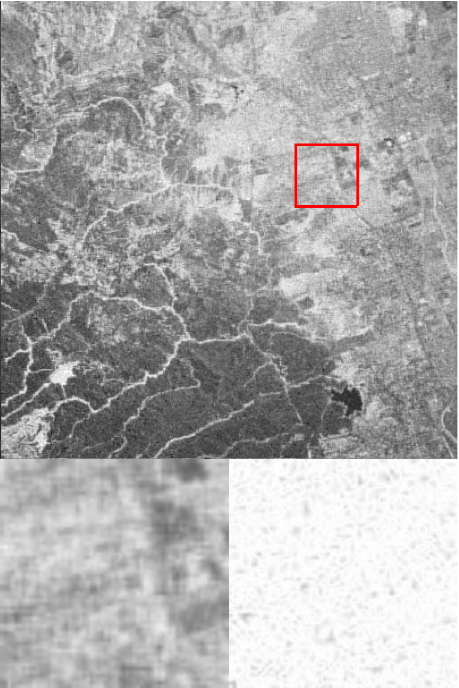}\vspace{0pt}
			\includegraphics[width=\linewidth]{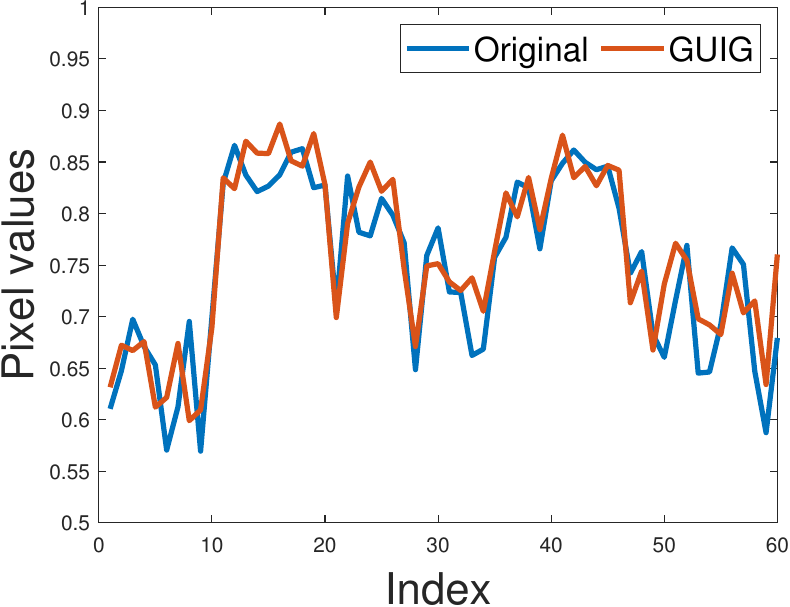}
			\caption*{GUIG}
		\end{subfigure}	
		\begin{subfigure}[b]{0.138\linewidth}
			\centering			
			\includegraphics[width=\linewidth]{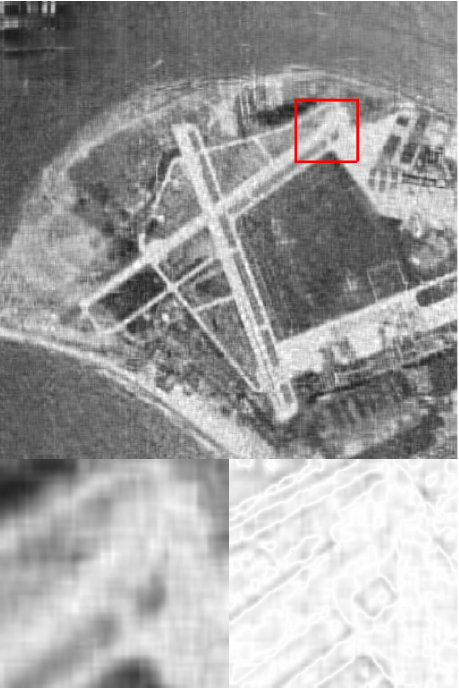}\vspace{0pt}
			\includegraphics[width=\linewidth]{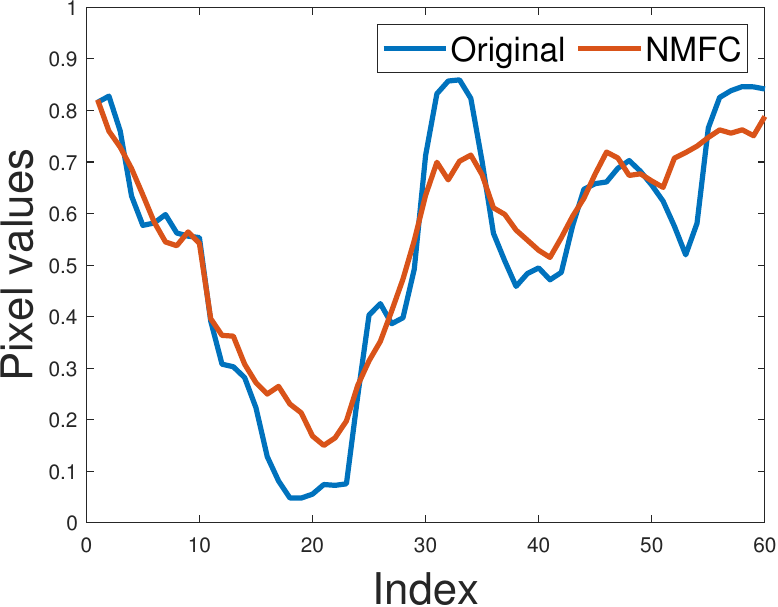}\vspace{0pt}
			\includegraphics[width=\linewidth]{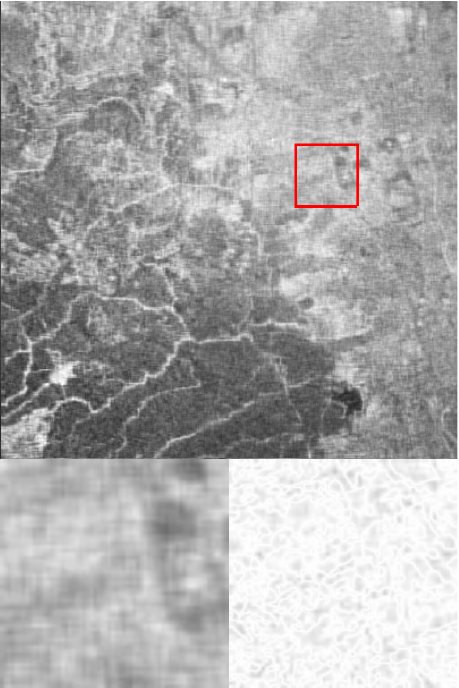}\vspace{0pt}
			\includegraphics[width=\linewidth]{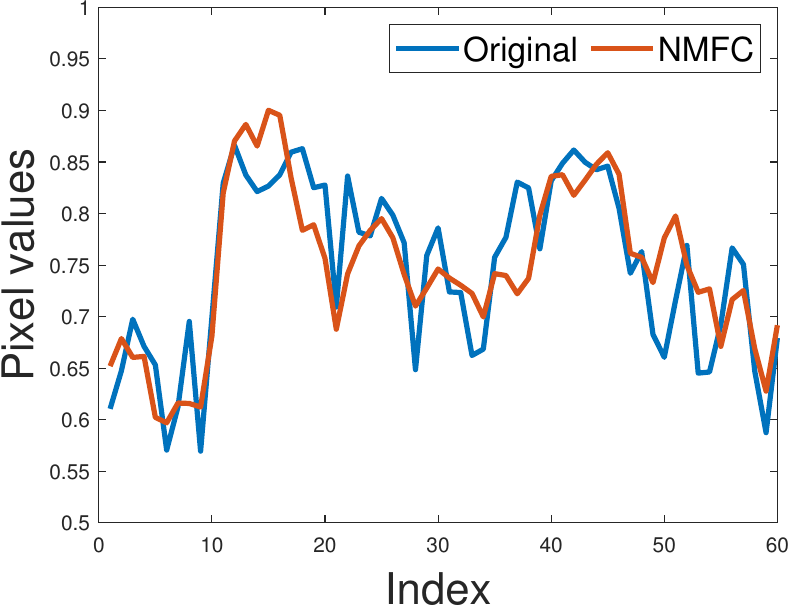}
			\caption*{NMFC}
		\end{subfigure}
		\begin{subfigure}[b]{0.138\linewidth}
			\centering			
			\includegraphics[width=\linewidth]{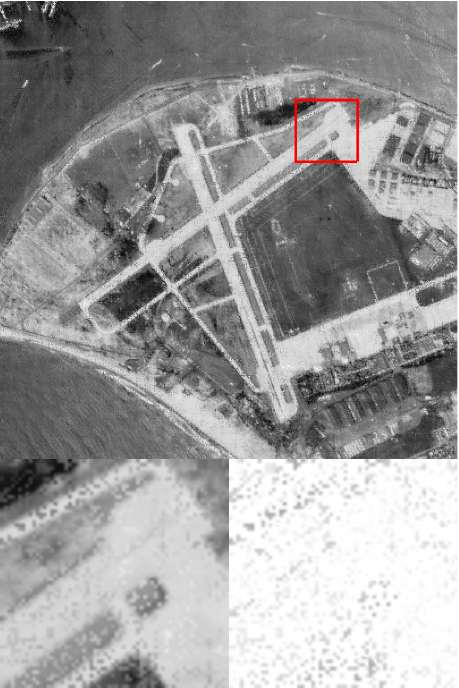}\vspace{0pt}
			\includegraphics[width=\linewidth]{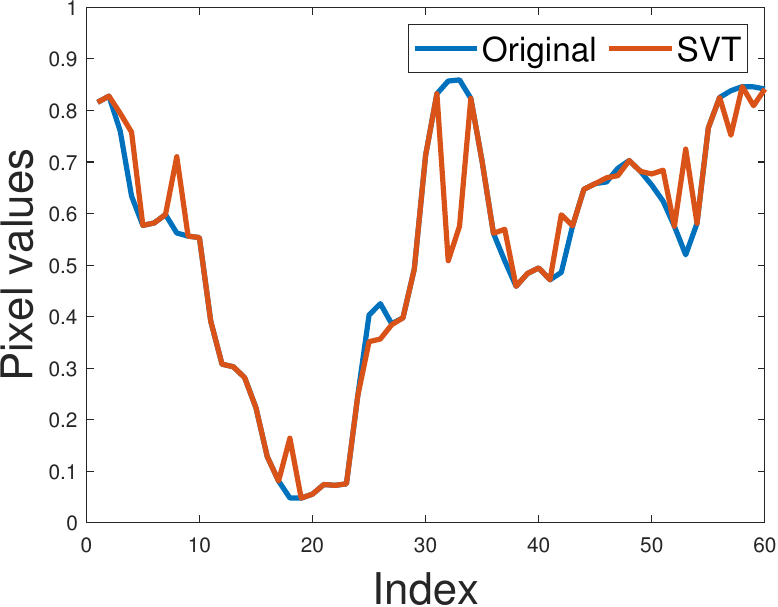}\vspace{0pt}
			\includegraphics[width=\linewidth]{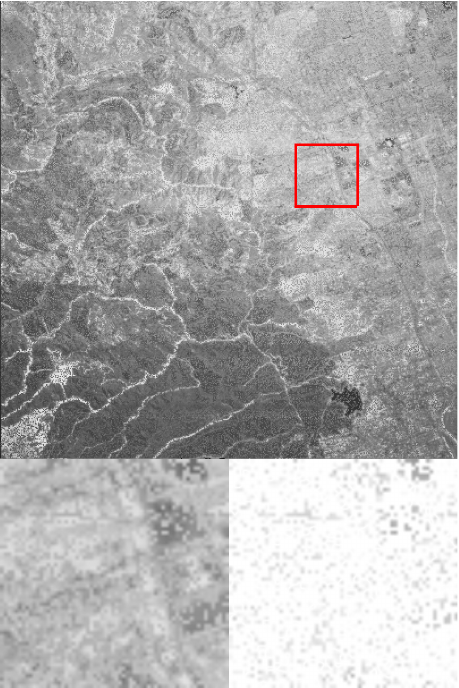}\vspace{0pt}
			\includegraphics[width=\linewidth]{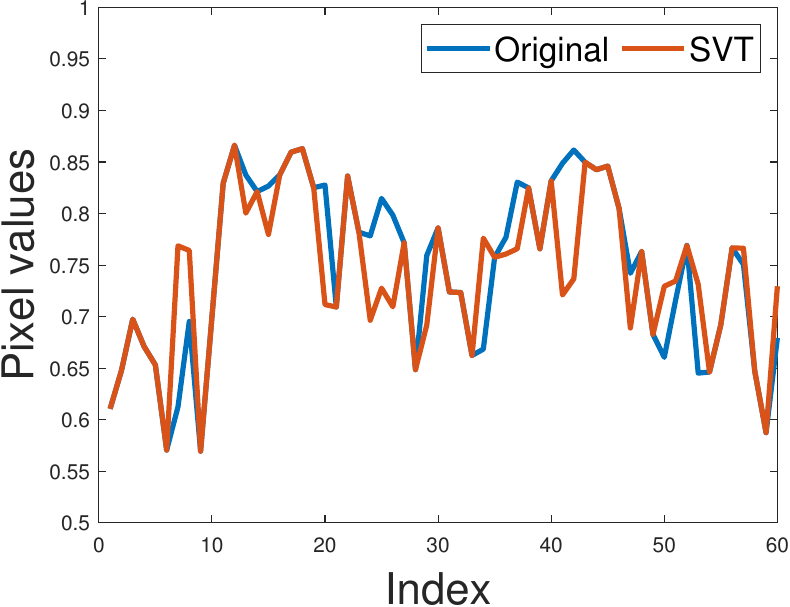}
			\caption*{SVT}
		\end{subfigure}						
	\end{subfigure}
	\vfill
	\caption{Examples of HAA image inpainting. From top to bottom are respectively corresponding to ``San Diego" and ``Woodland Hills". For better visualization, we show the zoom-in region and the corresponding partial residuals of the region. Under each image, we show enlargements of a demarcated patch and the corresponding error map (difference from the Original). Error maps with less color information indicate better restoration performance.}
	\label{fig:HAAimage}
\end{figure*}

\begin{table}[htbp]
	\centering
	\caption{HAA image inpainting performance comparison: PSNR, SSIM and running time}
	\begin{tabular}{ccccccc}
		\toprule
		Image & Index & FLGSR & FGSR  & GUIG  & NMFC  & SVT \\
		\midrule
		\multirow{3}{*}{San Diego} & PSNR  & \textbf{29.843 } & 28.182  & 27.113  & 22.539  & 25.133  \\
		& SSIM  & \textbf{0.866 } & 0.794  & 0.743  & 0.585  & 0.768  \\
		& TIME  & 7.578  & 10.651  & 17.139  & 0.358  & 4.356  \\
		\midrule
		\multirow{3}{*}{Woodland Hills} & PSNR  & \textbf{28.819 } & 26.873  & 26.143  & 24.330  & 23.287  \\
		& SSIM  & \textbf{0.857 } & 0.785  & 0.730  & 0.591  & 0.682  \\
		& TIME  & 6.521  & 10.159  & 11.581  & 0.387  & 2.555  \\
		\bottomrule
	\end{tabular}%
	\label{tab:HAAimage}%
\end{table}%

\section{Conclusions}
 In this paper, we have proposed a new group sparsity approach to the LRMR problem, which can recover low rank matrices from incomplete observations. Specifically, we have introduced the FLGSR, a novel regularizer that can group multiple columns of a matrix as a unit based on the data structure. By doing so, we have proved the equivalence between the matrix rank and the FLGSR, and shown that the LRMR problem with either of them has the same global minimizers. Furthermore, we have also established the equivalence between the relaxed and the penalty formulations of the LRMR problem with FLGSR. To optimize this model, we have devised an efficient algorithm to solve the LRMR problem with FLGSR, and analyzed its convergence properties. Finally, we have demonstrated the superiority of our method over state-of-the-art methods in terms of recovery accuracy, visual quality and computational efficiency on both grayscale images and high altitude aerial images.

\bigskip
\noindent
{\bf Acknowledgement}
Xinzhen Zhang was partly supported by the National Natural Science Foundation of China
(Grant No. 11871369). Minru Bai was partly supported by the National Natural Science Foundation of China (Grant No. 11971159, 12071399) and the Hunan Provincial Key Laboratory of Intelligent Information Processing and Applied Mathematics.

\bibliographystyle{plain}
\bibliography{MatrixGroup}

\begin{thebibliography}{10}

\bibitem{ABEV06}
Jacob~D. Abernethy, Francis~R. Bach, Theodoros Evgeniou, and Jean{-}Philippe
  Vert.
\newblock Low-rank matrix factorization with attributes.
\newblock {\em CoRR}, abs/cs/0611124, 2006.

\bibitem{AFSU07}
Yonatan Amit, Michael Fink, Nathan Srebro, and Shimon Ullman.
\newblock Uncovering shared structures in multiclass classification.
\newblock In {\em Proceedings of the 24th international conference on Machine
  learning}. {ACM} Press, 2007.

\bibitem{Bro12}
Andrew Browder.
\newblock {\em Mathematical Analysis}.
\newblock Undergraduate Texts in Mathematics. Springer New York, NY, 2012.

\bibitem{ClTCB15}
Ricardo Cabral, Fernando~De la~Torre, Joao~Paulo Costeira, and Alexandre
  Bernardino.
\newblock Matrix completion for weakly-supervised multi-label image
  classification.
\newblock {\em {IEEE} Transactions on Pattern Analysis and Machine
  Intelligence}, 37(1):121--135, jan 2015.

\bibitem{CCS10}
Jian-Feng Cai, Emmanuel~J. Cand{\`{e}}s, and Zuowei Shen.
\newblock A singular value thresholding algorithm for matrix completion.
\newblock {\em {SIAM} Journal on Optimization}, 20(4):1956--1982, jan 2010.

\bibitem{CR09}
Emmanuel~J. Cand{\`{e}}s and Benjamin Recht.
\newblock Exact matrix completion via convex optimization.
\newblock {\em Foundations of Computational Mathematics}, 9(6):717--772, apr
  2009.

\bibitem{CT10}
Emmanuel~J. Candes and Terence Tao.
\newblock The power of convex relaxation: Near-optimal matrix completion.
\newblock {\em {IEEE} Transactions on Information Theory}, 56(5):2053--2080,
  may 2010.

\bibitem{CGLY17}
Xiaojun Chen, Lei Guo, Zhaosong Lu, and Jane~J. Ye.
\newblock An augmented lagrangian method for non-lipschitz nonconvex
  programming.
\newblock {\em {SIAM} Journal on Numerical Analysis}, 55(1):168--193, jan 2017.

\bibitem{CLP16}
Xiaojun Chen, Zhaosong Lu, and Ting~Kei Pong.
\newblock Penalty methods for a class of non-lipschitz optimization problems.
\newblock {\em SIAM Journal on Optimization}, 26(3):1465--1492, January 2016.

\bibitem{DXG18}
Jing Dong, Zhichao Xue, Jian Guan, Zi-Fa Han, and Wenwu Wang.
\newblock Low rank matrix completion using truncated nuclear norm and sparse
  regularizer.
\newblock {\em Signal Processing: Image Communication}, 68:76--87, oct 2018.

\bibitem{EvdH12}
A.~Eriksson and A.~van~den Hengel.
\newblock Efficient computation of robust weighted low-rank matrix
  approximations using the l{\_}1 norm.
\newblock {\em {IEEE} Transactions on Pattern Analysis and Machine
  Intelligence}, 34(9):1681--1690, sep 2012.

\bibitem{FDCU19}
Jicong Fan, Lijun Ding, Yudong Chen, and Madeleine Udell.
\newblock Factor group-sparse regularization for efficient low-rank matrix
  recovery.
\newblock In {\em Proceedings of the 33rd International Conference on Neural
  Information Processing Systems}, NIPS'19, pages 5104--5114, Red Hook, NY,
  USA, 2019. Curran Associates Inc.

\bibitem{FHB04}
M.~Fazel, H.~Hindi, and S.~Boyd.
\newblock Rank minimization and applications in system theory.
\newblock In {\em Proceedings of the 2004 American Control Conference}. {IEEE},
  2004.

\bibitem{FHB01}
M.~Fazel, H.~Hindi, and S.P. Boyd.
\newblock A rank minimization heuristic with application to minimum order
  system approximation.
\newblock In {\em Proceedings of the 2001 American Control Conference.} {IEEE},
  2001.

\bibitem{JFWZ21}
Xixi Jia, Xiangchu Feng, Weiwei Wang, and Lei Zhang.
\newblock Generalized unitarily invariant gauge regularization for fast
  low-rank matrix recovery.
\newblock {\em {IEEE} Transactions on Neural Networks and Learning Systems},
  32(4):1627--1641, apr 2021.

\bibitem{KMO10}
Raghunandan~H. Keshavan, Andrea Montanari, and Sewoong Oh.
\newblock Matrix completion from a few entries.
\newblock {\em {IEEE} Transactions on Information Theory}, 56(6):2980--2998,
  jun 2010.

\bibitem{Kom}
N.~Komodakis.
\newblock Image completion using global optimization.
\newblock In {\em 2006 {IEEE} Computer Society Conference on Computer Vision
  and Pattern Recognition - Volume 1}. {IEEE}.

\bibitem{LL16}
Chul Lee and Edmund~Y. Lam.
\newblock Computationally efficient truncated nuclear norm minimization for
  high dynamic range imaging.
\newblock {\em {IEEE} Transactions on Image Processing}, 25(9):4145--4157, sep
  2016.

\bibitem{Lew95}
Adrian~S. Lewis.
\newblock The convex analysis of unitarily invariant matrix functions.
\newblock {\em Journal of Convex Analysis}, 2(1):173--183, 1995.

\bibitem{LLC23}
Wei Liu, Xin Liu, and Xiaojun Chen.
\newblock An inexact augmented lagrangian algorithm for training leaky {{ReLU}}
  neural network with group sparsity.
\newblock {\em Journal of Machine Learning Research}, 24(212):1--43, 2023.

\bibitem{LLM19}
Ya-Feng Liu, Xin Liu, and Shiqian Ma.
\newblock On the nonergodic convergence rate of an inexact augmented lagrangian
  framework for composite convex programming.
\newblock {\em Mathematics of Operations Research}, 44(2):632--650, May 2019.

\bibitem{LZ12}
Zhaosong Lu and Yong Zhang.
\newblock An augmented lagrangian approach for sparse principal component
  analysis.
\newblock {\em Mathematical Programming}, 135(1):149--193, apr 2012.

\bibitem{LL94}
Xiao-Dong Luo and Zhi-Quan Luo.
\newblock Extension of hoffman's error bound to polynomial systems.
\newblock {\em SIAM Journal on Optimization}, 4(2):383--392, May 1994.

\bibitem{LLTX15}
Yong Luo, Tongliang Liu, Dacheng Tao, and Chao Xu.
\newblock Multiview matrix completion for multilabel image classification.
\newblock {\em {IEEE} Transactions on Image Processing}, 24(8):2355--2368, aug
  2015.

\bibitem{MGC11}
Shiqian Ma, Donald Goldfarb, and Lifeng Chen.
\newblock Fixed point and bregman iterative methods for matrix rank
  minimization.
\newblock {\em Mathematical Programming}, 128(1-2):321--353, sep 2011.

\bibitem{MLH17}
Tian-Hui Ma, Yifei Lou, and Ting-Zhu Huang.
\newblock Truncated $l_{1-2}$ models for sparse recovery and rank minimization.
\newblock {\em {SIAM} Journal on Imaging Sciences}, 10(3):1346--1380, jan 2017.

\bibitem{MS12}
G.~Marjanovic and V.~Solo.
\newblock On $l_q$ optimization and matrix completion.
\newblock {\em {IEEE} Transactions on Signal Processing}, 60(11):5714--5724,
  nov 2012.

\bibitem{NWC12}
Feiping Nie, Hua Wang, Xiao Cai, Heng Huang, and Chris Ding.
\newblock Robust matrix completion via joint schatten p-norm and lp-norm
  minimization.
\newblock In {\em 2012 {IEEE} 12th International Conference on Data Mining}.
  {IEEE}, dec 2012.

\bibitem{PC21}
Lili Pan and Xiaojun Chen.
\newblock Group sparse optimization for images recovery using capped folded
  concave functions.
\newblock {\em {SIAM} Journal on Imaging Sciences}, 14(1):1--25, jan 2021.

\bibitem{RFP10}
Benjamin Recht, Maryam Fazel, and Pablo~A. Parrilo.
\newblock Guaranteed minimum-rank solutions of linear matrix equations via
  nuclear norm minimization.
\newblock {\em {SIAM} Review}, 52(3):471--501, jan 2010.

\bibitem{SWKT19}
Xinhua Su, Yilun Wang, Xuejing Kang, and Ran Tao.
\newblock Nonconvex truncated nuclear norm minimization based on adaptive
  bisection method.
\newblock {\em {IEEE} Transactions on Circuits and Systems for Video
  Technology}, 29(11):3159--3172, nov 2019.

\bibitem{WBSS04}
Z.~Wang, A.C. Bovik, H.R. Sheikh, and E.P. Simoncelli.
\newblock Image quality assessment: From error visibility to structural
  similarity.
\newblock {\em {IEEE} Transactions on Image Processing}, 13(4):600--612, apr
  2004.

\bibitem{WYZ12}
Zaiwen Wen, Wotao Yin, and Yin Zhang.
\newblock Solving a low-rank factorization model for matrix completion by a
  nonlinear successive over-relaxation algorithm.
\newblock {\em Mathematical Programming Computation}, 4(4):333--361, jul 2012.

\bibitem{XYWZ12}
Yangyang Xu, Wotao Yin, Zaiwen Wen, and Yin Zhang.
\newblock An alternating direction algorithm for matrix completion with
  nonnegative factors.
\newblock {\em Frontiers of Mathematics in China}, 7(2):365--384, apr 2012.

\bibitem{YKWL19}
Quanming Yao, James~T. Kwok, Taifeng Wang, and Tie-Yan Liu.
\newblock Large-scale low-rank matrix learning with nonconvex regularizers.
\newblock {\em {IEEE} Transactions on Pattern Analysis and Machine
  Intelligence}, 41(11):2628--2643, nov 2019.

\bibitem{YZ22}
Quan Yu and Xinzhen Zhang.
\newblock A smoothing proximal gradient algorithm for matrix rank minimization
  problem.
\newblock {\em Computational Optimization and Applications}, 81(2):519--538,
  January 2022.

\bibitem{ZZW22}
Xiaoqin Zhang, Jingjing Zheng, Di~Wang, Guiying Tang, Zhengyuan Zhou, and
  Zhouchen Lin.
\newblock Structured sparsity optimization with non-convex surrogates of
  $\ell_{2,0}$-norm: A unified algorithmic framework.
\newblock {\em {IEEE} Transactions on Pattern Analysis and Machine
  Intelligence}, pages 1--18, 2022.

\end{thebibliography}

\appendix

\section{Proof of Theorem \ref{Thm:rank}}\label{App-Thm:rank}
\begin{proof}
	On the one hand, from the definition of $ \left\|\cdot\right\|_{p,0} $, for any matrices $X:=\left[X_1,\ldots,X_s\right]\in\R^{m\times n}$ and $Y:=\left[Y_1,\ldots,Y_s\right]\in\R^{n\times n}$ that satisfy $C=XY^T$, we obtain that
	\begin{equation*}
		\begin{aligned}
			\rank\left(C\right)&\le \frac{1}{2}\left(\rank\left(X\right)+\rank\left(Y\right)\right)
			=\frac{1}{2}\rank\left(\left[X_1,\ldots,X_s\right]\right)+\frac{1}{2}\rank\left(\left[Y_1,\ldots,Y_s\right]\right)\\ 
			&\le \frac{1}{2}\sum_{i=1}^{s}\left(\rank\left(X_i\right)+\rank\left(Y_i\right)\right)
			\le \frac{1}{2}\sum_{i=1}^sn_i\left(\left\|X_i\right\|_p^0+\left\|Y_i\right\|_p^0\right)
			=\frac{1}{2}\left(\left\| X \right\|_{p,0}+\left\| Y \right\|_{p,0}\right). 
		\end{aligned}
	\end{equation*}
	On the other hand, it is clear that there exists two column full rank matrices $ {\bar X} \in \R^{m\times r} $ and $ {\bar Y} \in \R^{n\times r }$ such that $ C={\bar X}{\bar Y}^T $. By the given conditions, we know that there exist $ n_{i_1}, \ldots, n_{i_p} \in \left\lbrace n_1,\ldots, n_s \right\rbrace  $ such that $ \sum_{j=1}^{p} n_{i_j}=r $. Without loss of generality, we suppose $ n_{i_1}=n_1,\ldots,n_{i_p}=n_p $. Let $ \tilde{X}=[{\bar X},0] \in \R^{m\times n} $ and $ \tilde{Y}=[{\bar Y},0] \in \R^{n\times n} $, then $ C=\tilde{X}\tilde{Y}^T $ and $ \rank(\tilde{X})=\| \tilde{X} \|_{p,0}  $, $ \rank\left(\tilde{Y}\right)=\| \tilde{Y} \|_{p,0}  $.	
	
	From the above two aspects, we can conclude that
	$$
	\rank\left(C\right)=\mathop{\min} \limits_{C=XY^T} \frac{1}{2}\left(\left\| X \right\|_{p,0}+\left\| Y \right\|_{p,0}\right)=
	G_p^{\left|\cdot\right|^0}\left(C\right).
	$$
	The proof is completed.
\end{proof}

\section{Proof of Theorem \ref{Thm:nuclear}}\label{App-Thm:nuclear}

Before proof, we restate some definitions and lemmas here.

\begin{defi}\cite{Lew95}
A function $F: \mathbb{R}^{m \times n} \to[-\infty,+\infty]$ is called unitarily invariant if $F(U C V)=F(C)$ for any $C \in \mathbb{R}^{m \times n}$, where $U \in \mathbb{R}^{m \times m}, V\in \mathbb{R}^{n \times n}$ are arbitrary unitary matrices.
\end{defi}

\begin{defi}
	A function $f: \mathbb{R}^n \to[-\infty,+\infty]$ is called absolutely symmetric if
	$f(\bm{\hat c})=f(\bm{c})$ for any $ \bm{c} \in \mathbb{R}^n$, where $\bm{\hat c}$ is the vector with components $\left|c_i\right|$ arranged in a non-ascending order.
\end{defi}

\begin{lemma}\label{lem:ui}\cite[Proposition 2.2]{Lew95}
If the function $F: \mathbb{R}^{m \times n} \rightarrow$ $[-\infty,+\infty]$ is unitarily invariant, then there exists a absolutely symmetric function $f: \mathbb{R}^{\min\left\lbrace m,n \right\rbrace } \rightarrow[-\infty,+\infty]$ such that $F\left(X\right)=f \circ \sigma\left(X\right)$.
\end{lemma}

\begin{proof}
For any $C \in \mathbb{R}^{m \times n}$, one has
\begin{equation*}
	\begin{aligned}
		G_{2}^{\phi}\left(UCV\right)&=\mathop{\min} \limits_{UCV=XY^T} \frac{1}{2}\sum_{i=1}^{s}n_i\left(\phi\left(\left\|X_i\right\|_F\right)+\phi\left(\left\|Y_i\right\|_F\right)\right)\\
		&=\mathop{\min} \limits_{C=\left(U^TX\right)\left(VY\right)^T} \frac{1}{2}\sum_{i=1}^{s}n_i\left(\phi\left(\left\|U^TX_i\right\|_F\right)+\phi\left(\left\|VY_i\right\|_F\right)\right)\\
		&=\mathop{\min} \limits_{C={\bar X}{\bar Y}^T} \frac{1}{2}\sum_{i=1}^{s}n_i\left(\phi\left(\left\|{\bar X}_i\right\|_F\right)+\phi\left(\left\|{\bar Y}_i\right\|_F\right)\right)\\
		&=G_{2}^{\phi}\left(C\right)
	\end{aligned}
\end{equation*}
for arbitrary unitary matrices $U\in\R^{m\times m}$ and $V\in\R^{n\times n}$,
which encounters the expectation in Lemma \ref{lem:ui} and thus completes the proof.
\end{proof}

\end{document}